\documentclass{amsart}

\usepackage{amsmath}
\usepackage{amsfonts}
\usepackage{amssymb}		
\usepackage{amsthm}	
\usepackage{enumerate}
\usepackage[T1]{fontenc}
\usepackage{mathtools}
\usepackage[pdftex,pagebackref,colorlinks=true,pdfpagemode=none,urlcolor=black,linkcolor=black,citecolor=black,pdfstartview=FitH]{hyperref}
\usepackage{tikz, calc, mathdots}

\usetikzlibrary{arrows}

\numberwithin{table}{section}
\numberwithin{figure}{section}
\numberwithin{equation}{section}

\newtheorem{definition}{Definition}[section]
\newtheorem{lemma}[definition]{Lemma}
\newtheorem{question}[definition]{Question}
\newtheorem{theorem}[definition]{Theorem}
\newtheorem{corollary}[definition]{Corollary}
\newtheorem{proposition}[definition]{Proposition}

\newtheorem{fact}[definition]{Fact}
\newtheorem*{nquestion}{\nquestionname}
\newcommand{\nquestionname}{}
\newenvironment{questionn}[1]
 {\renewcommand{\nquestionname}{Question \ref{#1}}\nquestion}
 {\endnquestion}
\newtheorem*{nquestionb}{\nquestionbname}
\newcommand{\nquestionbname}{}
\newenvironment{defin}[1]
 {\renewcommand{\nquestionbname}{Definition \ref{#1}}\nquestionb}
 {\endnquestionb}
 \theoremstyle{definition}
\newtheorem{example}[definition]{Example}
\newtheorem{remark}[definition]{Remark}

\newcommand{\sbs}{\smallsetminus}
\DeclareMathOperator{\dv}{div}
\DeclareMathOperator{\mult}{mult}
\DeclareMathOperator{\id}{id}

\newcounter{row}
\newcounter{col}

\newcommand\setrow[6]{
  \setcounter{col}{1}
  \foreach \n in {#1, #2, #3, #4, #5, #6} {
    \edef\x{\value{col} - 0.5}
    \edef\y{6.5 - \value{row}}
    \node[anchor=center] at (\x, \y) {\n};
    \stepcounter{col}
  }
  \stepcounter{row}
}

\newcommand\setrowm[7]{
  \setcounter{col}{#7}
  \foreach \n in {#1, #2, #3, #4, #5, #6} {
    \edef\x{\value{col} - 0.5}
    \edef\y{6.5 - \value{row}}
    \node[anchor=center] at (\x, \y) {\n};
    \stepcounter{col}
  }
  \stepcounter{row}
}

\newcounter{rowf}
\newcounter{colf}

\newcommand\setrowf[8]{
  \setcounter{colf}{1}
  \foreach \n in {#1, #2, #3, #4, #5, #6, #7, #8} {
    \edef\x{\value{colf} - 0.5}
    \edef\y{8.5 - \value{rowf}}
    \node[anchor=center] at (\x, \y) {\n};
    \stepcounter{colf}
  }
  \stepcounter{rowf}
}
 
\begin{document}

\begin{abstract}
For which values of $n$ is it possible to color the positive integers using precisely $n$ colors in such a way that for any $a$, the numbers $a,2a,\dots,na$ all receive different 
colors? The third-named author posed the question around 2008-2009. Particular cases appeared in the Hungarian high school journal K\"oMaL in April 2010, and the general 
version appeared in May 2010 on MathOverflow, posted by D. P\'alv\"olgyi. The question remains open. We discuss the known partial results and investigate a series of related 
matters attempting to understand the structure of these \emph{$n$-satisfactory} colorings.

Specifically, we show that there is an $n$-satisfactory coloring whenever there is an abelian group operation $\oplus$ on the set $\{1,2,\dots,n\}$ that is compatible with 
multiplication in the sense that whenever $i$, $j$ and $ij$ are in $\{1,\dots,n\}$, then $ij=i\oplus j$. This includes in particular the cases where $n+1$ is prime, or $2n+1$ is prime, 
or $n=p^2-p$ for some prime $p$, or there is  a $k$ such that $q=nk+1$ is prime and $1^k,\dots,n^k$ are all distinct modulo $q$ (in which case we call $q$ a \emph{strong 
representative of order $n$}). The colorings obtained by this process we call multiplicative. We also show that nonmultiplicative colorings exist for some values of $n$. 

There is an $n$-satisfactory coloring of $\mathbb Z^+$ if and only if there is such a coloring of the set $K_n$ of $n$-smooth numbers. We identify all $n$-satisfactory colorings 
for $n\le 5$ and all multiplicative colorings for $n\le 8$, and show that there are as many nonmultiplicative colorings of $K_n$ as there are real numbers for $n=6$ and 8. We show 
that if $n$ admits a strong representative $q$ then it admits infinitely many and in fact the set of such $q$ has positive natural density in the set of all primes. 

We also show that the question of whether there is an $n$-satisfactory coloring is equivalent to a problem about tilings, and use this to give a geometric characterization of 
multiplicative colorings.
\end{abstract}

\author{Andr\'es Eduardo Caicedo}
\address{
Mathematical Reviews \\
  416 Fourth Street \\ 
  Ann Arbor, MI 48103-4820 \\ 
  USA
}
\email{aec@ams.org}
\urladdr{\href{http://www-personal.umich.edu/~caicedo/}{http://www-personal.umich.edu/~caicedo/}}

\author{Thomas A. C. Chartier}
\address{
777 W. Main St suite 900 \\
Boise, ID, 83702 \\
USA
}
\email{tommychartier@gmail.com}

\author{P\'eter P\'al Pach}
\address{
MTA-BME Lend\"ulet Arithmetic Combinatorics Research Group\\
Department of Computer Science and Information Theory\\ 
Budapest University of Technology and Economics\\ 
1117 Budapest, Magyar tud\'osok  k\"or\'utja 2.\\ 
Hungary
}
\email{ppp@cs.bme.hu}
\urladdr{\href{http://www.cs.bme.hu/~ppp/}{http://www.cs.bme.hu/~ppp/}}

\keywords{Smooth numbers, core, satisfactory coloring, multiplicative coloring, partial isomorphism}

\subjclass[2010]{Primary 11B75; Secondary 05B45, 20D60.}

\date{\today}

\title{Coloring the $n$-smooth numbers with $n$ colors}

\maketitle

\tableofcontents

\section{Introduction} \label{sec:introduction}

\subsection{A problem from K\"oMaL} \label{subsec:komal}

The following was posed by the third-named author as problem A.506 in the April 2010 issue of the Hungarian journal K\"oMaL (K\"oz\'episkolai Matematikai \'es Fizikai Lapok), a 
mathematics and physics journal primarily aimed at high school 
students\footnote{See \href{https://www.komal.hu/feladat?a=honap&h=201004&t=mat&l=en}{https://www.komal.hu/feladat?a=honap\&h=201004\&t=mat\&l=en}.}:
\begin{quote}
Prove that for every prime $p$, there exists a colouring of the positive integers with $p-1$ colours such that the colours of the numbers $\{a,2a,3a,\dots,(p-1)a\}$ are pairwise 
different for every positive integer $a$.
\end{quote}

We say that a coloring as required is \emph{$(p-1)$-satisfactory}. To get a feel for the problem, consider for instance the case $p=5$. Suppose $c$ is a 4-satisfactory coloring. 
In particular, $1,2,3,4$ have different colors. Note that 6 must have the same color as 1, since $c(6)\ne c(2),c(4)$ because $2,4,6,8$ all have different colors and $c(6)\ne c(3)$ 
since $3,6,9,12$ all have different colors. It follows that $c(8)=c(3)$. Also, since $c(12)\ne c(3),c(4),c(6)$, we must have $c(12)=c(2)$. It follows that $c(9)=c(4)$. Similar 
reasoning allows us to determine the color of many more numbers; the following table shows some of these findings. Here, the coloring is represented by means of 4 rows of 
integers, with each row representing one of the colors.

\begin{center}
 \begin{tabular}{ccccccc}
  1&6&16&36&81&$\cdots$\\
   \hline
   2&12&27&32&72&$\cdots$\\
   \hline
   3&8&18&48&108&$\cdots$\\
   \hline
   4&9&24&54&64&$\cdots$
  \end{tabular}
\end{center}

Nothing so far uses that 5 is a prime number, but the relevance of this fact comes into play once we note that the numbers in the $i^{\mathrm{th}}$ row are all congruent to $i$ modulo 
5, for $i=1,\dots,4$. This suggests how to define a 4-satisfactory coloring compatible with our observations. Indeed, as long as $a$ is not a multiple of 5, we can assign to $a$ the 
color $(a\bmod 5)$ and readily observe that if $1\le i<j\le 4$, then $(ai\bmod 5)\ne(aj\bmod 5)$. We are not quite done yet, as we still need to deal with the multiples of 5. For this, 
we can begin by noting that $5,10,15,20$ have different colors and impose some restrictions as above. For instance, $c(30)=c(5)$, $c(40)=c(15)$, $c(60)=c(10)$, $c(45)=c(20)$, 
etc., as illustrated in the table below.

\begin{center}
 \begin{tabular}{ccccccc}
  5&30&80&$\dots$\\
   \hline
   10&60&135&$\dots$\\
   \hline
   15&40&90&$\dots$\\
   \hline
   20&45&120&$\dots$
  \end{tabular}
\end{center}

The reader should promptly realize that this is the same table as before, with each entry multiplied by 5. This suggests that we can define the color of a positive integer $n$ by 
considering its prime factorization and ignoring powers of 5: letting $n=5^ab$ where $a\ge0$ and $5\nmid b$, we can assign to $n$ the color $c(n)=(b\bmod 5)$. It is 
straightforward to verify that this is indeed a 4-satisfactory coloring, and we are done in this case. 

The argument suggests an obvious generalization from which the K\"oMaL problem follows: 

\begin{theorem} \label{thm:caseprime}
If $p$ is prime, then there is a $(p-1)$-satisfactory coloring.
\end{theorem}

\begin{proof}
Define a coloring $c$ by writing $n=p^ab$ where $a\ge0$ and $p\nmid b$, and letting $c(n)=(b\bmod p)$, so that $c$ uses $p-1$ colors and if $1\le i<j<p$ and $n$ is as 
indicated, then $c(in)=(ib\bmod p)\ne(jb\bmod p)=c(jn)$.
\end{proof}

Although the solution just described makes essential use of the fact that $p$ is prime, it is natural to wonder whether such colorings are possible without this restriction. It is this 
version of the problem that we discuss in this paper. 

\subsection{The general question} \label{subsec:mathoverflow}

In May 29, 2010, D\"om\"ot\"or P\'alv\"olgyi posted on MathOverflow precisely the version just indicated.

\begin{question} \label{q:problem}
Given any positive integer $n$, is there a coloring of the positive integers using $n$ colors such that for any positive integer $a$, the numbers $a,2a,\dots,na$ all have 
different colors?\footnote{See \href{https://mathoverflow.net/q/26358/}{https://mathoverflow.net/q/26358/}}
\end{question}

It was through P\'alv\"olgyi's post that the first-named author became acquainted with the problem. He suggested it to the second-named author, and their partial results became 
the main content of the latter's master's thesis\footnote{See \href{http://scholarworks.boisestate.edu/td/231/}{http://scholarworks.boisestate.edu/td/231/}}.

Question \ref{q:problem} was originally formulated by the third-named author around 2008--2009, motivated by a question of G\"unter Pilz, see \S\,\ref{subs:pilz}. After working on 
it for a while, he posed several related questions in K\"oMaL. For instance, besides problem A.506, he also posed problem B.4265 in the April 2010 
issue,\footnote{See \href{https://www.komal.hu/feladat?a=honap&h=201004&t=mat&l=en}{https://www.komal.hu/feladat?a=honap\&h=201004\&t=mat\&l=en}.} asking about 
the case $n=7$. P\'alv\"olgyi first saw problem A.506 and became interested in the general version. He contacted the editor of K\"oMaL in charge of the ``A problems'' and asked
whether they knew the answer for general $n$. It was not until years later that P\'alv\"olgyi found out that the K\"oMaL questions and the general version of the problem were
originally posed by the third-named author.

Although the general problem remains open, there are enough partial results that we feel it is appropriate to publish this paper now, to further expose the mathematical community 
at large to question \ref{q:problem}, and to indicate the current state of affairs and the many additional questions that come out of this exploration. Question \ref{q:problem} has 
connections with number theory and group theory as well as a clearly combinatorial core. Some of the ideas we describe benefit from this interaction.

Several results we present are due to others, either from previous research on related topics or through suggestions posted on MathOverflow. We make every attempt to give 
credit as appropriate. 

As suggested above, we call \emph{$n$-satisfactory} a coloring as in the statement of question \ref{q:problem}. The analysis of the case $n=4$ in \S\,\ref{subsec:komal} reveals 
that we can in general restrict our attention to seeking $n$-satisfactory colorings of the set of \emph{$n$-smooth} numbers, that is, the set $K_n$ of positive integers whose 
prime factorization only includes primes less than or equal to $n$. We call this set the \emph{$n$-core}, see definition \ref{def:ncore}, and elaborate on this issue in section 
\ref{sec:preliminaries}; briefly, if there is an $n$-satisfactory coloring of $K_n$, we can color all positive integers by assigning to the number $km$, where $k\in K_n$ and 
$\gcd(m,n!)=1$ the color of $k$, and one can quickly check that this is an $n$-satisfactory coloring of $\mathbb Z^+$.

We note that, given $n$, even if question \ref{q:problem} has a negative answer for $n$, strictly fewer than $2n$ colors suffice to ensure that for any $a\in K_n$ all numbers 
$ia$, $1\le i\le n$, receive different colors: indeed, letting $p$ be the smallest prime larger than $n$, we can color $K_n$ with $p$ colors as in theorem \ref{thm:caseprime}, by 
assigning to $m\in K_n$ the color $(m\bmod p)$. If question \ref{q:problem} turns out to have a negative answer, it seems worth studying the following natural variant:

\begin{question} \label{q:prime}
Assuming that question \ref{q:problem} has a negative answer for $n$, can we find a better bound than the smallest prime larger than $n$ on the number of colors required to 
ensure a positive answer?
\end{question}

We close this introduction by discussing an application and the original motivation for question \ref{q:problem}.

\subsection{The Balasubramanian--Soundararajan theorem} \label{subsec:graham}

In 1970, Ronald Graham \cite{Graham70} conjectured the following:
\begin{quote}
If $n\ge 1$, and $0<a_1 <a_2 <\cdots<a_n$ are integers, then 
 $$ \max_{i,j}\frac{a_i}{\gcd(a_i,a_j)}\ge n. $$
\end{quote}

Graham's conjecture was finally verified in 1996 by Balasubramanian and Sound\-a\-ra\-ra\-jan via careful analytic estimates of average values of number-theoretic functions 
associated with the distribution of primes, see \cite{BalasubramanianSoundararajan96}. Assuming the existence of satisfactory colorings, we obtain a significantly simpler proof.

\begin{theorem} \label{thm:graham}
If there is an $(m-1)$-satisfactory coloring, then Graham's conjecture holds for $n=m$. 
\end{theorem}

\begin{proof}
Argue by contradiction. Accordingly, suppose that there are $(m-1)$-sat\-is\-fac\-to\-ry colorings and that $0 < b_1 < \cdots < b_m$ are integers such that 
 $$ \max_{i,j} b_i/\gcd(b_i,b_j)<m. $$ 
Suppose $i\ne j$ and let $M =\gcd(b_i,b_j)$. Let $a_i=b_i/M$ and $a_j=b_j/M$, so $a_i,a_j$ are both less than $m$, and $a_i\ne a_j$. Since $b_i =a_iM$ and $b_j=a_jM$, in 
any $(m-1)$-satisfactory coloring of $\mathbb Z^+$ we must have that $b_i$ is colored differently from $b_j$. This is impossible, since it would mean the coloring uses at least 
$m$ colors.
\end{proof}

The relationship highlighted in theorem \ref{thm:graham} between our question \ref{q:problem} and the Balasubramanian--Soundararajan theorem admits a nice graph-theoretic 
interpretation, that we now proceed to discuss. This connection was first mentioned by P\'eter Csikv\'ari to the third-named author, and was also noticed independently by Fedor 
Petrov in MathOverflow\footnote{See \href{https://mathoverflow.net/q/26358/}{https://mathoverflow.net/q/26358/}} and by Bosek, D\k ebski, Grytczuk, Sok\'o\l, 
\'Sleszy\'nska-Nowak and \.Zelazny, who also arrived independently of us at some of the observations below in their recent paper \cite{Boseketal18} (particularly, see their \S\,4).

Given a graph $G$, write $\chi(G)$ for its \emph{chromatic number}, that is, the least cardinal $\kappa$ such that the set of vertices of $G$ can be colored with $\kappa$ 
colors in such a way that adjacent vertices receive different colors. Note that if $G$ admits a clique (complete subgraph) on $r$ vertices, then $\chi(G)\ge r$, and 
therefore $\chi(G)\ge\omega(G)$, where $\omega(G)$ is the \emph{clique number} of $G$, that is, the supremum of the cardinalities of the cliques of $G$.

Given $n$, consider now the graph $\mathcal G_n$ with the positive integers as vertices and where any two $i\ne j$ in $\mathbb Z^+$ are connected if and only if 
$\max(i,j)/\gcd(i,j)\le n$ (that is, if and only if $i,j\in\{a,2a,\dots,na\}$ for some $a$); in \cite{Boseketal18}, this graph is denoted $B_n$. Note that $\mathcal G_n$ has many 
cliques of size $n$, namely each set $\{a,2a,\dots,na\}$ (and possibly others), so that $\omega(\mathcal G_n)\ge n$; see figure \ref{fig:g4}. 

\begin{figure}[ht] 
\centering
\begin{tikzpicture}[scale=.8,%
  every node/.style={draw, text=black,circle,minimum size=1pt},node distance=.8cm]
  \node[color=red]  (one) at (2,0) {\scriptsize$1_1$};
  \node[color=blue]  (two)  at (2*0.707,2*0.707) {\scriptsize$2_2$};
  \node[color=black]  (three) at (0,2) {\scriptsize$3_3$};
  \node[color=purple]  (four) at (2*-0.707,2*0.707) {\scriptsize$4_4$};
  \node[color=red]  (six)  at (-2,0) {\scriptsize$6_1$};
  \node[color=black]  (eight) at (2*-0.707,2*-0.707) {\scriptsize$8_3$};
  \node[color=purple]  (nine) at (0,-2) {\scriptsize$9_4$};
  \node[color=blue]  (twelve) at (2*0.707,2*-0.707) {\scriptsize$12_2$};
  \draw [violet] (one) -- (two) -- (three) -- (four) -- (one) -- (three) -- (six) -- (four) -- (two) -- (eight) -- (six) -- (nine) -- (twelve) -- (six) -- (two);
  \draw [violet] (nine) -- (three) -- (twelve) -- (eight) -- (four) -- (twelve);
\end{tikzpicture}
\caption{A portion of $\mathcal G_4$. Subindices indicate a coloring witnessing that $\chi(\mathcal G_4)=4$. Note the clique $\{2,3,4,6\}$. }
\label{fig:g4}
\end{figure}
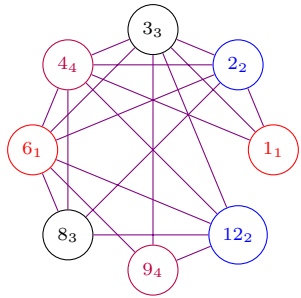

By definition, an $n$-satisfactory coloring $c$ of $\mathbb Z^+$ is a coloring of $\mathcal G_n$ with $n$ colors, that is, the existence of such a map $c$ is precisely the 
claim that $\chi(\mathcal G_n)$ is (at most, and therefore equal to) $n$. By the remark above, this implies that $\omega(\mathcal G_n)$ is (at most, and therefore equal to) 
$n$, but this is precisely Graham's conjecture for $n+1$, and we have reproved theorem \ref{thm:graham}. 

\subsection{Pilz's conjecture} \label{subs:pilz}

Recall that the symmetric difference $C\triangle D$ of two sets $C,D$, is the set of elements that belong to exactly one of $C,D$, that is, 
 $$ C\triangle D=(C\cup D)\sbs (C\cap D)=(C\sbs D)\cup (D\sbs C). $$ 
Note that $\triangle$ is associative, and so, given sets $C_1,\dots, C_m$, their symmetric difference $\bigtriangleup_{l=1}^m C_l$ is simply the set of elements that 
belong to precisely an odd number of sets $C_l$. 

If $X\subseteq \mathbb Z^+$ and $k\in\mathbb R$, we denote by $k\cdot X$ the \emph{dilation} of $X$ by a factor of $k$: 
 $$ k \cdot X = \{kx : x \in X\}. $$

Pilz's conjecture, the original motivation for question \ref{q:problem}, first appeared in 1992 \cite{Pilz92}. For our purposes, it is convenient to phrase it as follows:
\begin{quote}
If $n\ge 1$ and $A$ is a finite set of positive integers, then the size of the symmetric difference of the sets $A, 2\cdot A, \dots, n\cdot A$ is at least $n$.
\end{quote}
For $m$ a positive integer, it will be convenient in what follows to write $[m]$ for the set $\{1,2,\dots,m\}$. The particular case of Pilz's conjecture where $A=[k]$ for some 
$k\in\mathbb Z^+$ was eventually established independently during the academic year 2008--2009 by P.-Y. Huang, W.-F. Ke and G. F. Pilz \cite{HuangKePilz10} and by Pach 
and C. Szab\'o \cite{PachSzabo11}. The general case remains open.

\begin{theorem}
If there is an $n$-satisfactory coloring, then Pilz's conjecture holds for $n$ under the further assumption that $|A|$ is odd.
\end{theorem}

\begin{proof}
Say that $|A|=k$, let $A=\{a_j: j\in[k]\}$ and, for $j\in[k]$, set 
 $$ B_j=\{i\cdot a_j: i\in[n]\}. $$
Note first that the symmetric difference of the sets $i\cdot A$, $i\in[n]$, equals the symmetric difference of the sets $B_j$, $j\in[k]$. The point is that, denoting by 
$\chi_C(\cdot)$ the characteristic function of a set $C$, we have for any element $x$ that 
 $$ \chi_{\bigtriangleup_{i=1}^n i\cdot A}(x)=\biggl(\sum_{i=1}^n\chi_{i\cdot A}(x)\biggr)\bmod 2=|\{i\in[n]: \exists j\in[k]\,(x=i\cdot a_j)\}|\bmod 2 $$ 
and 
 $$ \chi_{\bigtriangleup_{j=1}^k B_j}(x)=\biggl(\sum_{j=1}^k\chi_{B_j}(x)\biggr)\bmod 2=|\{j\in[k]:\exists i\in[n]\,(x=i\cdot a_j)\}|\bmod 2, $$
and both expressions coincide since both equal 
 $$ |\{(i,j)\in[n]\times[k]: x=i\cdot a_j\}|\bmod 2. $$

For any $n$-satisfactory coloring, in every $B_j$ each color appears exactly once. That is, the sets $B_1,B_2,\dots,B_k$ contain $k$ numbers from each color class (counted 
with multiplicity). If $k$ is odd, then this means that their symmetric difference must contain an odd number of elements from each color class (and therefore at least one). But 
there are $n$ colors. 
\end{proof}

It was precisely this observation that motivated the third-named author to formulate question \ref{q:problem}. Sadly, when $|A|$ is even, the trick above does not apply and we do 
not see a way of establishing the conjecture in full generality. 

Pilz formulated in \cite{Pilz92} both the general case and the special case of his conjecture where $A=[k]$ for some $k$ (the latter is sometimes called the 1-2-3 conjecture). The 
paper \cite{PachSzabo11} is based on the third-named author's master's thesis\footnote{See \href{http://web.cs.elte.hu/blobs/diplomamunkak/mat/2009/pach_peter_pal.pdf}%
{http://web.cs.elte.hu/blobs/diplomamunkak/mat/2009/pach$\underline\ $peter$\underline\ $pal.pdf}}. 

For $A=[k]$, if there is a $k$-satisfactory coloring and $n$ is odd, then the same argument gives us that the size of the symmetric difference $\bigtriangleup_{i=1}^n i\cdot A$ is 
at least $n$. 

\subsection{Organization of this paper} \label{subsec:organization}

We begin section \ref{sec:preliminaries} with some preliminaries, emphasizing the role of what we call the $n$-core $K_n$. We also include some easy observations on the 
structure of the set $C_{K_n}$ of $n$-satisfactory colorings of the $n$-core. We reformulate question \ref{q:problem} as a problem about tilings, and close the section by giving 
an explicit description of all $n$-satisfactory colorings for $n\le 5$. In section~\ref{sec:pk+1} we explore an idea that directly generalizes the approach used to solve the 
original K\"oMaL problem (the case where $n+1$ is a prime number). This suggestion leads to several interesting number-theoretic questions that we also discuss. In section 
\ref{sec:multiplicative} we discuss a group-theoretic approach to question \ref{q:problem} that encompasses the suggestion from section \ref{sec:pk+1}. The colorings to which 
this suggestion applies we call multiplicative; we also characterize these colorings geometrically through the notion of translation invariance, and close the section by listing all 
multiplicative $n$-satisfactory colorings for $n\le 8$. We conclude in section \ref{sec:groupless} by indicating cases where the approach from section \ref{sec:multiplicative} fails. 
This includes a brief review of prior work by Forcade and Pollington \cite{ForcadePollington90}, and also a discussion of nonmultiplicative 6- and 8-satisfactory colorings. The final 
section \ref{sec:questions} lists several of the remaining open problems. We proceed to list some additional details.

The theorem below summarizes the values of $n$ for which a positive answer to question \ref{q:problem} is known, see also \S\,\ref{subsec:table}.

\begin{theorem} \label{thm:positive}
Question \ref{q:problem} has a positive answer for $n$, that is, there is an $n$-satisfactory coloring, in any of the following cases:
\begin{enumerate}
\item
$n+1$ is prime.
\item
$2n+1$ is prime.
\item
More generally, there is a strong representative of order $n$, i.e., a prime $p$ of the form $nk+1$ for some $k$ such that $1^k,\dots,n^k$ are pairwise distinct modulo $p$.
\item
Yet more generally, $n$ admits a partial $G$-isomorphism for some abelian group $(G,\oplus)$ of order $n$, i.e., there is a bijection $h\!:[n]\to G$ such that whenever 
$a,b\in[n]$, if $ab\in[n]$, then $h(ab)=h(a)\oplus h(b)$. In particular:
\item
For all $n<195$, and 
\item 
For all $n$ of the form $p^2-p$ for some prime $p$.
\end{enumerate}
\end{theorem}

\begin{proof}
(1) This is theorem \ref{thm:caseprime}.

\vspace{1mm}
(2) See corollary \ref{cor:2n+1}.

\vspace{1mm}
(3) See theorem \ref{thm:representative}.

\vspace{1mm}
(4) See theorem \ref{thm:partialisomorphism}.

\vspace{1mm}
(5) See theorem \ref{thm:195} or \cite{ForcadePollington90}.

\vspace{1mm}
(6) See theorem \ref{thm:p2-p}.

\vspace{1mm}
That (3) generalizes (1) and (2) is explained in \S\,\ref{subsec:strong}. That (4) generalizes (3) is explained in \S\,\ref{subs:multiplicative}. That (5) follows from (4) is 
explained in \S\,\ref{subsec:groupless}. That (6) follows from (4) is shown in the proof of theorem \ref{thm:p2-p}. 
\end{proof}

We feel that although question \ref{q:problem} was the guiding influence for much of the research reported in this paper, the topic will not be concluded even when the question is 
settled completely. Indeed, much of the paper is devoted to exploring how many $n$-satisfactory colorings there are for a given $n$ and, more generally, to studying the structure 
of these colorings. 

\begin{theorem} \label{thm:howmany}
If there is an $n$-satisfactory coloring, then there are as many such colorings as there are real numbers. Any $n$-satisfactory coloring of $\mathbb Z^+$ is determined by a sequence 
of $n$-satisfactory colorings of the core $K_n$ and a sequence of permutations of $[n]$. 

On the other hand, there are values of $n$ for which there are only finitely many $n$-satisfactory colorings of the core, and others for which there are again as many such colorings 
as there are real numbers.
\end{theorem}

\begin{proof}
See item (4) of proposition \ref{prop:appropriate}, where the precise way in which $n$-satisfactory colorings correspond to a sequence of colorings of the core and a sequence of 
permutations is described. From this, the number of $n$-satisfactory colorings is easily obtained, see corollary \ref{cor:many}. 

In \S\,\ref{subs:nle5} we show that for $n\le 5$ there are only finitely many $n$-satisfactory colorings of the core. In \S\,\ref{subs:n6} and \S\,\ref{subs:n8} we show that there are as
many $n$-satisfactory colorings of the core as there are real numbers for $n=6,8$.
\end{proof}

This result shows that the study of the structure of $n$-satisfactory colorings should really focus on the core, and we orient our efforts accordingly. Particular attention is paid to 
colorings with special structure. 

\begin{defin}{def:multiplicative}
An $n$-satisfactory coloring $c$ of $\mathbb Z^+$ or $K_n$ is multiplicative if and only if there is an abelian group structure $([n],\oplus)$ such 
that $c(ab)=c(a)\oplus c(b)$ for all $a,b$. 
\end{defin}

Note that for any given $n$ there are only finitely many such group structures. Nevertheless, a version of theorem \ref{thm:howmany} 
holds for this case as well.

\begin{theorem}
There is a multiplicative coloring of $K_n$ if and only if there are as many such colorings of $\mathbb Z^+$ as there are real numbers.

For any $n$, there are only finitely many multiplicative colorings of $K_n$.
\end{theorem}

\begin{proof}
See theorem \ref{thm:manymultiplicative} and corollary \ref{cor:finitelymultiplicative}.
\end{proof}

We explore $n$-satisfactory multiplicative colorings throughout the paper, and in particular in section \ref{sec:multiplicative}. We list all of them for $n\le 5$ in \S\,\ref{subs:nle5} and 
for $6\le n\le 8$ in \S\,\ref{sub:nle8}. In theorem \ref{thm:p2-p} we show that there are such colorings if $n=p^2-p$ for some prime $p$. The class of multiplicative colorings to which 
we devote most attention is the following, already encountered in theorem \ref{thm:positive}.

\begin{defin}{def:strongrepresentatives}
A strong representative of order $n$ is a prime $p$ of the form $kn+1$ for some $k$ such that $1^k,\dots,n^k$ are pairwise distinct modulo $p$. If there is such a prime $p$, we say 
that $n$ admits a strong representative.
\end{defin}

\begin{theorem}
If $p=kn+1$ is a strong representative of order $n$, then, up to renaming of the colors, the map $c(a)=(a^k\bmod p)$ is a multiplicative $n$-satisfactory coloring.
\end{theorem}

\begin{proof}
See the beginning of \S\,\ref{subs:multiplicative}.
\end{proof}

As indicated above, this gives us infinitely many examples of values of $n$ admitting $n$-satisfactory colorings. since it applies in particular when $n+1$ is prime (so $k=1$) and 
when $2n+1$ is prime (so $k=2$). On the other hand, if $k>2$, examples are harder to come by.

\begin{theorem} \label{thm:finite}
If $k>2$, then there are only finitely many $n$ such that $p=kn+1$ is a strong representative of order $n$.
\end{theorem}

\begin{proof}
There are no such primes $p$ when $k=3$, by theorem \ref{thm:3-representative}.

If $k$ is a multiple of 4, then any such prime $p$ must satisfy $p<k^2$, so there are only finitely many such $n$, by theorem \ref{thm:4m-representative}.

For the general case, see theorem \ref{thm:k-representative}. The argument uses the theory of Bernoulli polynomials and is due to Grinberg and Harcos.
\end{proof}

The proof of theorem \ref{thm:finite} provides us with an algorithm to find, for each $k>2$, all primes $p=kn+1$ that are strong representatives of order $n$. This is illustrated in 
\S\,\ref{subsec:k-representatives} with some examples. 

For fixed $k>2$, in the brief \S\,\ref{subs:asymptotics} we include two results by Elkies on the asymptotic number  of coincidences $a^k\equiv b^k\pmod p$ with $1\le a<b\le n$ as 
the prime $p=nk+1$ increases. That such coincidences occur is a consequence of theorem \ref{thm:finite}, and we feel that the inclusion of these observations rounds up the picture, 
as it provides a quantitative measure of how badly large values of $p$ of the form $kn+1$ fail to be strong representatives of order $n$. In particular, we have the following.

\begin{theorem}[Elkies]
For $k>2$, the number of coincidences $a^k\equiv b^k\pmod p$ for $p$ of the form $kn+1$ and sufficiently large, and distinct $a,b\in[n]$ is 
 $$ C_k p + O_k(p^{1-\epsilon(k)}), $$ 
where 
 $$ C_k = \begin{cases}
 \vspace{1mm}\displaystyle \frac{k - 1}{2k^2} & \mbox{ if $k$ is odd, and}\\ 
 \displaystyle \frac{k - 2}{2k^2} & \mbox{ if $k$ is even,}
 \end{cases}$$ 
and $\epsilon(k) = 1/\phi(k)$, where $\phi$ is Euler's totient function. 
\end{theorem}

\begin{proof}
See theorem \ref{thm:elkies}.
\end{proof}

The study of strong representatives is interesting in its own right. We devote section \ref{sec:pk+1} to it. In particular, using Chebotar\"ev's theorem and tools of algebraic number 
theory, we prove the following.

\begin{theorem}
If $n$ admits a strong representative $p$, then it admits infinitely many, and in fact the set of such primes is of positive natural density among all primes.
\end{theorem}

\begin{proof}
See theorem \ref{thm:positivenatural}.
\end{proof}

Besides this result, we also collect some related numerical data in \S\,\ref{subsec:density}. Part of the interest in this result is that early numerical explorations (while the 
second-named author was working on his master's thesis) suggested that the collection of strong representatives of order $n$ is very sparse, see for instance table 
\ref{table:representative}, and this result indicates that the opposite is indeed true. 

So far, our description of the results listed above emphasizes the number- and group-theoretic aspects of our work. We also bring to bear some geometric and combinatorial ideas,
by showing that the existence of $n$-satisfactory colorings is equivalent to the existence of certain tilings. To state the equivalence, recall that $\pi(n)$ is the number of prime 
numbers less than or equal to $n$. 

Say that a set $A\subseteq \mathbb Z^{\pi(n)}$ tiles another such set $C$ if and only if there is a $B$ such that $C$ is the direct sum of $A$ and $B$. Let $2=p_1<\dots<p_{\pi(n)}$ 
list the primes in $[n]$ in increasing order. Define $T_n$ be the image of $[n]$ under the map that sends $p_1^{\alpha_1}\cdots p_{\pi(n)}^{\alpha_{\pi(n)}}$ in $K_n$ to 
$(\alpha_1,\dots,\alpha_{\pi(n)})$ in the nonnegative orthant $\mathbb O_n$ of $\mathbb Z^{\pi(n)}$.

\begin{theorem} \label{thm:tile}
There is an $n$-satisfactory coloring of $K_n$ if and only if $T_n$ tiles a superset of $\mathbb O_n$.
\end{theorem}

\begin{proof}
See proposition \ref{proposition:tiling}.
\end{proof}

A compactness argument shows that we can replace $\mathbb O_n$ with $\mathbb Z^{\pi(n)}$ itself, and $K_n$ with the quotient field $\hat K_n=\{a/b:a,b\in K_n\}$, see proposition
\ref{prop:tilingz} and the remarks immediately preceding it.

Theorem \ref{thm:tile} transforms the problem of finding satisfactory colorings into a geometric question. The approach is fruitful, as it was essential to the results in 
\S\,\ref{subs:n6} and \S\,\ref{subs:n8}. 

Using tilings we also obtain an elegant characterization of multiplicative colorings. If $c$ is a coloring of $K_n$ and $k\in K_n$, let $c_k$ be the coloring where two numbers 
$m, m'\in K_n$ receive the same color if and only if $c(km) = c(km')$.

\begin{defin}{def:trans}
A coloring $c$ of $K_n$ is translation invariant if and only if $c_k = c$ for all $k\in K_n$.
\end{defin}

This admits a natural geometric description, see \S\,\ref{subs:translation}.

\begin{theorem}
An $n$-satisfactory coloring is translation invariant if and only if it is multiplicative.
\end{theorem}

\begin{proof}
See theorem \ref{thm:inv}.
\end{proof}

We admit we understand very little those colorings that are not multiplicative. We show examples in \S\,\ref{subs:n6} and \S\,\ref{subs:n8}, but more is needed. In particular, whether 
question \ref{q:problem} admits a positive answer depends essentially on whether there are nonmultiplicative $n$-satisfactory colorings for various $n$, such as $n=195$. The 
point is that there are various $n$ which do not admit multiplicative colorings. This is briefly discussed in \S\,\ref{subsec:groupless}, which reviews the work of 
Forcade and Pollington \cite{ForcadePollington90}. These numbers $n$ we call groupless. Several questions we ask suggest ways of trying to understand some of the structure 
of nonmultiplicative colorings, see in particular question \ref{q:closed}, which refers to the topology of the collection of $n$-satisfactory colorings, described in item (1) of 
\S\,\ref{subsec:structure}.

\section{The core} \label{sec:preliminaries}

An \emph{$n$-coloring} of a set $X$ is a coloring of $X$ using exactly $n$ colors. An \emph{$n$-satisfactory} coloring is an $n$-coloring witnessing a positive answer to the 
$n^{\mathrm{th}}$ instance of question \ref{q:problem}. The nature of these colors is of course irrelevant, but we need some convention since we want to address questions about the 
number of $n$-colorings satisfying some property (such as, primarily, being $n$-satisfactory). There are two natural ways of thinking about $n$-colorings, and we adopt both 
in what follows. We will typically consider only colorings of the $n$-core $K_n$ rather than of all of $\mathbb Z^+$, but what follows applies in either case. 

In the first approach, we think of an $n$-coloring as a map $c$ with range $[n]$, and we further adopt the convention that $c(i)=i$ for $i\in[n]$. The point of this convention is 
to avoid overcounting when looking at the number of $n$-satisfactory colorings for  fixed $n$. For instance, as we will see in \S\,\ref{subs:nle5}, there is precisely one 
3-satisfactory coloring of $K_3$, but without the convention it would seem as if there are six.

The second approach is perhaps more natural: rather than thinking of a coloring as a map, we think of it as an equivalence relation, whose classes are precisely the colors. We 
still adopt functional notation, so we write, for example, $c(a)=c(b)$ to indicate that $c$ assigns the same color to the numbers $a$ and $b$. 

Still, on occasion we may stray from these conventions for ease of exposition.

\subsection{The core $K_n$ and $n$-appropriate sets}

\begin{definition} \label{def:ncore}
The \emph{$n$-core}, or simply the \emph{core} if $n$ is understood, is the set $K_n$ of all positive integers whose prime decomposition only involves primes less than or equal 
to $n$. This is the set of numbers usually called \emph{$n$-smooth}.
\end{definition}

In the literature the notion of $n$-smooth numbers is typically reserved for the case where $n$ itself is a prime number. We do not impose this requirement so, for instance, 
$K_7=K_8=K_9=K_{10}$. In the notation of \cite{Boseketal18}, $K_n$ is denoted $\mathbb N_n$.

The key reason for considering cores is that there is an $n$-satisfactory coloring (of $\mathbb Z^+$) if and only if there is an $n$-satisfactory coloring of the $n$-core. In fact 
we prove a bit more, indicating that in order to understand the structure of the set of $n$-satisfactory colorings, attention can be restricted to those of the $n$-core. Once we 
establish this fact, we proceed accordingly, which explains the title of this paper.

In particular, restricting attention to colorings of the $n$-core allows us to address the following question.

\begin{question} \label{q:number}
Given $n>1$, how many $n$-satisfactory colorings are there, if any at all?
\end{question}

As we will see, the answer to question \ref{q:number} is $\mathfrak c=|\mathbb R|$ for colorings of $\mathbb Z^+$ even in cases where there are only finitely many 
$n$-satisfactory colorings of $K_n$, see corollary \ref{cor:many} and theorem \ref{thm:manymultiplicative}.

\begin{definition} \label{def:appropriate} 
Say that $X\subseteq\mathbb Z^+$ is \emph{$n$-appropriate} if and only if $X$ is nonempty and contains $ix$ and $x/ j$ whenever $x \in X$, $i,j \le n$, and $j$ divides $x$.

If $X$ is $n$-appropriate, say that an $n$-coloring $c$ of $X$ is $n$-satisfactory if and only if $c(ix) \ne c(jx)$ whenever $x \in X$ and $i < j \le n$. Note that this coincides 
with the previous notion of $n$-satisfactory when $X = \mathbb Z^+$ (or $X=K_n$). Considering colorings as maps, if $1\in X$ we add the restriction mentioned earlier that 
$n$-satisfactory colorings must be the identity on $[n]$. 
\end{definition}

Note that we are insisting that if $X$ is $n$-appropriate, $1\in X$, and $c$ is $n$-satisfactory on $X$, then $c$ is the identity on $[n]$, while we impose no such restrictions on
the satisfactory colorings of other appropriate sets; for instance, one could wonder why we do not ask that if $a$ is the minimum of $X$, then $c(ai)=i$ for all $i\in[n]$. The reason 
for this convention is that we want that if $X,Y$ are disjoint and $n$-appropriate, then the union of $n$-satisfactory colorings of $X$ and $Y$ is an $n$-satisfactory coloring of 
$X\cup Y$, and any $n$-satisfactory coloring of $X\cup Y$ is obtained this way. 

We denote by $P_n$ the set of numbers relatively prime to $n!$, i.e., those positive integers whose prime decomposition only involves prime numbers strictly larger than $n$. In the 
literature, these numbers are referred to as \emph{$n$-rough}. Note that 1 is $n$-rough for any $n$.

The notation $X = \dot\bigcup_{a\in A} X_a$ means both that $X$ is the union of the sets $X_a$ for $a\in A$, and that the sets $X_a$ are pairwise disjoint.

\begin{lemma} \label{lem:nappropriate}
A set $X\subseteq \mathbb Z^+$ is $n$-appropriate if and only if there is a nonempty set $A \subseteq P_n$ such that
 $$ X = \dot{\bigcup_{a \in A}}a\cdot K_n. $$
Moreover, if this is the case, then we have $A = P_n \cap X$.
\end{lemma}

\begin{proof}
Note that $K_n$ is $n$-appropriate and therefore so is $a \cdot K_n$ for any $a \in P_n$. It follows that any $X$ of the form $\dot\bigcup_{a\in A} a\cdot K_n$ for 
$A \subseteq P_n$ and nonempty is $n$-appropriate as well.

Towards the converse, suppose now that $X$ is $n$-appropriate. Each $m \in\mathbb Z^+$ can be uniquely written in the form $m=a_mk_m$ where $a_m \in P_n$ and 
$k_m \in K_n$. Let $A=\{a_x : x\in X\}$. We claim that $X = \dot\bigcup_{a\in A} a \cdot K_n$.

First, note that if $a \ne b$ are in $P_n$, then $a\cdot K_n$ and $b\cdot K_n$ are pairwise disjoint. Now, if $a \in A$, then there is some $x \in X$ such that $a = a_x$. 
Since $X$ is $n$-appropriate, $h /j \in X$ whenever $h \in X$ and $j \in K_n$ divides $h$. In particular, $a = a_x = x /k_x \in X$. Similarly, $hi \in X$ whenever $h \in X$ and  
$i \in K_n$. Therefore, $a \cdot K_n \subseteq X$. This means that
 $$  \dot{\bigcup_{a\in A}}a\cdot K_n \subseteq X. $$
But, if $x\in X$, then $x\in a_x \cdot K_n$, and we have that
  $$ \dot{\bigcup_{a\in A}}a\cdot K_n \supseteq X. $$
This proves the equality and establishes the equivalence.

Second, if $a\in P_n$, then the only member of $P_n$ in $a\cdot K_n$ is $a$ itself. It follows that if $X=\dot\bigcup_{a\in A}a\cdot K_n$ for some $A\subseteq P_n$, then 
in fact $A=P_n\cap X$, and we are done.
\end{proof}

For $X$ $n$-appropriate, let $C_{X,n}$ be the set of $n$-satisfactory colorings of $X$, and denote by $C_n$ the set $C_{\mathbb Z^+,n}$. We also write $C_X,C$ if $n$ is clear
from context.

The following proposition shows that $C \ne\emptyset$ if and only if $C_X \ne\emptyset$ for some $n$-appropriate set $X$ if and only if $C_X\ne\emptyset$ for all $n$-appropriate 
sets $X$.

In particular, as emphasized earlier, it follows that the question of whether there are any $n$-satisfactory colorings is really a question about whether there are $n$-satisfactory 
colorings of $K_n$. In fact, the proposition shows how the satisfactory colorings of the core completely determine all satisfactory colorings.

\begin{proposition} \label{prop:appropriate}
Let $n\in\mathbb Z^+$.
\begin{enumerate}
\item
Suppose $X \subseteq Y$ are $n$-appropriate. If $C_Y \ne\emptyset$, then $C_X \ne\emptyset$. In fact, thinking of colorings as maps, the restriction $c\upharpoonright X$ 
is in $C_X$ for any $c\in C_Y$.
\item
Given $a\in P_n$ and $c\in C_{a\cdot K_n}$, if $\dv(c,a)$ denotes the $n$-coloring of $K_n$ such that 
 $$ \dv(c,a)(k)=\dv(c,a)(l)\quad \mbox{ if and only if }\quad c(ak)=c(al) $$ 
for all $k,l\in K_n$, then $\dv(c,a)\in C_{K_n}$. Considering colorings as maps with range $[n]$, 
 $$ \dv(c,a)(k)=\pi\circ c(ak) $$ 
for all $k\in K_n$, where $\pi$ is the permutation of $n$ such that $\pi(c(ai))=i$ for all $i\in [n]$.
\item
Given $a\in P_n$ and $c\in C_{K_n}$, if $\mult(c,a)$ denotes the $n$-coloring of $a\cdot K_n$ such that 
 $$ \mult(c,a)(ak)=\mult(c,a)(al)\quad \mbox{ if and only if }\quad c(k)=c(l) $$
for all $k,l\in K_n$, then $\mult(c,a)\in C_{a\cdot K_n}$. Considering colorings as maps, $\mult(c,a)(a\cdot k)=c(k)$ for all $k\in K_n$. However, note that if $a\ne 1$, then for any 
permutation $\pi$ of $[n]$, $\pi\circ\mult(c,a)$ is also in $C_{a\cdot K_n}$.
\item
If $X$ is $n$-appropriate, then a map $c$ is in $C_X$ if and only if for each number $a\in P_n \cap X$ there is a map $c_a \in C_{K_n}$ and a permutation $\pi_a$ of $[n]$ such 
that $\pi_1$ is the identity and 
 $$ c= \dot{\bigcup_{a\in P_n\cap X}} \pi_a\circ \mult(c_a,a). $$
\item 
$C\ne\emptyset$ if and only if $C_X\ne\emptyset$ for any $n$-appropriate $X$ if and only if $C_X\ne\emptyset$ for some $n$-appropriate set $X$. Moreover if 
$X \subseteq Y$ and both are $n$-appropriate, then $d\in C_X$ if and only if $d=c \upharpoonright X$ for some $c \in C_Y$.
\end{enumerate}
\end{proposition}

\begin{proof}
(1) This is clear.

\vspace{1mm}
(2) Given $a\in P_n$ and $c\in C_{a\cdot K_n}$, if $\dv(c,a)$ is defined as in item (2), then 
 $$ \dv(c,a)(ib)\ne \dv(c,a)(jb) $$
for any $b\in K_n$ and any $i<j\le n$ because $c$ is $n$-satisfactory and therefore $c(aib)\ne c(ajb)$. But this means that $\dv(c,a)$ is $n$-satisfactory as well.  

Considering colorings as functions, $c$ is a map with range $[n]$ and the only obstacle for $k\mapsto c(ak)$ to be an $n$-satisfactory coloring is the additional restriction we have 
imposed that such a map must be the identity on $[n]$, which explains why we may need to precompose it with a permutation to achieve this.

\vspace{1mm}
(3) Conversely, if $c \in C_{K_n}$, $a \in P_n$, and $\mult(c,a)$ is defined as in item (3), then 
 $$ \mult(c,a)(im) \ne \mult(c,a)(jm) $$ 
for any $m \in a\cdot K_n$ and any $i < j \le n$ since $c$ is satisfactory and therefore $c(i(m /a)) \ne c(j(m/ a))$. But this means that $\mult(c,a)$ is satisfactory as well. Considering 
colorings as maps, the inequalities just indicated are maintained under any permutation $\pi$ of $[n]$, so if $a\ne 1$, then $\pi\circ\mult(c,a)$ is also satisfactory.

\vspace{1mm}
(4) Suppose that $X$ is $n$-appropriate. First, $X=\bigcup_{a\in P_n\cap X}a\cdot K_n$, by lemma \ref{lem:nappropriate}. If $c\in C_X$, then, by item (1), 
$d=c \upharpoonright a\cdot K_n \in C_{a\cdot K_n}$ for any $a\in P_n\cap X$, and $c_a \coloneqq \dv(d,a)\in C_{K_n}$ by item (2). Writing $\pi_a^{-1}$ for the permutation as 
in item (2), we have that $d = \pi_a\circ\mult(c_a,a)$, and therefore
$$ c =  \bigcup_{a\in P_n\cap X} \pi_a\circ \mult(c_a,a). $$

Conversely, suppose that $C_{K_n}\ne\emptyset$. For each $a\in P_n \cap X$ let $c_a \in C_{K_n}$ and $\pi_a$ be a permutation of $[n]$, with $\pi_1$ being the identity if 
$1\in X$. Define 
 $$ c=\bigcup_{a\in P_n\cap X}\pi_a\circ\mult(c_a,a). $$ 
As mentioned in lemma \ref{lem:nappropriate}, $a\cdot K_n\cap b\cdot K_n = \emptyset$ whenever $a\ne b$ are in $P_n$. From this, and item (3), $c$ is well defined and has 
domain $\bigcup_{a\in P_n\cap X} a \cdot K_n$, which equals $X$, again by lemma \ref{lem:nappropriate}. If $m \in X$ and $i < j \le n$, then there is a unique $a \in P_n \cap X$ 
such that $mi$ and $mj$ belong to $a \cdot K_n$, and by item (3) it follows that $c(mi) \ne c(mj)$. This proves that $c$ is satisfactory, and completes the proof of item (4).

\vspace{1mm}
(5) Now, if $X$ is $n$-appropriate, and $C_X\ne\emptyset$, then $C_{a\cdot K_n}\ne\emptyset$ for any $a$ in the nonempty set $P_n \cap X$, by item (1). But this implies
that $C_{K_n}\ne\emptyset$, by item (2). It follows from item (4) that $C =C_{\mathbb Z^+}\ne \emptyset$. Thus, $C_Y\ne\emptyset$ for any $n$-appropriate $Y$, again by 
item (1).

Finally, if $X\subseteq Y$ are $n$-appropriate and $c\in C_Y$, then $d=c \upharpoonright X\in C_X$, by item (1). Conversely, if $d\in C_X$, let $e \in C_{K_n}$, which exists as 
shown above. Let $c_a = e$ and $\pi=\id$ for $a\in P_n\cap (Y\sbs X)$. For $a\in P_n\cap X$, let $c_a =\dv(d\upharpoonright a\cdot K_n,a)$ and $\pi_a$ be the permutation 
such that $d\upharpoonright a\cdot K_n=\pi_a\circ \mult(c_a,a)$ .  As in item (4), we have that $c = \bigcup_{a\in P_n\cap Y} \pi_a\circ \mult(c_a,a) \in C_Y$. And, by construction, 
$d = c\upharpoonright X$. This completes the proof of item (5).
\end{proof}

\begin{remark} \label{rmk:quotient}
The notion of $n$-appropriate can be extended in a natural way, allowing us to verify that, for instance, there is an $n$-satisfactory coloring of $K_n$ if and only if there is 
one of $\mathbb Z\sbs\{0\}$. More interesting is whether this is also equivalent to the existence of an $n$-satisfactory coloring of $\mathbb Q\sbs\{0\}$ or, what is the same, 
of $\hat K_n\coloneqq\{a/b:a,b\in K_n\}$. We show below that this is indeed the case, see proposition \ref{prop:tilingz}. We also suggest a subtler problem in question 
\ref{q:equation}. In \cite{Boseketal18}, $\hat K_n$ is denoted $\mathbb Q_n$.
\end{remark}

When discussing $n$-satisfactory colorings, proposition \ref{prop:appropriate} provides us with the ability to restrict our attention from all of $\mathbb Z^+$ to $K_n$. The 
relation the proposition details between arbitrary satisfactory colorings and colorings of the core has the following corollary.

\begin{corollary} \label{cor:many}
For $n > 1$, if $C_{K_n}\ne\emptyset$, then $|C|  = \mathfrak c$.
\end{corollary}

\begin{proof}
If there is a coloring of the core (thought of as a function with range $[n]$), then there are at least $n! \ge 2$ such colorings of any $a\cdot K_n$ for $a\in P_n$ different from $1$,
obtained by invoking item (3) of proposition \ref{prop:appropriate} and varying the permutation $\pi$.  By item (4) of proposition \ref{prop:appropriate}, there is a bijective 
correspondence between the elements of $C$, and the set of functions with domain $P_n$ that pick for each $a\in P_n$ a member of $C_{K_n}$ and a permutation of $[n]$ 
(with the permutation being the identity if $a=1$), from which we get that $|C|  \ge  n!^{|P_n\sbs\{1\}|} = \mathfrak c$.

On the other hand, any element of $C$ is a function from $\mathbb Z^+$ to $[n]$ (satisfying certain restrictions), so $|C|\le|[n]^{\mathbb Z^+}|  = n^{\aleph_0}=\mathfrak c$, and 
it follows that $|C|=\mathfrak c$ by the Cantor--Schr\"oder--Bernstein theorem.
\end{proof}

Proposition \ref{prop:appropriate} and corollary \ref{cor:many} give us that if there is an $n$-satisfactory coloring of $K_n$, then there are as many $n$-satisfactory colorings of 
$\mathbb Z^+$ as there are real numbers. However, this abundance of colorings is a distraction since the underlying structure of any satisfactory coloring can be described in 
terms of what is happening on the core.

\subsection{The structure of $C_{K_n}$} \label{subsec:structure}

We mention here some easy observations regarding the closure of $C_{K_n}$ under some natural operations.

\vspace{1mm}
(1) First, $C_{K_n}$ is a closed subset of the Polish space $[n]^{K_n}$ of functions from $K_n$ to $[n]$ under the product topology (with $[n]$ discrete): $c\in C_{K_n}$ if and 
only if 
 $$ c\in \{f\in [n]^{K_n}:\forall i\in[n]\,(f(i)=i)\}\cap\bigcap_{a\in K_n}\bigcap_{1\le i<j\le n}\{g\in[n]^{K_n}:g(ia)\ne g(ja)\}, $$
and note that for any distinct $b,c\in K_n$, 
 $$ \{g\in[n]^{K_n}:g(b)\ne g(c)\}=\bigcup_{\substack{(\alpha,\beta)\in[n]\times [n]\\ \alpha\ne \beta}} \{g\in[n]^{K_n}:g(b)=\alpha\mbox{ and }g(c)=\beta\} $$
is a finite union of closed sets. 

This topological fact is trivial in some cases, since (as shown in \S\,\ref{subs:nle5}) $C_{K_n}$ is sometimes finite, but see \S\,\ref{subs:n6}. The Polish topology of the space 
$[n]^{K_n}$ is generated by a natural metric: enumerate $K_n$ in increasing order as $\{k_i:i\in\mathbb Z^+\}$. The distance between two distinct colorings $c,c'$ is $1/N$, 
where $N$ is the least index of a disagreement, that is, $c(k_i)=c'(k_i)$ for all $i<N$, but $c(k_N)\ne c'(k_N)$. This metric is complete both in the whole space $[n]^{K_n}$ and 
in $C_{K_n}$.

The truth is, we understand very little of the topological structure of $C_{K_n}$. It is unclear, for instance, whether the following question should have a positive answer.

\begin{question} \label{q:closed}
Given $n\in\mathbb Z^+$, suppose that $C_{K_n}$ is nonempty. Should it have isolated points?
\end{question}

\vspace{1mm}
(2) The following is an immediate but useful observation.

\begin{lemma} \label{lem:indiscernible}
Let $\rho$ be an automorphism of the structure $([n],|)$, that is, of the Hasse diagram for divisibility on $[n]$. Extend $\rho$ to a bijection of $K_n$ in the natural way: if 
$m=2^{\alpha_1}\dots {p_k}^{\alpha_k}$ is the prime factorization of $m\in K_n$, then 
\begin{equation} \label{eq:rho}
\rho(m)=\rho(2)^{\alpha_1}\dots \rho(p_k)^{\alpha_k}.
\end{equation} 
If $c$ is an $n$-satisfactory coloring of $K_n$, then so is $\tilde c$, where $\tilde c(a)=\tilde c(b)$ if and only if $c(\rho(a))=c(\rho(b))$.
\end{lemma}

For an application, see the discussion of the case $n=5$ in \S\,\ref{subs:nle5}, where it is also shown that the more inclusive condition that $\rho$ be a permutation of the set of 
primes in $[n]$ is not enough in general. 

\begin{proof}
Note first that $\rho$ permutes the primes less than or equal to $n$, and equation (\ref{eq:rho}) holds for all $m\in[n]$, so that the suggested extension is well defined and maps 
$K_n$ to itself. Since $\rho$ is a permutation on the primes in $[n]$, it follows as well that $\rho$ is surjective on $K_n$. Note also that for any $m_1,m_2\in K_n$, 
$\rho(m_1m_2)=\rho(m_1)\rho(m_2)$.

Now, if $c$ is satisfactory and $\tilde c$ is as indicated, then for $i\ne j$ in $[n]$ and $a\in K_n$, we have that $\tilde c(ia)=c(\rho(i)\rho(a))\ne c(\rho(j)\rho(a))=\tilde c(ja)$ since 
$\rho(i),\rho(j)\in[n]$.
\end{proof}

\vspace{1mm}
(3) Another natural operation on $C_{K_n}$ can be defined by letting $c_k\in C_{K_n}$, for $c\in C_{K_n}$ and $k\in K_n$, be given by $c_k(l)=c_k(m)$ if and only if 
$c(kl)=c(km)$. (Abusing slightly\footnote{In proposition \ref{prop:appropriate} we require $k\in P_n$, but $P_n\cap K_n=\{1\}$.} the notation used in proposition 
\ref{prop:appropriate}, $c_k=\dv(c\upharpoonright k\cdot K_n,k)$.) We remark that although we concentrate on $n$-satisfactory colorings throughout the whole paper, on occasion 
we may consider \emph{translations} $c_k$ of arbitrary $n$-colorings $c$, with the understanding that the definition just given applies in general.

Most of the colorings we consider in this paper are \emph{multiplicative} (see \S\,\ref{subs:multiplicative}). For them, this operation is uninteresting: $c=c_k$ for any $k\in K_n$ 
whenever $c$ is multiplicative. However, the operation may generate new colorings otherwise, see \S\,\ref{subs:n6}. It also suggests the following natural problem.

\begin{question} \label{q:equation}
Given an $n$-satisfactory coloring $c$ and $k\in K_n$, is there is an $n$-satisfactory coloring $d$ such that $d_k=c$? In that case, how many such colorings $d$ are there?
\end{question}

\vspace{1mm}
(4) Because $C_{K_n}$ is closed in $[n]^{K_n}$, it is also closed under a construction coming from applications of K\H onig's infinity lemma. We discuss this construction in the 
next subsection, once the appropriate notation and terminology have been introduced, see remark \ref{rem:konig}.

\subsection{Tilings} \label{subs:tilings}

The switch from $\mathbb Z^+$ to the set $K_n$ allows us to restate question \ref{q:problem} as a problem about tilings (this restatement is also mentioned by P\'alv\"olgyi 
on his post in MathOverflow, and is the subject of \cite[\S\,4]{Boseketal18}). 

To simplify the description, consider for now the case $n=3$. In this case, the problem lives in the integer grid. Identify $m=2^a3^b\in K_3$ with the point $(a,b)$ in the 
first quadrant of the integer lattice or, equivalently, the unit square with sides parallel to the axes and bottom left corner at $(a,b)$. Now, a coloring is 3-satisfactory if and only 
if for any such pair $(a,b)$, the pairs $(a,b)$, $(a+1,b)$ (corresponding to $2m$) and $(a,b+1)$ (corresponding to $3m$) all receive different colors. The question of whether 
there is a 3-satisfactory coloring of $K_3$ becomes the question of whether we can assign to each unit square in the first quadrant one of three colors in such a way that all 
translates of the triomino consisting of the three unit squares in the bottom left corner of the quadrant  contain tiles of all colors.

The case $n=3$ is simple enough that one can easily see that there is a unique way of accomplishing this, illustrated in figures \ref{figure:3tiling} and \ref{figure:3tiling-b}. 
(The names of the colors in figure \ref{figure:3tiling} are chosen so that the color of each $i=1,2,3$ is $i$ itself.)

\begin{figure}[ht]
\begin{tikzpicture}[scale=.4]
\draw[step=1cm,gray,very thin] (0,0) grid (6.1,6.1);
\draw[thick,->] (0,0) -- (6.1,0);
\draw[thick,->] (0,0) -- (0,6.1);

    \setcounter{row}{1}
    \setrow {\scriptsize 2}{\scriptsize 3}{\scriptsize 1}{\scriptsize 2}{\scriptsize 3}{\scriptsize 1}  
    \setrow {\scriptsize 3}{\scriptsize 1}{\scriptsize 2}{\scriptsize 3}{\scriptsize 1}{\scriptsize 2}  
    \setrow {\scriptsize 1}{\scriptsize 2}{\scriptsize 3}{\scriptsize 1}{\scriptsize 2}{\scriptsize 3}
    \setrow {\scriptsize 2}{\scriptsize 3}{\scriptsize 1}{\scriptsize 2}{\scriptsize 3}{\scriptsize 1}  
    \setrow {\scriptsize 3}{\scriptsize 1}{\scriptsize 2}{\scriptsize 3}{\scriptsize 1}{\scriptsize 2}  
    \setrow {\scriptsize 1}{\scriptsize 2}{\scriptsize 3}{\scriptsize 1}{\scriptsize 2}{\scriptsize 3}
        
\draw[very thick]  (0,0) -- (2,0) -- (2,1) -- (1,1) -- (1,2) -- (0,2) -- (0,0);
\draw[blue,very thick] (3,2) -- (5,2) -- (5,3) -- (4,3) -- (4,4) -- (3,4) -- (3,2);

\end{tikzpicture}
\caption{Tiling of the first quadrant of $\mathbb Z^2$ corresponding to a 3-satisfactory coloring. The relevant triomino is shown at the bottom left corner. Any copy of the triomino 
contains all three colors.}
\label{figure:3tiling}
\end{figure}

\begin{figure}[ht]
\begin{tikzpicture}[scale=.75]
\draw[step=1cm,gray,very thin] (-0.1,-0.1) grid (6.1,6.1);
\draw[thick,->] (0,0) -- (6.1,0) node [anchor=north west] {$\times2$};
\draw[thick,->] (0,0) -- (0,6.1) node [anchor=south east] {$\times3$};
\draw[very thick] (0,3) -- (6.1,3);
\draw[very thick] (3,0) -- (3,6.1);
\draw[very thick] (0,6) -- (6.1,6);
\draw[very thick] (6,0) -- (6,6.1);

   \setcounter{row}{1}
    \setrow {{\color{blue}{$243_2$}}}{$486_3$}{{\color{red}{$972_1$}}}{\footnotesize{\color{blue}{$1944_2$}}}{\footnotesize{$3888_3$}}{\footnotesize{\color{red}{$7776_1$}}}
    \setrow {$81_3$}{{\color{red}{$162_1$}}}{{\color{blue}{$324_2$}}}{$648_3$}{\footnotesize{\color{red}{$1296_1$}}}{\footnotesize{\color{blue}{$2592_2$}}}   
    \setrow {{\color{red}{$27_1$}}}{{\color{blue}{$54_2$}}}{$108_3$}{{\color{red}{$216_1$}}}{{\color{blue}{$432_2$}}}{$864_3$}
    \setrow {{\color{blue}$9_2$}}{$18_3$}{{\color{red}$36_1$}}{{\color{blue}$72_2$}}{$144_3$}{{\color{red}{$288_1$}}}
    \setrow {$3_3$}{{\color{red}$6_1$}}{{\color{blue}$12_2$}}{$24_3$}{{\color{red}{$48_1$}}}{{\color{blue}{$96_2$}}}
    \setrow {{\color{red}$1_1$}}{{\color{blue}$2_2$}}{$4_3$}{{\color{red}$8_1$}}{{\color{blue}{$16_2$}}}{$32_3$}
\end{tikzpicture}
\vspace{-1mm}
\caption{The unique 3-satisfactory coloring of $K_3$. (Subindices indicate colors.) Notice the periodicity of the coloring, resulting in a tiling of the first quadrant with identically 
colored $3\times 3$ squares.}
\label{figure:3tiling-b}
\end{figure}

Before proceeding, the reader may enjoy verifying that, similarly, there is a unique 4-satisfactory tiling of the first quadrant, as illustrated in figures \ref{figure:4tiling} and 
\ref{figure:4tiling-b}.

\begin{figure}[ht]
\begin{tikzpicture}[scale=.4]
\draw[step=1cm,gray,very thin] (0,0) grid (8.1,8.1);
\draw[thick,->] (0,0) -- (8.1,0);
\draw[thick,->] (0,0) -- (0,8.1);

    \setcounter{rowf}{1}
    \setrowf {\scriptsize 2}{\scriptsize 4}{\scriptsize 3}{\scriptsize 1}{\scriptsize 2}{\scriptsize 4}{\scriptsize 3}{\scriptsize 1}  
    \setrowf {\scriptsize 4}{\scriptsize 3}{\scriptsize 1}{\scriptsize 2}{\scriptsize 4}{\scriptsize 3}{\scriptsize 1}{\scriptsize 2}  
    \setrowf {\scriptsize 3}{\scriptsize 1}{\scriptsize 2}{\scriptsize 4}{\scriptsize 3}{\scriptsize 1}{\scriptsize 2}{\scriptsize 4}
    \setrowf {\scriptsize 1}{\scriptsize 2}{\scriptsize 4}{\scriptsize 3}{\scriptsize 1}{\scriptsize 2}{\scriptsize 4}{\scriptsize 3}  
    \setrowf {\scriptsize 2}{\scriptsize 4}{\scriptsize 3}{\scriptsize 1}{\scriptsize 2}{\scriptsize 4}{\scriptsize 3}{\scriptsize 1}  
    \setrowf {\scriptsize 4}{\scriptsize 3}{\scriptsize 1}{\scriptsize 2}{\scriptsize 4}{\scriptsize 3}{\scriptsize 1}{\scriptsize 2}  
    \setrowf {\scriptsize 3}{\scriptsize 1}{\scriptsize 2}{\scriptsize 4}{\scriptsize 3}{\scriptsize 1}{\scriptsize 2}{\scriptsize 4}
    \setrowf {\scriptsize 1}{\scriptsize 2}{\scriptsize 4}{\scriptsize 3}{\scriptsize 1}{\scriptsize 2}{\scriptsize 4}{\scriptsize 3}  
        
\draw[very thick]  (0,0) -- (3,0) -- (3,1) -- (1,1) -- (1,2) -- (0,2) -- (0,0);
\draw[blue,very thick] (4,3) -- (7,3) -- (7,4) -- (5,4) -- (5,5) -- (4,5) -- (4,3);

\end{tikzpicture}
\caption{Tiling of the first quadrant of $\mathbb Z^2$ corresponding to a 4-satisfactory coloring. The relevant polyomino is shown at the bottom left corner. Any copy of the 
polyomino contains all 4 colors.}
\label{figure:4tiling}
\end{figure}

\begin{figure}[ht]
\begin{tikzpicture}[scale=.75]
\draw[step=1cm,gray,very thin] (0,0) grid (6.1,6.1);
\draw[thick,->] (0,0) -- (6.1,0) node [anchor=north west] {$\times 2$};
\draw[thick,->] (0,0) -- (0,6.1) node [anchor=south east] {$\times 3$};
\draw[very thick] (0,4) -- (6.1,4);
\draw[very thick] (4,0) -- (4,6.1);

    \setcounter{row}{1}
    \setrow {$243_3$}{{\color{red}{$486_1$}}}{{\color{blue}{$972_2$}}}{\footnotesize{\color{purple}{$1944_4$}}}{\footnotesize{$3888_3$}}{\footnotesize{\color{red}{$7776_1$}}}
    \setrow {{\color{red}{$81_1$}}}{{\color{blue}{$162_2$}}}{{\color{purple}{$324_4$}}}{$648_3$}{\footnotesize{\color{red}{$1296_1$}}}{\footnotesize{\color{blue}{$2592_2$}}} 
    \setrow {{\color{blue}{$27_2$}}}{{\color{purple}{$54_4$}}}{$108_3$}{{\color{red}{$216_1$}}}{{\color{blue}{$432_2$}}}{{\color{purple}{$864_4$}}} 
    \setrow {{\color{purple}{$9_4$}}}{$18_3$}{{\color{red}{$36_1$}}}{{\color{blue}{$72_2$}}}{{\color{purple}{$144_4$}}}{$288_3$} 
    \setrow {$3_3$}{{\color{red}{$6_1$}}}{{\color{blue}{$12_2$}}}{{\color{purple}{$24_4$}}}{$48_3$}{{\color{red}{$96_1$}}}
    \setrow {{\color{red}{$1_1$}}}{{\color{blue}{$2_2$}}}{{\color{purple}{$4_4$}}}{$8_3$}{{\color{red}{$16_1$}}}{{\color{blue}{$32_2$}}} 
        
\end{tikzpicture}
\vspace{-1mm}
\caption{The unique 4-satisfactory coloring of $K_4$. The coloring is periodic, resulting in a tiling of the first quadrant with identically colored $4\times 4$ squares.}
\label{figure:4tiling-b}
\end{figure}

Further cases are harder to illustrate, as they correspond in general to tilings of the first orthant of $\mathbb Z^{\pi(n)}$ where, as usual, $\pi(\cdot)$ denotes the prime 
counting function (and in general lack the periodicity displayed in these two examples, but see remark \ref{rmk:periodic}). These tilings use unit ``cubes'' of $n$ possible 
colors as tiles. More interestingly, we can instead restate question \ref{q:problem} as a problem about tilings with translates of the $\pi(n)$-dimensional polyomino 
corresponding to the set $\{1,2,\dots,n\}$ as tiles. We proceed now to explain this connection. 

Given $n$, work in $\mathbb Z^{\pi(n)}$. As suggested above, we identify each member of $K_n$ with the tuple of its prime exponents: any $m\in K_n$ can be written in a 
unique way as $m=\prod_{i=1}^{\pi(n)}p_i^{\alpha_i}$, where $2=p_1<\dots<p_{\pi(n)}$ are the primes less than or equal to $n$, listed in increasing order, and the $\alpha_i$ 
are nonnegative integers. We identify $m$ with the tuple $t(m)=(\alpha_1,\dots,\alpha_{\pi(n)})$ in the first orthant $\mathbb O_n$ of $\mathbb Z^{\pi(n)}$, noting that 
$t\!:K_n\to\mathbb O_n$ is a bijection, and let $T_n=\{t(i): i\in [n]\}$. Note that $t$ turns multiplication into vector addition in the sense that $t(kk')=t(k)+t(k')$ for any 
$k,k'\in K_n$. We will find several maps with similar properties in what follows, see for instance definition \ref{def:multiplicative}, where we call them \emph{multiplicative}.

Given $A,C\subseteq \mathbb Z^{\pi(n)}$, say that \emph{$A$ tiles $C$} (or, equivalently, that \emph{$C$ can be tiled by $A$}) if and only if there is a set 
$B\subseteq \mathbb Z^{\pi(n)}$ such that $C$ is the direct sum of $A$ and $B$, that is,
\begin{enumerate}
\item
$C=A+B\coloneqq\{a+b: a\in A,b\in B\}$, and in fact
\item
any $c\in C$ admits a unique decomposition as a sum of a member of $A$ and a member of $B$, that is, there is a unique pair $(a,b)\in A\times B$ with $c=a+b$.
\end{enumerate}
Also, say that \emph{$A$ essentially tiles $C$} if and only if $C$ can be covered by a set that can be tiled by $A$ (in which case, we call such a tiling of a superset of $C$ an 
\emph{essential tiling} of $C$ by $A$). We remark that, as we did above, we may identify without further comment points $(\alpha_i: i\in[\pi(n)])$ in $\mathbb Z^{\pi(n)}$ with 
the corresponding $\pi(n)$-dimensional cubes $\{(x_i: i\in[\pi(n)]): \alpha_i\le x_i\le \alpha_i+1\}$.

For a fixed value of $n$, consider now the following two statements:
\begin{enumerate}[(i)]
\item
There is an $n$-satisfactory coloring of $K_n$.
\item
$T_n$ essentially tiles the orthant $\mathbb O_n$.
\end{enumerate}

We have the following result.

\begin{proposition} \label{proposition:tiling}
With notation as above, (i) and (ii) are equivalent.
\end{proposition}

\begin{proof}
To see that (ii) implies (i), consider a tiling by $T_n$ of a superset $C$ of $\mathbb O_n$, say $C=T_n+B$, the sum being direct. Note that via this direct sum, each element of 
$C$, and therefore each $x\in\mathbb O_n$, belongs to exactly one tile, that is, a unique copy of $T_n$. There is a unique $m\in K_n$ such that $x=t(m)$, where $t$ is the map 
described above. Color $m$ with the position of $x$ within this tile. In other words, let the color classes be the preimages under $t$ of the sets $a+B$ for $a\in T_n$. We must 
argue that this coloring is $n$-satisfactory. Indeed, given $k\in K_n$ and $i,j\in[n]$, suppose that $ki$ and $kj$ receive the same color, that is, there are $\alpha\in T_n$ and 
$b_1,b_2\in B$ such that $t(ki)=\alpha+b_1$ and $t(kj)=\alpha+b_2$. By the multiplicative property of $t$, it follows that 
 $$ t(j)+b_1=t(i) +b_2. $$
Since the sum $T_n+B$ is direct, this means that ($b_1=b_2$ and) $t(i)=t(j)$, thus $i=j$, and the coloring is indeed $n$-satisfactory; see figure \ref{figure:A3tiling}.

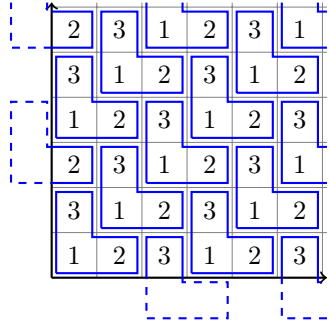
\begin{figure}[ht]
\begin{tikzpicture}[scale=.6]
\draw[step=1cm,gray,very thin] (0,0) grid (6.1,6.1);
\draw[thick,->] (0,0) -- (6.1,0);
\draw[thick,->] (0,0) -- (0,6.1);

    \setcounter{row}{1}
    \setrow {2}{3}{1}{2}{3}{1}  
    \setrow {3}{1}{2}{3}{1}{2}  
    \setrow {1}{2}{3}{1}{2}{3}
    \setrow {2}{3}{1}{2}{3}{1}  
    \setrow {3}{1}{2}{3}{1}{2}  
    \setrow {1}{2}{3}{1}{2}{3}
        
\draw[thick,blue] (0.1,0.1) -- (1.9,0.1) -- (1.9,0.9) -- (0.9,0.9) -- (0.9,1.9) -- (0.1,1.9) -- (0.1,0.1);
\draw[thick,blue] (3.1,0.1) -- (4.9,0.1) -- (4.9,0.9) -- (3.9,0.9) -- (3.9,1.9) -- (3.1,1.9) -- (3.1,0.1);
\draw[thick,blue] (1.1,1.1) -- (2.9,1.1) -- (2.9,1.9) -- (1.9,1.9) -- (1.9,2.9) -- (1.1,2.9) -- (1.1,1.1);
\draw[thick,blue] (4.1,1.1) -- (5.9,1.1) -- (5.9,1.9) -- (4.9,1.9) -- (4.9,2.9) -- (4.1,2.9) -- (4.1,1.1);
\draw[thick,blue] (2.1,2.1) -- (3.9,2.1) -- (3.9,2.9) -- (2.9,2.9) -- (2.9,3.9) -- (2.1,3.9) -- (2.1,2.1);
\draw[thick,blue] (0.1,3.1) -- (1.9,3.1) -- (1.9,3.9) -- (0.9,3.9) -- (0.9,4.9) -- (0.1,4.9) -- (0.1,3.1);
\draw[thick,blue] (3.1,3.1) -- (4.9,3.1) -- (4.9,3.9) -- (3.9,3.9) -- (3.9,4.9) -- (3.1,4.9) -- (3.1,3.1);
\draw[thick,blue] (1.1,4.1) -- (2.9,4.1) -- (2.9,4.9) -- (1.9,4.9) -- (1.9,5.9) -- (1.1,5.9) -- (1.1,4.1);
\draw[thick,blue] (4.1,4.1) -- (5.9,4.1) -- (5.9,4.9) -- (4.9,4.9) -- (4.9,5.9) -- (4.1,5.9) -- (4.1,4.1);
\draw[thick,blue] (2.1,6.1) -- (2.1,5.1) -- (3.9,5.1) -- (3.9,5.9) -- (2.9,5.9) -- (2.9,6.1);
\draw[thick,blue] (6.1,2.1) -- (5.1,2.1) -- (5.1,3.9) -- (5.9,3.9) -- (5.9,2.9) -- (6.1,2.9);
\draw[thick,blue] (5.1,6.1) -- (5.1,5.1) -- (6.1,5.1);
\draw[thick,blue] (5.9,6.1) -- (5.9,5.9) -- (6.1,5.9);
\draw[thick,blue] (0,5.1) -- (0.9,5.1) -- (0.9,5.9) -- (0,5.9);
\draw[thick,dashed,blue] (0,5.1) -- (-0.9,5.1) -- (-0.9,6.1);
\draw[thick,dashed,blue] (-0.1,6.1) -- (-0.1,5.9) -- (0,5.9);
\draw[thick,blue] (0,2.1) -- (0.9,2.1) -- (0.9,2.9) -- (0,2.9);
\draw[thick,dashed,blue] (0,2.9) -- (-0.1,2.9) -- (-0.1,3.9) -- (-0.9,3.9) -- (-0.9,2.1) -- (0,2.1);
\draw[thick,blue] (2.1,0) -- (2.1,0.9) -- (2.9,0.9) -- (2.9,0);
\draw[thick,dashed,blue] (2.9,0) -- (2.9,-0.1) -- (3.9,-0.1) -- (3.9,-0.9) -- (2.1,-0.9) -- (2.1,0);
\draw[thick,blue] (5.1,0) -- (5.1,0.9) -- (5.9,0.9) -- (5.9,0);
\draw[thick,dashed,blue] (5.9,0) -- (5.9,-0.1) -- (6.1,-0.1);
\draw[thick,dashed,blue] (6.1,-0.9) -- (5.1,-0.9) -- (5.1,0);

\end{tikzpicture}
\caption{Tiling of a superset of $\mathbb O_3$ by $T_3$, and the 3-satisfactory coloring it induces.}
\label{figure:A3tiling}
\end{figure}

\vspace{1mm}
To see that, conversely, (i) implies (ii), consider an $n$-satisfactory coloring $c$. Letting $B'$ be the image under the map $t$ of one of the color classes, note that the sum 
$T_n+B'$ is direct. Indeed, suppose that $i,j\in [n]$, and $k,k'\in K_n$ are such that $t(k),t(k')\in B'$ and $t(i)+t(k)=t(j)+t(k')$, that is, $ik=jk'$. By removing common factors 
if necessary, we may further assume that $i,j$ are relatively prime. This means that there is a positive integer $k''$ such that $k=jk''$ and $k'=ik''$. Observe that 
$k''\in K_n$. The assumption that $B'$ is the image of a color class gives us that $c(jk'')=c(ik'')$ and, since $c$ is $n$-satisfactory, then $i=j$ and so also $k=k'$. This 
proves that the sum $T_n+B'$ is indeed direct.

Let now $x\in\mathbb O_n$ be sufficiently far from the boundary of $\mathbb O_n$, in the sense that all of $x-t(1),\dots,x-t(n)$ are themselves in $\mathbb O_n$ 
(equivalently, if $x=t(m)$, then all of $m,m/2,\dots,m/n$ are positive integers), and fix the image $B'$ of a color class. We claim that for some $i\in [n]$, we have that $x-t(i)\in B'$. 
Otherwise, by the pigeonhole principle, for some $i\ne j$, both in $[n]$, it must be that $x-t(i)$ and $x-t(j)$ are in the same image $B''$ of a color class. This is impossible, since 
the decompositions $x=t(i)+(x-t(i))=t(j)+(x-t(j))$ contradict that the sum $T_n+B''$ is direct, as shown in the previous paragraph. 

We have shown that for any image $B'$ of a color class, the sum $T_n+B'$ is direct and contains a translate of $\mathbb O_n$, for instance $x_0+\mathbb O_n$, where 
$x_0=t(\operatorname{lcm}([n]))$. Setting $B=B'-x_0$, then $T_n+B$ is a direct sum and covers $\mathbb O_n$, as desired; see figure \ref{figure:A4tiling}.   
\end{proof}

\begin{figure}[ht]
\begin{tikzpicture}[scale=.6]
\draw[step=1cm,gray,very thin] (0,0) grid (6.1,6.1);
\draw[thick,->] (0,0) -- (6.1,0);
\draw[thick,->] (0,0) -- (0,6.1);

    \setcounter{row}{1}
    \setrow {3}{1}{2}{4}{3}{1}  
    \setrow {1}{2}{4}{3}{1}{2}  
    \setrow {2}{4}{3}{1}{2}{4}
    \setrow {4}{3}{1}{2}{4}{3}  
    \setrow {3}{1}{2}{4}{3}{1}  
    \setrow {1}{2}{4}{3}{1}{2}
        
\draw[very thick]  (2,6.1) -- (2,1) -- (6.1,1);
\draw[thick,dotted,red] (0.1,1.9) -- (0.1,0.1) -- (2.9,0.1) -- (2.9,0.9) -- (0.9,0.9) -- (0.9,1.9) -- (0.1,1.9);
\draw[thick,blue] (2,4.1) -- (2.9,4.1) -- (2.9,4.9) -- (2,4.9); 
\draw[thick,dashed,blue] (2,4.9) -- (0.9,4.9) -- (0.9,5.9) -- (0.1,5.9) -- (0.1,4.1) -- (2,4.1); 
\draw[thick,blue] (2,1.1) -- (3.9,1.1) -- (3.9,1.9) -- (2,1.9);
\draw[thick,dashed,blue] (1.9,1.9) -- (1.9,2.9) -- (1.1,2.9) -- (1.1,1.1) -- (2,1.1);
\draw[thick,blue] (2.1,3.9) -- (2.1,2.1) -- (4.9,2.1) -- (4.9,2.9) -- (2.9,2.9) -- (2.9,3.9) -- (2.1,3.9);
\draw[thick,blue] (3.1,4.9) -- (3.1,3.1) -- (5.9,3.1) -- (5.9,3.9) -- (3.9,3.9) -- (3.9,4.9) -- (3.1,4.9);
\draw[thick,blue] (6.1,4.9) -- (4.9,4.9) -- (4.9,5.9) -- (4.1,5.9) -- (4.1,4.1) -- (6.1,4.1); 
\draw[thick,blue] (6.1,5.9) -- (5.9,5.9) -- (5.9,6.1); 
\draw[thick,blue] (5.1,6.1) -- (5.1,5.1) -- (6.1,5.1); 
\draw[thick,blue] (4.9,1) -- (4.9,1.9) -- (4.1,1.9) -- (4.1,1);
\draw[thick,dashed,blue] (6.1,0.9) -- (4.9,0.9) -- (4.9,1);
 \draw[thick,dashed,blue] (4.1,1) -- (4.1,0.1) -- (6.1,0.1);
\draw[thick,blue] (6.1,1.9) -- (5.9,1.9) -- (5.9,2.9) -- (5.1,2.9) -- (5.1,1.1) -- (6.1,1.1);
\draw[thick,blue] (2,5.1) -- (3.9,5.1) -- (3.9,5.9) -- (2,5.9);
\draw[thick,dashed,blue] (1.1,6.1) -- (1.1,5.1) -- (2,5.1); 
\draw[thick,dashed,blue] (2,5.9) -- (1.9,5.9) -- (1.9,6.1);

\end{tikzpicture}
\caption{Tiling by $T_4$ of a superset of a translate of $\mathbb O_4$ induced by a 4-satisfactory coloring.}
\label{figure:A4tiling}
\end{figure}
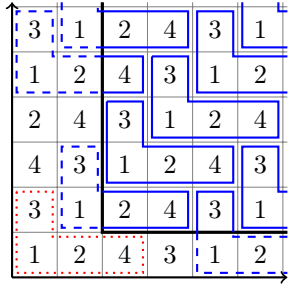

Now consider the following additional statement:

\begin{enumerate}
\item[(iii)]
$\mathbb Z^{\pi(n)}$ can be tiled by $T_n$.
\end{enumerate}

Obviously, (iii) implies (ii) (and therefore (i)), and it is natural to ask whether the converse holds. We argue below that this is indeed the case. Note that the proof of proposition 
\ref{proposition:tiling} shows that (iii) is equivalent to the following statement  (cf. remark \ref{rmk:quotient}):

\begin{enumerate}
\item[(iv)]
There is an $n$-satisfactory coloring of $\hat K_n=\{a/b:a,b\in K_n\}$.
\end{enumerate}

\begin{proposition} \label{prop:tilingz}
With notation as above, (i) implies (iii), and therefore (i)--(iv) are all equivalent.
\end{proposition}

\begin{proof}
Let $\hat{\mathcal G}_n$ be the graph with set of vertices $\hat K_n$ where two points $x,y$ are connected if and only if there is an $m\in\hat K_n$ such that $x,y\in\{im:i\in[n]\}$ 
(cf. \S\,\ref{subsec:graham}). It is enough to argue that $\chi(\hat{\mathcal G}_n)=n$, since this is equivalent to (iv). But this is a consequence of compactness (in the form of the 
de Bruijn--Erd\H os theorem \cite{deBruijnErdos51}): given any finite subgraph $G$ of $\hat{\mathcal G}_n$, by multiplying all vertices by an appropriate $k\in K_n$ we see that 
$G$ is isomorphic to a finite subgraph of $\mathcal G_n$ and is therefore $n$-colorable, since (i) is equivalent to the assertion that $\chi(\mathcal G_n)=n$.
\end{proof}

Essentially the same argument was also noted in \cite{Boseketal18}, where $\hat{\mathcal G}_n$ is denoted $W_n$. Incorporating into the argument the proof of the de 
Bruijn--Erd\H os theorem in the countable case reveals a subtlety worth pointing out, as it leads to the interesting question \ref{q:tilingz} below. To help see the connection, we 
rephrase the proof just given using directly the integer grid rather than the accompanying graph.

For each positive integer $m$, let $D_m$ be the hypercube
 $$ D_m=\{(a_1,\dots,a_{\pi(n)}): |a_i|\le m\mbox{ for all }i\}. $$
Each $D_m$ admits a coloring $d^m$ that is ``partially $n$-satisfactory'' in the sense that any copy of the polyomino $T_n$ completely contained in $D_m$ receives $n$ colors.
Namely, consider an $n$-satisfactory coloring of $K_n$, seen as a coloring of the orthant $\mathbb O_n$, and a cube $D_m'$ of the same size as $D_m$ but completely contained 
in $\mathbb O_n$. Now define $d^m$ simply by translating $D_m'$ onto $D_m$ and copying the given coloring. Naturally, the colorings $d^m$ are not compatible in general. To 
obtain an actual $n$-satisfactory coloring of $\hat K_n$, that is, a coloring of the whole integer grid where any copy of $T_n$ receives $n$ colors, we need an additional step, 
which amounts to a standard application of K\H onig's infinity lemma. 

Explicitly: enumerate the points in the grid as $v_1,v_2,\dots$. Note that if $m<m'$, then $D_m\subsetneq D_{m'}$, and that the union over all $m$ of the hypercubes $D_m$ is the 
whole space $\mathbb Z^{\pi(n)}$ so that, for any $k$, $v_k$ is in $D_m$ for all sufficiently large $m$. Consider the sequence of colorings $\vec d=(d^m)_{m>0}$. Since only 
$n$ colors are possible, there is a subsequence of $\vec d$ that always assigns to $v_1$ the same color. Passing to a subsubsequence, we can also fix the color assigned to 
$v_2$. Going to yet a further subsequence, we can fix the color of $v_3$. Recursively carrying this procedure out produces a ``limit'' $n$-coloring $d$ of the whole grid that is in 
addition satisfactory, since for any two points $x,y$ with $x=t(ik)$, $y=t(jk)$ for $i\ne j$ in $[n]$ and some $k\in\hat K_n$, for $m$ large enough (say, $m\ge m_0$) the tile 
$T_n+t(k)$ is completely contained in $D_m$ and so all associated colorings $d^m$ assign to $x,y$ different colors. Say $x=v_r$ and $y=v_s$ with $r<s$. Since the color that 
$d$ assigns to $x$ is $d^m(x)$ for infinitely many $m$, and the color that it assigns to $y$ is $d^m(y)$ for a subsequence of these $m$ (using that 
$r<s$), in particular there is such an $m$ with $m\ge m_0$ and therefore $d(x)\ne d(y)$, as needed.

The subtlety we referred to above is that, naturally, the resulting coloring $d$ needs not be compatible with the initial coloring $c$ of the orthant; in fact, for any $k\in K_n$, $d$ 
may not be compatible with any of the translates $c_k$ (that is, with the original coloring of any of the translates $\mathbb O_n+t(k)$) and the question remains whether we can 
further impose this compatibility requirement. Say that a tiling $T_n+B'$ of $\mathbb Z^{\pi(n)}$ \emph{essentially extends} a tiling $T_n+B$ of a superset of $\mathbb O_n$ if 
and only if any tile of $T_n+B$ completely contained in $\mathbb O_n$ is a tile of $T_n+B'$ (so that, if at all, only partial tiles covering a part of the boundary of the orthant could 
in principle change). 

\begin{question} \label{q:tilingz}
Let $n\in\mathbb Z^+$.
\begin{enumerate}
\item
Does any $n$-satisfactory coloring of $K_n$ extend to one of $\hat K_n$?
\item
If $T_n$ essentially tiles $\mathbb O_n$ via a tiling $T_n+B$, is there a tiling by $T_n$ of all of $\mathbb Z^{\pi(n)}$ that essentially extends it?
\end{enumerate}
\end{question}

Question \ref{q:tilingz} has a positive answer for $n=3,4$ but seems delicate in general. It is clear that a multiplicative $n$-satisfactory coloring can be extended as in (1) 
(see  remark \ref{rmk:periodic}); in particular, (1) has a positive answer if all $n$-satisfactory colorings are multiplicative.  In that case, (2) has a positive answer as well. More 
generally, (2) has a positive answer if the coloring induced by $T_n+B$ (as in the proof of proposition \ref{proposition:tiling}) is multiplicative. Question \ref{q:equation} (whether for 
any $n$-satisfactory $c$ and any $k\in K_n$ we can find an $n$-satisfactory $d$ such that $d_k=c$) is a close relative; we briefly explore the latter in a particular case in 
\S\,\ref{subs:n6}. In terms of tilings, question \ref{q:equation} is asking whether we can extend any tiling by $T_n$ of (a superset of) $\mathbb O_n$ to one of 
$\mathbb O_n-t(k)$. If this is always possible for a given $n$, it provides us with a positive answer to question \ref{q:tilingz}.

\begin{lemma}
For any $n$, a positive answer to question \ref{q:equation} implies a positive answer to question \ref{q:tilingz}.
\end{lemma}

\begin{proof}
Let $p_1<\dots<p_{\pi(n)}$ be the primes in $[n]$. Starting with an $n$-satisfactory coloring $c=d^0$ of $K_n$, iteratively extend the corresponding coloring of the orthant ``one 
layer'' in each dimension, i.e., find $n$-satisfactory colorings $d^1,d^2,\dots$ such that $d^0={d^1}_{p_1}$ and, in general, 
 $$ d^{\pi(n)\cdot m+a-1}={d^{\pi(n)\cdot m+a}}_{p_a} $$ 
for any nonnegative $m$ and any $a\in[\pi(n)]$. If $d^j={d^{j+1}}_p$, we can think of $d^{j+1}$ as extending the domain of $d^j$ to the set 
$\operatorname{dom}(d^{j+1})=\{a/p:a\in\operatorname{dom}(d^j)\}$. The colorings $d^j$ agree on their common domains as $j$ increases, and any $m\in\hat K_n$ is eventually 
included in these domains. This means that there is a unique ``limit'' $n$-satisfactory coloring $d$ of all of $\hat K_n$ obtained by this process.
\end{proof}

\begin{remark} \label{rem:konig}
The second proof we gave of proposition \ref{prop:tilingz} illustrates an application of K\H onig's lemma that can be interpreted as a construction that the space $C_{K_n}$ is 
closed under. We promised in (4) of \S\,\ref{subsec:structure} to explain this construction here.

Consider a finite coloring $c$ of $K_n$, seen as a coloring of $\mathbb O_n$. For each $l\in\mathbb N$, let 
 $$ D_l=\{(a_1,\dots,a_{\pi(n)})\in\mathbb O_n: 0\le a_i\le l\mbox{ for all }i\in[\pi(n)]\} $$ 
be the $\pi(n)$-dimensional cube with sides of length $l$. There are only finitely many colorings $d$ of $D_l$ that are realized by $c$ in the sense that for some 
$\mathbf x\in\mathbb O_n$, the coloring $c\upharpoonright(\mathbf x+D_l)$, seen as a coloring of $D_l$ in the natural way, coincides with $d$, that is, 
$d(\mathbf y)=c(\mathbf x+\mathbf y)$ for any $\mathbf y\in D_l$.

Define an infinite finitely branching tree $\mathcal T$ as follows: start with the empty coloring of the empty set (seen as the only node of $\mathcal T$ at level $-1$) and, for each 
$l\in\mathbb N$, use as nodes of the $l^{\mathrm{th}}$ level of $\mathcal T$ the colorings $d$ realized by $c$ on infinitely many distinct copies of $D_l$ (that is, those $d$ for which 
there are infinitely many $\mathbf x$ as above). Use as immediate successors of such a coloring $d$ the colorings $d'$ at the $(l+1)^{\mathrm{st}}$ level that extend $d$ in the 
sense that $d'\upharpoonright D_l=d$.

By K\H onig's lemma, the tree admits an infinite branch, that is, a sequence of colorings $(d^l:l\in\mathbb N)$ such that $\operatorname{dom}(d^l)=D_l$ for each $l$, and the 
colorings are compatible in the sense that $d^l=d^{l+1}\upharpoonright D_l$ for each $l$. The union of these colorings is a coloring $d$ of $\mathbb O_n$ and, just as in the 
second proof of proposition \ref{prop:tilingz}, if $c$ is $n$-satisfactory, then so is $d$.

Note that we could define another finitely branching tree $\mathcal T'$ by being more generous and considering all colorings that are realized rather than only those that are 
realized infinitely often, but this version realizes as branches many colorings we already had access to by other procedures (for instance, starting with $c$, all colorings $c_k$, 
$k\in K_n$, appear as branches of $\mathcal T'$), while restricting attention to $\mathcal T$ may potentially result in different colorings. Moreover, even if, say, $c$ itself appears 
as a branch through $\mathcal T$, this now reveals something about the structure of $c$.

The combinatorial fact behind both the argument just given and the second proof of proposition \ref{prop:tilingz} is that in order to show that there are $n$-satisfactory colorings
of $K_n$ it is enough to argue that there are partially $n$-satisfactory colorings of the cubes $D_l$ for all $l$ (in the sense mentioned earlier, that any copy of the polyomino 
$T_n$ completely contained in $D_l$ receives all colors). However, we do not see at the moment a scenario allowing us to verify the latter without directly exhibiting the former.
\end{remark}

A related matter is whether the operations described in the proof of proposition \ref{proposition:tiling} are inverses of each other. We formulate this as a question about the proof of 
the equivalence between (iii) and (iv) above. 

\begin{question} \label{q:inverse}
Given an $n$-satisfactory coloring $c$ of $\hat K_n$, let $B$ be the image under $t$ of a color class of $c$. The proof of proposition \ref{proposition:tiling} shows that the sum 
$T_n+B$ is a tiling of $\mathbb Z^{\pi(n)}$. From this tiling we can define an $n$-satisfactory coloring $c'$ with color classes the preimages under $t$ of the translates $t(i)+B$, 
$i\in[n]$. Is $c'=c$?
\end{question}

The answer is positive for multiplicative colorings, see remark \ref{rmk:lattice2}. Also, note the order in which we consider the operations: if instead we start with a tiling, define a 
coloring from it, and use the coloring to derive a tiling, we simply return to the original tiling. 

Instead of (iii) and (iv) we could consider (i) and (ii). The situation here (where we only consider the orthant $\mathbb O_n$ rather than the whole $\mathbb Z^{\pi(n)}$) is somewhat 
more delicate: now from an essential tiling of $\mathbb O_n$ we get an $n$-satisfactory coloring of $K_n$ just as before, but from such a coloring $c$ we only get a tiling of a subset 
of the orthant, and it was only by translation that we got a tiling of a superset in the proof of proposition \ref{proposition:tiling}. However, the process of translation may effectively 
change even well-behaved colorings (see for instance theorem \ref{thm:nonmultiplicative}). On the other hand, $c$ gives us not just one, but $n$ partial tilings of $\mathbb O_n$, 
and any point in the orthant belongs to at least one of the resulting direct sums $T_n+B$. Any of these sums defines a partial $n$-satisfactory coloring of $K_n$.

\begin{question} \label{q:inverse2}
In the setting just described, are the resulting partial colorings compatible? If they are, their union gives us a coloring $c'$ of $K_n$. Is $c'=c$?
\end{question}

We can also ask whether the partial tilings can be extended to essential tilings in compatible ways. We address question \ref{q:inverse2} in remark \ref{rmk:n6} where we show that, 
perhaps surprisingly, there are instances where the answer is negative.

\subsection{Satisfactory colorings with $n\le 5$} \label{subs:nle5}

For $n\le 5$ it is easy to give an explicit description of all $n$-satisfactory colorings. For each $n<5$ there is exactly one such coloring, and there are precisely two for $n=5$. 
As we will see in \S\,\ref{subs:n6}, such an explicit list is no longer possible even for $n=6$. We use $\mathbb N$ for the set of natural numbers, including 0.

\begin{itemize}
\item 
$n=1$.
\end{itemize}

Trivially, there is only one 1-satisfactory coloring of $K_1=\{1\}$. 

\begin{itemize}
\item
$n=2$.
\end{itemize}

Similarly, there is only one 2-satisfactory coloring $c$ of $K_2=\{2^a:a\in\mathbb N\}$: presented as an equivalence relation with 2 classes $c(1)\ne c(2)$, we have 
\begin{equation} \label{eqn:col2}
c(2^\alpha)=c(2^{\alpha\bmod2})
\end{equation}
for all $\alpha\in\mathbb N$.

This introduces a recurring theme: we could describe the coloring simply as $c(2^\alpha)=(\alpha\bmod 2)$, that is, the coloring described this way has precisely the same 
classes as the one in equation (\ref{eqn:col2}).

Note that $2+1=3$ is prime, so $c(m)=(m\bmod 3)$ is a 2-satisfactory coloring of $K_2$, as shown in \S\,\ref{subsec:komal}. One can easily verify that (as expected) this 
coloring also coincides with the one in equation (\ref{eqn:col2}).

\begin{itemize}
\item
$n=3$.
\end{itemize}

Suppose now that $c$ is a 3-satisfactory coloring of $K_3=\{2^\alpha3^\beta:\alpha,\beta\in\mathbb N\}$. For $a$ a positive integer we have that $c(a)$, $c(2a)$ and $c(3a)$ 
are all different. 

It follows that $c(2a),c(4a),c(6a)$ are different, and so are $c(3a),c(6a),c(9a)$. In particular, $c(6a)\ne c(2a),c(3a)$, so $c(6a)=c(a)$ and therefore $c(4a)=c(3a)$, from which
we conclude that, in fact, $c(2^\alpha3^\beta)=c(2^{\alpha+2\beta})$. Also, since $c(a),c(2a),c(4a)$ are different, we have that $c(2a),c(4a),c(8a)$ are different as well, and
it follows  that $c(8a)=c(a)$. This means that 
\begin{equation} \label{eq:3} 
c(2^\alpha3^\beta)=c(2^{\alpha+2\beta\bmod3})
\end{equation}
for all $\alpha,\beta\in\mathbb N$. Conversely, it is easy to check that equation (\ref{eq:3}) together with the requirement that $c(1),c(2),c(3)$ are distinct describes a 
3-satisfactory coloring of $K_3$. 

Naturally, this is the coloring indicated in figures \ref{figure:3tiling} and \ref{figure:3tiling-b}, which in turn can be described (up to the name of the colors used) by saying that 
for nonnegative integers $\alpha,\beta$, the unit square with bottom left corner at $(\alpha,\beta)$ has color $(\alpha+2\beta\bmod 3)$, so that as before, the coloring of a 
number $m=2^\alpha3^\beta\in K_3$ is given as a linear equation in the exponents of the prime factorization of $m$.

Note that the permutation $(23)$ is an automorphism of $([3],|)$, see figure \ref{figure:Hasse3}. By lemma \ref{lem:indiscernible}, the coloring $\tilde c$ is also 3-satisfactory, 
where 
 $$ \tilde c(2^\alpha3^\beta)=\tilde c(3^{\beta+2\alpha\bmod 3}) $$
and, again, we require $\tilde c(1),\tilde c(2),\tilde c(3)$ to be different. Uniqueness, of course, simply means that $c$ is invariant under this permutation, i.e., $c=\tilde c$. 

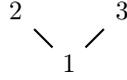
\begin{figure}[htt]
 \begin{tikzpicture}[-,>=stealth',shorten >=1pt,auto,node distance=1cm,thick, scale=.8]
  
 \node (1) {1};
 \node (2) [above left of=1] {2};
 \node (3)[above right of=1]{3};
 
\path
        (1) edge node {} (2)
        (1) edge node {} (3);
        
\end{tikzpicture}
\caption{Hasse diagram for divisibility on $[3]$. Note $(23)$ is an automorphism of this structure.}
\label{figure:Hasse3}
\end{figure}
 
\begin{itemize}
\item
$n=4$.
\end{itemize}

The argument is similar to the previous case: if $c$ is a 4-satisfactory coloring of $K_4=K_3$, then for any positive integer $a$, the colors $c(a),c(2a),c(3a),c(4a)$ are different, 
and $c(6a)\ne c(2a),c(4a),c(8a)$ and $c(6a)\ne c(3a),c(9a),c(12a)$, so $c(6a)=c(a)$ and $c(8a)=c(3a)$. 

\begin{figure}[htt]
\begin{tikzpicture}[scale=.8]
\draw[step=1cm,gray,very thin] (0,0) grid (5.1,4.1);
\draw[very thick] (0,0) -- (3,0) -- (3,1) -- (1,1) -- (1,2) -- (0,2) -- (0,0);
\draw[thick,red] (0.1,2.9) -- (0.1,1.1) -- (2.9,1.1) -- (2.9,1.8) -- (0.9,1.8) -- (0.9,2.9) -- (0.1,2.9);
\draw[thick,blue] (1.1,1.9) -- (1.1,0.1) -- (3.9,0.1) -- (3.9,0.9) -- (1.9,0.9) -- (1.9,1.9) -- (1.1,1.9);

\node[anchor=center] at (0.5, 0.5) {{\footnotesize$c(a)$}};
\node[anchor=center] at (1.5,0.5) {{\footnotesize${}\,c(2a)$}};
\node[anchor=center] at (2.5,0.5) {{\footnotesize$c(4a)$}};
\node[anchor=center] at (0.5,1.5) {{\footnotesize${}\,c(3a)$}};
\node[anchor=center] at (1.5,1.5) {{\footnotesize$\underline{c(a)}$}};
\node[anchor=center] at (3.5,0.5) {{\footnotesize$\underline{c(3a)}$}};
\node[anchor=center] at (4.5,0.5) {{\footnotesize$\underline{c(a)}$}};

\end{tikzpicture}
\caption{A 4-satisfactory coloring $c$ satisfies $c(6a)=c(a)$, $c(8a)=c(3a)$, and $c(16a)=c(a)$ for all $a\in K_4$.}
\label{figure:4coloring}
\end{figure}
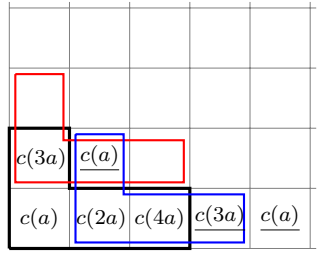

Also, since $c(a),c(2a),c(4a),c(8a)$ are different, we see that $c(16a)=c(a)$, see figure \ref{figure:4coloring}. Thus, 
\begin{equation}
\label{eq:4}
c(2^\alpha3^\beta)=c(2^{\alpha+3\beta\bmod4})
\end{equation}
for all $\alpha,\beta\in\mathbb N$. Conversely, equation (\ref{eq:4}) and the requirement that $c(1),c(2)$, $c(3),c(4)$ are distinct describes a 4-satisfactory coloring of $K_4$. 
This is the coloring indicated in figure \ref{figure:4tiling}. Again the coloring can be succinctly described by a linear equation as $c(2^\alpha3^\beta)=(\alpha+3\beta\bmod 4)$. 

Again by uniqueness and the result of \S\,\ref{subsec:komal} (since $4+1=5$ is prime), the coloring can also be described by $c(m)=(m\bmod 5)$.

\begin{itemize}
\item
$n=5$.
\end{itemize}

Note first that if $c$ is a 5-satisfactory coloring of $K_5$, then $c(8a)\ne c(a)$ for any $a\in K_5$. Indeed, otherwise $c(10a)=c(3a)$ and $c(6a)=c(5a)$, which further forces 
$c(12a)=c(2a)$ and no color can be assigned to $20a$. This means that either $c(6a)=c(a)$ or $c(10a)=c(a)$. In particular, either $c(6)=c(1)$ or $c(10)=c(1)$.

Consider first the case where $c(6)=c(1)$, and let 
 $$ K^1=\{n\in K_5:c(6n)=c(n)\}, $$ 
so that $1\in K^1$. Suppose that $a\in K^1$, that is, $c(6a)=c(a)$. We have that $c(10a)=c(3a)$ and $c(8a)=c(5a)$, thus $c(12a)=c(2a)$ or $2a\in K^1$. It follows that 
$c(15a)=c(4a)$ and $c(9a)=c(5a)$. 

Now: we just proved that $a\in K^1$ implies $c(9a)=c(5a)$; since it also implies that $2a\in K^1$, it follows that 
$c(6\cdot 3a)=c(18a)=c(9\cdot 2a)=c(5\cdot 2a)=c(10a)=c(3a)$ and, similarly, $c(6\cdot 5a)=c(30a)=c(15\cdot 2a)=c(4\cdot 2a)=c(8a)=c(5a)$. That is, $3a,5a\in K^1$ as 
well. 

This means that $K^1=K_5$ and 
 $$ c(2^5 a)=c(8\cdot 4a)=c(5\cdot 4a)=c(10\cdot 2a)=c(6a)=c(a) $$ 
for all $a\in K_5$. Now we can proceed as in the previous cases: note that $c(5a)=c(8a)$ and $c(3a)=c(10a)=c(5\cdot 2a)=c(8\cdot 2a)=c(16a)$, so 
\begin{equation}
\label{eq:5-1}
c(2^\alpha3^\beta5^\gamma)=c(2^{\alpha+4\beta+3\gamma\bmod5})
\end{equation}
for all $\alpha,\beta,\gamma\in\mathbb N$. Conversely, equation (\ref{eq:5-1}) and the requirement that $c(1),\dots,c(5)$ are distinct describes a 5-satisfactory coloring of 
$K_5$; moreover, this is the only such coloring with $c(6)=c(1)$.

A similar analysis shows that 
\begin{equation}
\label{eq:5-2}
c(2^\alpha3^\beta5^\gamma)=c(2^{\alpha+3\beta+4\gamma\bmod5})
\end{equation}
for all $\alpha,\beta,\gamma\in\mathbb N$, and the requirement that $c(1),\dots,c(5)$ are distinct describes the unique 5-satisfactory coloring of $K_5$ with $c(10)=c(1)$. 

Actually, the latter analysis can be avoided by noting that 3 and 5 are indiscernible in $K_5$ in the sense that the transposition $(35)$ is an automorphism of the Hasse 
diagram for divisibility in $[5]$, see figure \ref{figure:Hasse5}. By lemma \ref{lem:indiscernible}, there is a one-to-one correspondence between 5-satisfactory colorings with 
$c(2\cdot 5)=c(1)$ and those with $c(2\cdot 3)=c(1)$, so in particular there is a unique 5-satisfactory coloring with $c(10)=c(1)$, and the lemma allows us to recover the precise 
form of equation (\ref{eq:5-2}). 

\begin{figure}[htt]
 \begin{tikzpicture}[-,>=stealth',shorten >=1pt,auto,node distance=1cm,thick, scale=.8]
  
 \node (1) {1};
 \node (2) [above left of=1] {2};
 \node (3)[above of=1]{3};
 \node (4)[above of=2]{4};
 \node (5)[above right of=1]{5};
 
\path
        (1) edge node {} (2)
        (2) edge node {} (4)
        (1) edge node {} (3)
        (1) edge node {} (5);
        
\end{tikzpicture}
\caption{Hasse diagram for divisibility on $[5]$. Note $(35)$ is an automorphism of this structure.}
\label{figure:Hasse5}
\end{figure}
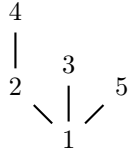

In terms of linear equations, the two colorings we obtained are 
$$ c^1(2^\alpha3^\beta5^\gamma)=(\alpha+4\beta+3\gamma\bmod5) \quad \mbox{ and }\quad c^5(2^\alpha3^\beta5^\gamma)=(\alpha+3\beta+4\gamma\bmod5), $$ 
where the superindices in $c^1,c^5$ refer to whether $c(6)=c(1)$ or $c(6)=c(5)$, respectively.

(Note that typically the same coloring can be described by several linear equations. For example, the 5-coloring that to $2^\alpha3^\beta5^\gamma$ assigns 
$(2\alpha+\beta+3\gamma\bmod 5)$ coincides with $c^1$.)

Finally, we can now illustrate why lemma \ref{lem:indiscernible} cannot be strengthened by allowing $\rho$ to be any permutation of the set of primes in $[n]$. Indeed,
consider the transposition $\rho=(23)$, extend it to a permutation of $K_5$ as in equation (\ref{eq:rho}), and note that the coloring $\tilde{c^1}$, given by 
$\tilde{c^1}(a)=\tilde{c^1}(b)$ if and only if $c^1(\rho(a))=c^1(\rho(b))$ is \emph{not} a 5-satisfactory coloring, since $\tilde{c^1}(4)=c^1(9)=c^1(5)=\tilde{c^1}(5)$. 

\subsection{A table of linear equations} \label{subsec:table}

We close the section by providing in table \ref{table:linear} a nonexhaustive list of linear equations verifying a positive solution to question \ref{q:problem} for $n\le 31$: given 
such an $n$, let $k=\pi(n)$ and let $p_1<\dots<p_k$ be the prime numbers less than or equal to $n$. We exhibit coefficients $a_1,\dots,a_k$ such that the $n$-coloring $c$ of 
$K_n$ given by 
 $$ c\left(\prod_{i=1}^k p_i^{\alpha_i}\right)=\left(\sum_{i=1}^k a_i \alpha_i\bmod n\right) $$ 
is $n$-satisfactory. In particular, note that the entry for $n=7$ provides a positive solution to K\"oMaL problem B.4265.

We identify these coefficients by a naive greedy algorithm, where for each $i$ we choose $a_i$ as small as possible so that no repeated colors occur among the numbers in 
$[n]$ of the form $\prod_{j\le i}p_i^{\beta_i}$. The point is, of course, that such a linear coloring $c$ is $n$-satisfactory if and only if it is injective on $[n]$. For instance, for 
$n=7$, the coloring indicated in table \ref{table:linear} is given by 
 $$ c(2^\alpha 3^\beta 5^\gamma 7^\delta)= (\alpha+3\beta+5\gamma+6\delta\bmod 7), $$
so that if $c(m)=(k\bmod 7)$, say, then 
 $$ (c(im): i\in[7])=(k,k+1,k+3,k+2,k+5,k+4,k+6)\bmod 7, $$
and $c$ is indeed 7-satisfactory. 

We revisit this approach and provide additional context through the notion of partial isomorphism in \S\,\ref{subs:partialisomorphism}.

{\footnotesize\begin{table}[ht]
\begin{tabular}{||c|c|c||} 
\hline \hline
$n$&$m\in K_n$&$c(m)$ modulo $n$\\
\hline \hline
1&1&0\\
2&$2^\alpha$&$\alpha$\\
3&$2^\alpha3^\beta$&$\alpha+2\beta$\\
4&$2^\alpha3^\beta$&$\alpha+3\beta$\\
5&$2^\alpha3^\beta5^\gamma$&$\alpha+3\beta+4\gamma$\\
6&$2^\alpha3^\beta5^\gamma$&$\alpha+3\beta+5\gamma$\\
7&$2^\alpha3^\beta5^\gamma7^\delta$&$\alpha+3\beta+5\gamma+6\delta$\\
8&$2^\alpha3^\beta5^\gamma7^\delta$&$\alpha+4\beta+6\gamma+7\delta$\\
9&$2^\alpha3^\beta5^\gamma7^\delta$&$\alpha+4\beta+6\gamma+7\delta$\\
10&$2^\alpha3^\beta5^\gamma7^\delta$&$\alpha+4\beta+6\gamma+9\delta$\\
11&$2^\alpha3^\beta5^\gamma7^\delta11^\epsilon$&$\alpha+4\beta+6\gamma+9\delta+10\epsilon$\\
12&$2^\alpha3^\beta5^\gamma7^\delta11^\epsilon$&$\alpha+4\beta+9\gamma+7\delta+11\epsilon$\\
13&$2^\alpha3^\beta5^\gamma7^\delta11^\epsilon13^\zeta$&$\alpha+4\beta+9\gamma+7\delta+11\epsilon+12\zeta$\\
14&$2^\alpha3^\beta5^\gamma7^\delta11^\epsilon13^\zeta$&$\alpha+4\beta+9\gamma+11\delta+7\epsilon+13\zeta$\\
15&$2^\alpha3^\beta5^\gamma7^\delta11^\epsilon13^\zeta$&$\alpha+4\beta+9\gamma+11\delta+7\epsilon+14\zeta$\\
16&$2^\alpha3^\beta5^\gamma7^\delta11^\epsilon13^\zeta$&$\alpha+5\beta+8\gamma+11\delta+14\epsilon+15\zeta$\\
17&$2^\alpha3^\beta5^\gamma7^\delta11^\epsilon13^\zeta17^\eta$&$\alpha+5\beta+8\gamma+11\delta+14\epsilon+15\zeta+16\eta$\\
18&$2^\alpha3^\beta5^\gamma7^\delta11^\epsilon13^\zeta17^\eta$&$\alpha+5\beta+8\gamma+14\delta+12\epsilon+16\zeta+17\eta$\\
19&$2^\alpha3^\beta5^\gamma7^\delta11^\epsilon13^\zeta17^\eta19^\theta$&$\alpha+5\beta+8\gamma+14\delta+12\epsilon+16\zeta+17\eta+18\theta$\\
20&$2^\alpha3^\beta5^\gamma7^\delta11^\epsilon13^\zeta17^\eta19^\theta$&$\alpha+5\beta+12\gamma+8\delta+15\epsilon+16\zeta+18\eta+19\theta$\\
21&$2^\alpha3^\beta5^\gamma7^\delta11^\epsilon13^\zeta17^\eta19^\theta$&$\alpha+5\beta+12\gamma+15\delta+8\epsilon+9\zeta+18\eta+19\theta$\\
22&$2^\alpha3^\beta5^\gamma7^\delta11^\epsilon13^\zeta17^\eta19^\theta$&$\alpha+5\beta+12\gamma+15\delta+8\epsilon+18\zeta+19\eta+21\theta$\\
23&$2^\alpha3^\beta5^\gamma7^\delta11^\epsilon13^\zeta17^\eta19^\theta23^\iota$&$\alpha+5\beta+12\gamma+15\delta+8\epsilon+18\zeta+19\eta+21\theta+22\iota$\\
24&$2^\alpha3^\beta5^\gamma7^\delta11^\epsilon13^\zeta17^\eta19^\theta23^\iota$&$\alpha+5\beta+12\gamma+15\delta+18\epsilon+9\zeta+21\eta+22\theta+23\iota$\\
25&$2^\alpha3^\beta5^\gamma7^\delta11^\epsilon13^\zeta17^\eta19^\theta23^\iota$&$\alpha+5\beta+12\gamma+15\delta+18\epsilon+9\zeta+21\eta+22\theta+23\iota$\\
26&$2^\alpha3^\beta5^\gamma7^\delta11^\epsilon13^\zeta17^\eta19^\theta23^\iota$&$\alpha+5\beta+12\gamma+15\delta+18\epsilon+21\zeta+9\eta+23\theta+25\iota$\\
27&$2^\alpha3^\beta5^\gamma7^\delta11^\epsilon13^\zeta17^\eta19^\theta23^\iota$&$\alpha+5\beta+12\gamma+18\delta+20\epsilon+25\zeta+9\eta+16\theta+22\iota$\\
28&$2^\alpha3^\beta5^\gamma7^\delta11^\epsilon13^\zeta17^\eta19^\theta23^\iota$&$\alpha+5\beta+12\gamma+18\delta+21\epsilon+25\zeta+9\eta+16\theta+27\iota$\\
29&$2^\alpha3^\beta5^\gamma7^\delta11^\epsilon13^\zeta17^\eta19^\theta23^\iota29^\kappa$&$
\alpha+5\beta+12\gamma+18\delta+21\epsilon+25\zeta+9\eta+16\theta+27\iota+28\kappa$\\
30&$2^\alpha3^\beta5^\gamma7^\delta11^\epsilon13^\zeta17^\eta19^\theta23^\iota29^\kappa$&$
\alpha+5\beta+12\gamma+20\delta+26\epsilon+28\zeta+9\eta+16\theta+19\iota+23\kappa$\\
31&$2^\alpha3^\beta5^\gamma7^\delta11^\epsilon13^\zeta17^\eta19^\theta23^\iota29^\kappa31^\lambda$&
$\alpha+5\beta+12\gamma+20\delta+26\epsilon+28\zeta+9\eta+16\theta+19\iota+23\kappa+30\lambda$\\
\hline \hline
\end{tabular}\vspace{3mm}
\caption{Linear equations describing an $n$-satisfactory coloring $c$ for $n\le 31$.}
\label{table:linear}
\vspace{-7mm}
\end{table}}

\section{Generalizing the approach for $p$ prime} \label{sec:pk+1}

\subsection{Strong representatives} \label{subsec:strong}

In this section, we present a condition on $n$ that, if satisfied, ensures the existence of $n$-satisfactory colorings. The construction below was first noticed by the third-named 
author in 2009. It was suggested independently in MathOverflow by Victor Protsak\footnote{See \href{https://mathoverflow.net/q/26358/}{https://mathoverflow.net/q/26358/}}. It 
has also been considered before in connection with Graham's conjecture discussed in \S\,\ref{subsec:graham}, see for instance \cite[\S\,2]{ForcadePollington90} and references 
therein.

\begin{theorem} \label{thm:representative}
If $n,k$ are positive integers such that $p=kn+1$ is prime and  $1^k,2^k,\dots,n^k$ are distinct modulo $p$, then $c(m)=({m^k}\bmod p)$, $m\in K_n$, is an $n$-satisfactory 
coloring of $K_n$.
\end{theorem}

\begin{proof}
We begin noting that there are exactly $n$ pairwise incongruent nonzero $k^{\mathrm{th}}$ power residues modulo $kn + 1$. 

For $i\ne j\in[n]$ and $a\in K_n$, we note that $c(ia)\ne c(ja)$ 
since the hypothesis implies that $a^k i^k\not\equiv a^k j^k\pmod p$.
\end{proof}

In particular, we recover the proof of problem A.506 from K\"oMaL given in the introduction since the assumption that $1,2,\dots,n$ are distinct modulo $n+1$ is trivially valid. We 
also have the following consequence (and note that there are infinitely many $n$ such that $2n+1$ is prime).

\begin{corollary} \label{cor:2n+1}
If $p=2n+1$ is prime, then $c(m)=({m^2}\bmod p)$ is an $n$-satisfactory coloring of $K_n$.
\end{corollary}

\begin{proof}
It is enough to verify that $1^2,\dots,n^2$ are pairwise incongruent modulo $p$. This is immediate since $i^2\equiv j^2\pmod p$ if and only if either $i\equiv j\pmod p$ or 
$i\equiv-j\pmod p$, but the latter is impossible if $i,j\in[n]$. 
\end{proof}

This leads us to the following definition.

\begin{definition}[Strong representatives] \label{def:strongrepresentatives}
A satisfactory $n$-coloring $c$ \emph{admits a strong representation} if and only if there exists a prime $p$ of the form $kn+1$ for some positive integer $k$ such that 
$1^k,\dots,n^k$ are pairwise distinct modulo $p$, and $c(m)=(m^k\bmod p)$ for all $m\in K_n$. In this case, we call $p$ a \emph{strong representative of order $n$} (for $c$). 
If some satisfactory $n$-coloring admits a strong representation, we also say that $n$ \emph{admits a strong representative}.
\end{definition}

Whenever it applies, theorem \ref{thm:representative} allows us to exhibit satisfactory colorings with a simple structure. However, given $n$, even if there are primes $p=kn+1$ 
as required by the theorem, identifying them is not necessarily feasible. For instance, the smallest strong representative of order $32$ is $p =5,209,690,063,553$. Table 
\ref{table:representative} lists for $n \le 33$ the smallest strong representative of order $n$. 

\begin{table}[ht]
\begin{tabular}{||c||c|c||} 
\hline \hline
$n$&$k$&$p$ \\
\hline \hline
$\mathbf 1$ & 1 & 2\\
$\mathbf 2$ & 1 & 3\\
$\mathbf 3$ & 2 & 7\\
$\mathbf 4$ & 1 & 5\\
$\mathbf 5$ & 2 & 11\\
$\mathbf 6$ & 1 & 7\\
$\mathbf 7$ & 94 & 659\\
$\mathbf 8$ & 2 & 17\\
$\mathbf 9$ & 2 & 19\\
$\mathbf{10}$ & 1 & 11\\
$\mathbf{11}$ & 2 & 23\\
$\mathbf{12}$ & 1 & 13\\
$\mathbf{13}$ & 198,364 & 2,578,733\\
$\mathbf{14}$ & 2 & 29\\
$\mathbf{15}$ & 2 & 31\\
$\mathbf{16}$ & 1 & 17\\
$\mathbf{17}$ & 2,859,480 & 48,611,161\\
$\mathbf{18}$ & 1 & 19\\
$\mathbf{19}$ & 533,410 & 10,134,791\\
$\mathbf{20}$ & 2 & 41\\
$\mathbf{21}$ & 2 & 43\\
$\mathbf{22}$ & 1 & 23\\
$\mathbf{23}$ & 2 & 47\\
$\mathbf{24}$ & 56,610,508 & 1,358,652,193\\
$\mathbf{25}$ & 1,170,546,910 & 29,263,672,751\\
$\mathbf{26}$ & 2 & 53\\
$\mathbf{27}$ & 6,700,156,678 & 180,904,230,307\\ 
$\mathbf{28}$ & 1 & 29\\
$\mathbf{29}$ & 2 & 59\\
$\mathbf{30}$ & 1 & 31\\
$\mathbf{31}$ & 27,184,496,610 & 842,719,394,911\\
$\mathbf{32}$ & 162,802,814,486 & 5,209,690,063,553\\ 
$\mathbf{33}$ & 2 & 67\\ 
\hline \hline
\end{tabular}\vspace{3mm}
\caption{Smallest strong representative $p = kn + 1$ of order $n$ for $n \le 33$.}
\label{table:representative}
\vspace{-7mm}
\end{table}

As table \ref{table:representative} suggests, unlike the cases $k=1,2$, the primality of $kn+1$ for $k>2$ does not automatically ensure that the hypothesis of theorem 
\ref{thm:representative} is satisfied. For example, if $n=3$, then $p=4n+1=13$ is prime. However, $2^4=16$, $3^4=81$, and $16\equiv 81\pmod{13}$. This is further discussed 
in \S\,\ref{subsec:krepresentatives}.

\begin{remark} \label{rmk:variety}
In terms of notions introduced below, all colorings obtained through strong representatives are multiplicative, and in fact are $\mathbb Z/n\mathbb Z$-colorings. However, there are 
satisfactory colorings that are nonmultiplicative (see theorem \ref{thm:n6}), multiplicative colorings that are not $\mathbb Z/n\mathbb Z$-colorings (see table \ref{table:3751} for 
$a=1$), and $\mathbb Z/n\mathbb Z$-colorings that do not admit a strong representative (see table \ref{table:z81537}). 
\end{remark}

\subsection{$k$-representatives} \label {subsec:krepresentatives}

\begin{definition} \label{def:krepresentative}
Let $k\in\mathbb Z^+$. A prime $p$ of the form $kn + 1$ is a \emph{$k$-representative} if and only if $p$ is a strong representative of order $n$, that is, the numbers 
$1^k,\dots,n^k$ are distinct modulo $p$.
\end{definition}

Note that, in general, the roles of $k$ and $n$ cannot be interchanged. If $p=k+1$ is prime, it is trivially a $k$-representative. Our goal in this subsection is to show that for 
every $k > 2$ there are only finitely many $n$ such that $p = kn + 1$ is a $k$-representative. In fact, we will show that for some values of $k$ there are no such $n$.

We begin by discussing the case $k = 3$; this case was also the subject of K\"oMaL problem B.4401 in November 
2011,\footnote{See \href{https://www.komal.hu/feladat?a=honap&h=201111&t=mat&l=en}{https://www.komal.hu/feladat?a=honap\&h=201111\&t=mat\&l=en}.} proposed by the 
third-named author.

\begin{theorem} \label{thm:3-representative}
If $p=3n+1$ is prime, then $p$ is not a 3-representative. 

In particular, if $n>2$, then there is an $i\in[n]$, $i>2$, such that $i^3\equiv 1\pmod p$ or $i^3\equiv 8\pmod p$.
\end{theorem}

\begin{proof}
For $n=2$ we have that $1^3\equiv 2^3\pmod 7$. Suppose now that $n>2$ and $p=3n+1$ is prime. Work in $\mathbb Z/p\mathbb Z$. Note that $x^3 =1$ and $x\ne 1$ if and 
only if $x^2+x+1=0$ if and only if $4x^2+4x+4=0$, or $(2x+1)^2 =-3$. Also, $x^3 =8$ and $x\ne 2$ if and only if $x^2 +2x+4=0$, or $(x+1)^2 =-3$.

We claim that at least one of these two situations must happen for some $x\in[n]$. Note first that $-3$ is a quadratic residue modulo $p$:
 $$ \left(\frac{-3}p\right)=\left(\frac{-1}p\right)\left(\frac3p\right)=(-1)^{\frac{p-1}2}(-1)^{\frac{p-1}2\frac{3-1}2}\left(\frac p3\right)=\left(\frac{3n+1}3\right)=\left(\frac13\right)=1, $$ 
where $\left(\frac qp\right)$ denotes the Legendre symbol. 

It follows that the equation $y^2 = -3$ has two solutions, one in the first half of the interval $[1, p - 1]$. If $y$ is actually in the first third, we are done, we get $x = y - 1 \in [n]$. 
Suppose otherwise. Note that either $y$ or $p-y$ is odd. Call it $z$, and note that $z\le 2p/3$, and therefore $x=(z-1)/2$ is at most $(p-1)/3$, so it is in $[n]$.
\end{proof}

The case when $k$ is a multiple of 4 can also be treated by elementary means. The key is Fermat's result that an odd prime $p$ is a sum of two squares if and only if 
$p\equiv1\pmod4$.

\begin{theorem} \label{thm:4m-representative}
If $k$ is a multiple of 4 and $p = kn+1$ is a $k$-representative, then $p <k^2$, so in particular, there are only finitely many $k$-representatives.
\end{theorem}

\begin{proof}
Suppose $p = kn + 1$ is a $k$-representative. By Fermat's result, there are integers $x$ and $y$ with $1\le x<y$ such that $p=x^2 +y^2$. Note that if $p\ge k^2$, then 
$p^2/k^2=p\cdot \frac p{k^2}\ge p>y^2$, so $x<y\le p/k=n+1/k$ and therefore in fact $x<y\in[n]$, but $x^2\equiv-y^2 \pmod p$, so $x^k\equiv y^k \pmod p$.
\end{proof}

The bound on $p$ found in the theorem allows us to identify by a quick exhaustive search all the possible values of $p$ that are $k$-representatives, for any given value of $k$ 
that is a multiple of 4. Table \ref{table:4mrepresentatives} lists these values $p=4mn+1$ for all $k=4m\le 100$.

\begin{table}[ht]
\begin{tabular}{||c||c|c||} 
\hline \hline
$k=4m$&$n$&$p=kn+1$\\
\hline \hline
4 & 1 & 5\\
8 & none & none\\
12 & 1 and 3 & 13 and 37\\
16 & 1 & 17\\
20 & none & none\\
24 & none & none\\
28 & 1 & 29\\
32 & none & none\\
36 & 1 & 37\\
40 & 1 & 41\\
44 & none & none\\
48 & none & none\\
52 & 1 & 53\\
56 & none & none\\
60 & 1 and 3 & 61 and 181\\
64 & none & none\\
68 & none & none\\
72 & 1 & 73\\
76 & none & none\\
80 & 3 & 241\\
84 & 5 & 421\\
88 & 1 & 89\\
92 & none & none\\
96 & 1 & 97\\
100 & 1 & 101\\
\hline \hline
\end{tabular}\vspace{3mm}
\caption{$4m$-representatives for $m\le 25$.}
\label{table:4mrepresentatives}
\vspace{-7mm}
\end{table}

We now proceed to the general case. The key observation is that if $p$ is prime and $G\le(\mathbb Z/p\mathbb Z)^*$ is nontrivial, then $\sum_{g\in G}g=0$. Indeed, let 
$S=\sum_{g\in G}g$ and let $h\in G$ be different from the identity. The map $g\mapsto hg$ is a permutation of $G$, and we have $S=\sum_{g\in G}hg=hS$. 

Suppose now that $(n,k)\ne(1,1)$ and $p=kn+1$ is prime. The observation, applied to the case where $G$ is the group of $k^{\mathrm{th}}$ powers of nonzero elements of 
$(\mathbb Z/p\mathbb Z)^*$, gives us that if $p$ is a $k$-representative, then $\sum_{i=1}^n i^k\equiv 0\pmod p$. But this sum is a polynomial $P_k(x)$ (of degree $k+1$) with 
rational coefficients evaluated at $x=n$. \emph{If} $kx+1$ is not a factor of $P_k(x)$, then, by applying the division algorithm and clearing out denominators, there are integers 
$a,b,c$ with $c\ne 0$, such that $aP_k(x)$ has integer coefficients and 
 $$ aP_k(x)+b(kx+1)=c. $$ 
For $x=n$ we have that $p=kn+1$ divides $P_k(n)$ (by the observation) and therefore $p\mid c$, and there are only finitely many possibilities for $p$, all of which lie among the 
prime factors of $c$. The only remaining issue is how to prove that indeed $kx+1$ is not a factor of $P_k(x)$. 

We circumvent this obstacle by arguing instead that also $\sum_{i=1}^n i^{2k}\equiv0\pmod p$, and the polynomial $P_{2k}(n)$ is now $B_{2k+1}(n+1)$ for $B_{2k+1}(x)$ the 
$(2k+1)^{\mathrm{st}}$ Bernoulli polynomial, for which all its rational roots are known, and we can proceed as above.

\begin{theorem} \label{thm:k-representative}
If $k > 2$, then only finitely many primes are $k$-representatives.
\end{theorem}

The argument we have been outlining was suggested by Darij Grinberg and Gergely 
Harcos\footnote{See \href{https://mathoverflow.net/q/78270/}{https://mathoverflow.net/q/78270/}}.

Note also that for $k=1,2$ we have that 
 $$ \sum_{i=1}^n i=\frac{n(n+1)}2\quad\mbox{ and }\quad\sum_{i=1}^n i^2=\frac{n(n+1)(2n+1)}6, $$ 
so the argument above fails (as it should) since $kn+1$ is in both cases a factor of the corresponding polynomial.

\begin{proof}
Let $B(t,x)=\frac{te^{tx}}{e^t-1}$. The Bernoulli polynomials $B_m(x)$ are defined as follows using the power series expansion in terms of $t$ of $B(t,x)$:
 $$ B(t,x)=\sum_{m=0}^\infty B_m(x)\frac{t^m}{m!}. $$
It is well known that each $B_m(x)$ is a polynomial in $x$ of degree $m$ with rational coefficients, and 
 $$ \sum_{i=1}^n i^m=\frac{B_{m+1}(n+1)-B_{m+1}(0)}{m+1} $$
for all positive integers $n$, see for instance \cite[chapter 4]{Washington97}.

Writing
 $$ B_m(x)=\sum_{k=0}^m\binom m{m-k} b_k x^{m-k}, $$ 
the numbers $b_k = B_k(0)$ are usually called the \emph{Bernoulli numbers}; they satisfy $b_{2k+1} = 0$ for all $k \ge 1$. As indicated above, it will be important for us to 
know all the rational linear factors of the polynomial $B_m(x)-B_m(0)$; when $m$ is odd this reduces to determining the rational linear factors of $B_m(x)$. The following 
result of K. Inkeri \cite[theorem 3]{Inkeri59} solves this problem.

\begin{theorem}[Inkeri] \label{thm:inkeri}
The rational roots of a Bernoulli polynomial $B_m(x)$ can be only 0, $1/2$, and 1. Moreover, all these are roots when $m > 1$ is odd.
\end{theorem}

Suppose $p = kn + 1$ is a $k$-representative. We claim that
 $$ 1^{2k} + 2^{2k} + \cdots + n^{2k} \equiv 0\pmod p. $$
To see this, notice that there are precisely $\frac{p-1}d = n/\gcd(2,n)$ incongruent $(2k)^{\mathrm{th}}$ power residues modulo $p$, where $d = \gcd(2k, p - 1) = k\gcd(2, n)$. If 
$n$ is odd, this is precisely $n$, which means that the numbers $1^{2k},\dots,n^{2k}$ are all distinct and are precisely all the nonzero $(2k)^{\mathrm{th}}$ powers. If $n$ is even, 
this means that each nonzero $(2k)^{\mathrm{th}}$ power appears exactly twice among these numbers. In either case, it follows that the sum is zero by the same argument as 
above.

Since 
 $$ \sum_{i=0}^n i^{2k}=\frac{B_{2k+1}(n+1)}{2k+1}, $$
it must be the case that $(kn + 1)\mid B_{2k+1}(n + 1)$. By Inkeri's theorem \ref{thm:inkeri}, since $k > 2$, the polynomial $kx + 1$ is relatively prime to the polynomial 
$B_{2k+1}(x + 1)$. Thus, there must be polynomials $u, v \in\mathbb Q[x]$ such that
 $$ (kx + 1) \cdot u(x) + B_{2k+1}(x + 1) \cdot v(x) = 1. $$
(In fact, $v$ is a constant.)

Multiplying this identity by an appropriate integer constant $L = L_1L_2$, it follows that there are polynomials $\check u= Lu$, $\check B_{2k+1}= L_1B_{2k+1}$, and 
$\check v = L_2v$, all in $\mathbb Z[x]$ such that 
 $$ ( k x + 1 ) \cdot \check u  ( x ) + \check B_{2k + 1} ( x + 1 ) \cdot\check v ( x ) = L . $$
Since $B_{2k+1}(n + 1) \equiv 0 \pmod{kn + 1}$, evaluating the last displayed equation at $x = n$ gives us that $p = kn + 1 \mid L$. But there are only finitely many such $p$.
\end{proof}

Note that this argument does not supersede theorems \ref{thm:3-representative} or \ref{thm:4m-representative}. For theorem \ref{thm:4m-representative} in particular, note that 
the bound obtained there is in general much smaller than the bound $L$ found in the proof of theorem \ref{thm:k-representative}, which depends on the size of the denominator 
of $B_{2k+1}(x + 1)$. 

\subsection{Examples}  \label{subsec:k-representatives}

Let us illustrate theorem \ref{thm:k-representative} with some examples, for which it suffices to consider $\sum_{i=1}^n i^k$ rather than the sum of $(2k)^{\mathrm{th}}$ powers.

\begin{itemize}
\item 
$k=3$.
\end{itemize}

Recall that 
 $$ \sum_{i=1}^n i^3=\frac{n^2(n+1)^2}4. $$ 
Clearly, if $3n+1$ is prime, it does not divide $n^2(n+1)^2$, and it follows that no prime is a 3-representative. This provides another solution to K\"oMaL problem B.4401.

\begin{itemize}
\item
$k=4$.
\end{itemize}

We have that 
 $$ \sum_{i=1}^n i^4=\frac{ n(n+1)(2n+1)(3n^2 +3n-1)}{30}. $$
If $4n + 1$ is a 4-representative, then it must divide $3n^2 + 3n - 1$. But 
 $$ 16(3n^2 +3n-1)=(9+12n)(4n+1)-25, $$
so $4n + 1$ must divide $25$. Hence $n = 1$ and $p = 5$ is the only 4-representative.

\begin{itemize}
\item
$k=5$.
\end{itemize}

We have that 
 $$ \sum_{i=1}^n i^5=\frac{n^2(n+1)^2(2n^2+2n-1)}{12}. $$
If $5n+1$ is a 5-representative, then it must divide $2n^2+2n-1$. But 
 $$ 25(2n^2 +2n-1)=(10n+8)(5n+1)-33, $$
so $5n+1$ must divide $33$. Hence $n=2$. Since $2^5 =32\equiv-1\not\equiv1 \pmod{11}$, it follows that $p = 11$ is indeed the only 5-representative.

\begin{itemize}
\item
$k=6$.
\end{itemize}

We have that 
 $$ \sum_{i=1}^n i^6 = \frac{n(n+1)(2n+1)(3n^4 +6n^3 -3n+1)}{42}. $$
If $6n+1$ is a 6-representative, then it must divide $3n^4+6n^3-3n+1$. But 
 $$ 432(3n^4 +6n^3 -3n+1)=(216n^3 +396n^2 -66n-205)(6n+1)+637, $$
so $6n+1$ must divide $637=7^2\cdot 13$, and $n=1$ or $n=2$. 

Since $2^6 =64\equiv-1\not\equiv 1 \pmod{13}$, it follows that $p = 7$ and $p = 13$ are the only 6-representatives.

\begin{itemize}
\item
$k=7$.
\end{itemize}

We have 
 $$ \sum_{i=1}^n i^7 = \frac{n^2(n+1)^2(3n^4 +6n^3 -n^2 -4n+2)}{24}. $$
If $7n + 1$ is a 7-representative, then it must divide $3n^4 + 6n^3 - n^2 - 4n + 2$. But
 $$ 2401(3n^4 +6n^3 -n^2 -4n+2)=(1029n^3 +1911n^2 -616n-1284)(7n+1)+6086, $$
so $7n + 1$ must divide $6086 = 2 \cdot 17 \cdot 179$. However, since none of these prime factors is congruent to 1 modulo 7, it follows that there are no 7-representatives.

\begin{itemize}
\item
$k=8$.
\end{itemize}

We have that
 $$ \sum_{i=1}^n i^8 = \frac{n(n+1)(2n+1)(5n^6 +15n^5 +5n^4 -15n^3 -n^2 +9n-3)}{90}. $$
If $8n+1$ is an 8-representative, it must divide $5n^6 +15n^5 +5n^4 -15n^3 -n^2 +9n-3$. 
But
\begin{multline*}
262144(5n^6 +15n^5 +5n^4 -15n^3 -n^2 +9n-3)=\\
(163840n^5 + 471040n^4 + 104960n^3 - 504640n^2 + 30312n+291123)(8n + 1) - 1077555, 
\end{multline*}
so $8n + 1$ must divide $1077555 = 3 \cdot 5 \cdot 71837$.
However, since none of these prime factors is congruent to 1 modulo 8, it follows that there are no 8-representatives.

Although it ends up not making a significant difference here, using theorem 5.21, we only had to consider primes not exceeding 64.

\begin{itemize}
\item
$k=9$.
\end{itemize}

We have that
 $$ \sum_{i=1}^n i^9 = \frac{n^2(n+1)^2(n^2 +n-1)(2n^4 +4n^3 -n^2 -3n+3)}{20}. $$
If $9n+1$ is a 9-representative, then it must divide $n^2 +n-1$ or $2n^4 +4n^3 -n^2 -3n+3$.
The first case is impossible since
 $$ 81(n^2+n-1)=(9n+8)(9n+1)-89, $$ 
so  $9n + 1$ would have to divide $89\not\equiv 1 \pmod 9$. Now, since
 $$ 6561(2n^4 +4n^3 -n^2 -3n+3)=(1458n^3 +2754n^2 -1035n-2072)(9n+1)+21755, $$ 
then in the second case $9n+1$ must divide $21755=5\cdot 19\cdot 229$, so the only possibility is $n=2$. Since
$2^9 =512\equiv-1\not\equiv 1\pmod{19}$,
it follows that $p = 19$ is indeed the only 9-representative. 

\begin{itemize}
\item 
$k=10$.
\end{itemize}

We have that 
 $$ \sum_{i=1}^n i^{10}=\frac{n(n+1)(2n+1)(n^2 +n-1)(3n^6 +9n^5 +2n^4 -11n^3 +3n^2 +10n-5)}{66}. $$
If $10n + 1$ is a $10$-representative, it must divide one of the two factors $n^2 +n-1$ or $3n^6 +9n^5 +2n^4 -11n^3 +3n^2 +10n-5$.
But
 $$ 100(n^2 +n-1)=(10n+9)(10n+1)-109 $$ 
and 
\begin{multline*}
10^6(3n^6 +9n^5 +2n^4 -11n^3 +3n^2 +10n-5)=\\
(3\cdot 10^4n^4+6\cdot 10^4 n^3-42700 n^2-72700 n+106543)(10n+9)(10n + 1) - 5958887, 
\end{multline*}
so $10n+1$ must divide $109$ or $5958887 = 11^5 \cdot 37$, so $n = 1$. It follows that $p = 11$ is the
only $10$-representative.

\subsection{Density of strong representatives} \label{subsec:density}

All satisfactory $n$-colorings with $n\le 5$ admit strong representations: first, for each $n\le 4$ there is exactly one satisfactory $n$-coloring, as shown in \S\,\ref{subs:nle5}; 
moreover, $n+1$ is prime for $n=1,2,4$, and $2n+1$ is prime for $n=3$. 

For $n=5$, there are precisely two satisfactory $n$-colorings, which we labeled $c^1$ and $c^5$ so that $c^i(6)=c^i(i)$, see \S\,\ref{subs:nle5}. Note that $2n+1=11$ is prime; 
the corresponding coloring is $c^5$ since $6^2=36\equiv 25=5^2\pmod{11}$. 

Similarly, $421=84\cdot 5+1$ is prime, and is a strong representative of order 5 for $c^1$ since 
 $$ (1^{84},2^{84},3^{84},4^{84},5^{84},6^{84})\equiv(1,279,252,377,354,1)\pmod{421}. $$  
 
Strong representatives are hardly unique. For instance, any prime is a strong representative of order 1. The case $n=2$ is more interesting.

\begin{fact} \label{fact:den-2}
A prime $p$ is a strong representative of order 2 if and only if 
 $$ p\equiv\pm3\pmod 8. $$ 
In particular, there are infinitely many such primes. 
\end{fact}

\begin{proof}
The first sentence is a restatement of the supplementary law for quadratic reciprocity: a prime $p=2k+1$ is a strong representative of order 2 if and only if 
$2^k=2^{\frac{p-1}2}\not\equiv 1\pmod p$, which is equivalent to asserting that $2$ is not a square modulo $p$. The supplementary law tells us that 
 $$ \left(\frac 2p\right)=(-1)^{\frac{p^2-1}8}, $$
which equals $-1$ if and only if $p\equiv\pm 3\pmod 8$. By the prime number theorem for arithmetic progressions, half of all primes 
are of this form in the sense of natural density:  simply note that $\phi(8)=4$, where $\phi(\cdot)$ is Euler's totient function, and that, asymptotically, $1/4$ of all primes have 
the form $8k+a$ for any given $a\in[8]$ relatively prime with 8, see \cite[chapter 22]{Davenport00}.
\end{proof}

Some natural questions occur at this point.

\begin{question} \label{qu:natdensity}
Let $n\in\mathbb Z^+$.
\begin{enumerate}
\item
If $n$ admits a strong representative $p$, does it admit infinitely many? 
\item
If the answer to item \emph{(1)} is positive, is the set of such primes of positive natural density among all primes? 
\item
Further, suppose $n$ admits a strong representative. For a satisfactory $n$-coloring $c$, is the set of strong representatives of order $n$ for $c$ of positive 
natural density, and is this density independent of $c$?
\end{enumerate}
\end{question}

We present some numerical data for $n\le 10$. Consider, for instance, $n=5$. First, none of the 17 primes of the form $5n+1$ in the interval $[12,420]$ is a strong 
representative of order 5, and neither are any of the 11 such primes in the interval $[422,700]$. However, additional strong representatives eventually appear.

\begin{example}
The prime $p = 701 = 140\cdot 5 + 1$ is a strong representative of order 5 for $c^1$. In effect,
 $$ (1^{140},2^{140},3^{140},4^{140},5^{140},6^{140})\equiv(1,210,464,638,89,1)\pmod{701}. $$ 
Similarly, one can check that $p = 2311 = 462\cdot 5 + 1$ is a strong representative of order 5 for $c^5$.
\end{example}

Given a real $x$, denote by $\mathcal C^1(x)$ and $\mathcal C^5(x)$ the sets of primes $p \le x$ that are strong representatives of order 5 for $c^1$ and $c^5$, respectively, 
and let $\mathcal C(x) = \mathcal C^1(x) \cup \mathcal C^5(x)$. Also, write $\mathcal C_T(x)$ for the set of all primes $p\le x$ of the form $5n+1$. Table \ref{table:con-5} 
provides some numerical evidence suggesting a positive answer to item (3) of question \ref{qu:natdensity} for $n=5$.

\begin{table}[ht]
\begin{tabular}{|c||c|c|c|c|c|}
\hline 
$m$&$10^6$&$2\cdot 10^6$&$3\cdot 10^6$&$4\cdot 10^6$&$5\cdot 10^6$\\
\hline\hline 
$|\mathcal C^1(m)|$ & 626 & 1203 & 1757 & 2314 & 2838\\
\hline
$|\mathcal C^5(m)|$ & 626 & 1210 & 1783 & 2291 & 2822\\
\hline
$|\mathcal C(m)|$ & 1252 & 2413 & 3540 & 4605 & 5660\\
\hline
$|\mathcal C_T(m)|$ & 19617 & 37188 & 54175 & 70779 & 87062\\
\hline
$\frac{|\mathcal C^1(m)|}{|\mathcal C^5(m)|}$ & 1 & 0.994215 & 0.985418 & 1.010039 & 1.005670\\
\hline
$\frac{|\mathcal C(m)|}{|\mathcal C_T(m)|}$ & 0.063822 & 0.064887 & 0.065344 & 0.065062 & 0.065011\\
\hline 
\end{tabular}
\vspace{2mm}

\begin{tabular}{|c||c|c|c|c|c|}
\hline 
$m$&$6\cdot 10^6$&$7\cdot 10^6$&$8\cdot 10^6$&$9\cdot 10^6$&$10^7$\\
\hline\hline 
$|\mathcal C^1(m)|$ & 3376 & 3873 & 4386 & 4886 & 5358\\
\hline
$|\mathcal C^5(m)|$ & 3309 & 3843 & 4302 & 4772 & 5265\\
\hline
$|\mathcal C(m)|$ & 6685 & 7716 & 8688 & 9658 & 10623\\
\hline
$|\mathcal C_T(m)|$ & 103153 & 119109 & 134912 & 150604 & 166104\\
\hline
$\frac{|\mathcal C^1(m)|}{|\mathcal C^5(m)|}$ & 1.020248 & 1.007806 & 1.019526 & 1.023889 & 1.017664\\
\hline
$\frac{|\mathcal C(m)|}{|\mathcal C_T(m)|}$ & 0.064807 & 0.064781 & 0.064398 & 0.064128 & 0.063954\\
\hline 
\end{tabular}
\vspace{3mm}
\caption{Density of strong representatives of order 5.}
\label{table:con-5}
\vspace{-7mm}
\end{table}

For $x\ge0$, denote by $\pi_n(x)$ the number of strong representatives of order $n$ less than or equal to $x$. In table \ref{table:con-10} we provide data suggesting the density 
of strong representatives of order $n$ in the set of primes, for $n\le 10$.  

\begin{table}[ht]
\begin{tabular}{|c|c||c|c||c|c||c|c|}
\hline 
$N$&$\pi(N)$&$\pi_2(N)$&$\pi_2(N)/\pi(N)$&$\pi_3(N)$&$\pi_3(N)/\pi(N)$&$\pi_4(N)$&$\pi_4(N)/\pi(N)$\\
\hline\hline 
$10^2$ & 25 & 13 & 0.52& 2 & 0.08 & 1 & 0.04\\
\hline
$10^3$ & 168 & 87 & 0.51785\dots & 20 & 0.11904\dots& 10 & 0.05952\dots\\
\hline
$10^4$ & 1229 & 625& 0.50854\dots & 134 & 0.10903\dots& 82 & 0.06672\dots\\
\hline
$10^5$ & 9592 & 4808 & 0.50125\dots & 1087 & 0.11332\dots& 602 & 0.06276\dots\\
\hline
$10^6$ & 78498 & 39276 & 0.50034\dots & 8732 & 0.11123\dots& 4857 & 0.06187\dots \\
\hline
\end{tabular}
\vspace{2mm} 

\begin{tabular}{|c|c||c|c||c|c|}
\hline 
$N$&$\pi(N)$&$\pi_5(N)$&$\pi_5(N)/\pi(N)$&$\pi_6(N)$&$\pi_6(N)/\pi(N)$\\
\hline\hline 
$10^2$ & 25 & 1 & 0.04 & 2 & 0.08\\
\hline
$10^3$ & 168 & 3 & 0.01785\dots & 7 & $0.041\overline6$\\
\hline
$10^4$ & 1229 & 16 & 0.01301\dots & 19 & 0.01545\dots\\
\hline
$10^5$ & 9592 & 147 & 0.01532\dots & 203 & 0.02116\dots\\
\hline
$10^6$ & 78498 & 1252 & 0.01594\dots & 1803 & 0.02296\\
\hline
\end{tabular}
\vspace{2mm}

\begin{tabular}{|c|c||c|c||c|c|}
\hline 
$N$&$\pi(N)$&$\pi_7(N)$&$\pi_7(N)/\pi(N)$&$\pi_8(N)$&$\pi_8(N)/\pi(N)$\\
\hline\hline 
$10^2$ & 25 & 0 & 0 & 1 & 0.04 \\
\hline
$10^3$ & 168 & 1 & 0.00595\dots & 1& 0.00595\dots\\
\hline
$10^4$ & 1229 & 6& 0.00488\dots & 5 & 0.00406\dots\\
\hline
$10^5$ & 9592 & 30 & 0.00312\dots & 21 & 0.00218\dots\\
\hline
$10^6$ & 78498 & 195 & 0.00248\dots & 165 & 0.00210\dots\\
\hline
$10^7$ & 664579 & 1624 & 0.00244\dots & 1344 & 0.00202\dots\\
\hline
\end{tabular}
\vspace{2mm}

\begin{tabular}{|c|c||c|c||c|c|}
\hline 
$N$&$\pi(N)$&$\pi_9(N)$&$\pi_9(N)/\pi(N)$&$\pi_{10}(N)$&$\pi_{10}(N)/\pi(N)$\\
\hline\hline 
$10^2$ & 25 & 1 & 0.04 & 1 & 0.04 \\
\hline
$10^3$ & 168 & 1 & 0.00595\dots & 1& 0.00595\dots\\
\hline
$10^4$ & 1229 & 1 & 0.00081\dots & 2 & 0.00162\dots\\
\hline
$10^5$ & 9592 & 7 & 0.00072\dots & 5 & 0.00052\dots\\
\hline
$10^6$ & 78498 & 42 & 0.00053\dots & 31 & 0.00039\dots\\
\hline
$10^7$ & 664579 & 374 & 0.00056\dots & 281 & 0.00042\dots\\
\hline
\end{tabular}
\vspace{3mm}
\caption{Density of strong representatives of order $n$, $2\le n\le 10$.}
\label{table:con-10}
\vspace{-7mm}
\end{table}

We now mention some remarks explaining that items (1) and (2) of question \ref{qu:natdensity} admit a positive answer. First, we recall a well-known observation.

\begin{lemma} \label{lemma:allnpowers}
Suppose that $n$ and $p$ are prime. If not all numbers are $n^{\mathrm{th}}$ powers modulo $p$, then $p\equiv 1\pmod n$.
\end{lemma}

\begin{proof}
Indeed, all numbers are $n^{\mathrm{th}}$ powers modulo $n$. For $p$ of the form $nk+a$ with $1<a<n$, let $\alpha\in[n]$ be the multiplicative inverse of $1-a$ modulo $n$, and 
note that $x^{\alpha(p-1)+1}\equiv x\pmod p$ for any $x$ and that $\alpha(p-1)+1\equiv 0\pmod n$. 
\end{proof}

Many of the intricacies of the general case seem to be present already for $n=3$, so we consider this case first in some detail. 

\begin{theorem}
The set of primes that are strong representatives of order 3 has natural (asymptotic) density $1/9$ in the set of all primes.
\end{theorem}

\begin{proof}
A prime $p=3k+1$ is a strong representative of order 3 if and only if 
$$ 2^k\not\equiv 1\pmod p,\quad 3^k\not\equiv 1\pmod p,\quad \mbox{and} \quad 2^k\not\equiv 3^k\pmod p, $$ 
and this is equivalent to asserting that 2, 3, and $12=2^3\cdot 3/2$ are not cubes modulo $p$ (note that the fact that $p\equiv 1\pmod 3$ follows from the assertion that 2 is 
not a cube modulo $p$, by lemma \ref{lemma:allnpowers}). This indicates that the key technical result needed to determine whether question \ref{qu:natdensity} holds is 
Chebotar\"ev's theorem, see \cite[Theorem VIII.10]{Lang94}\footnote{Lang states the result in terms of Dirichlet density, but the same conclusion holds for natural 
density, see \cite[\S\,XV.5]{Lang94}.}. We recall the theorem and some basic facts from algebraic number theory that should allow us to apply it in the case at hand.

Denote by $\mathbb P$ the set of integral primes. Recall that a set $A\subseteq \mathbb P$ has \emph{natural density} $\delta$ in $\mathbb P$ if and only if 
$\lim_{n\to\infty}|A\cap[n]|/\pi(n)$ exists and equals $\delta$.

\begin{theorem}[Chebotar\"ev] \label{thm:chebotarev}
Let $L/k$ be a Galois extension with Galois group $G=\operatorname{Gal}(L/k)$, and let $C$ be a conjugacy class of $G$. The set of primes $\mathfrak p$ of $k$ that are 
unramified in $L$ and for which the Frobenius symbol $\sigma_{\mathfrak p}$ of $\mathfrak p$ in $G$ is in $C$ has natural density $|C|/|G|$.
\end{theorem}

For instance, in the example under consideration, that $2,3,12$ are not cubes modulo $p$ means that 
 $$ f(x)\coloneqq(x^3-2)(x^3-3)(x^3-12) $$ 
has no roots modulo $p$. This suggests to consider $L=\mathbb Q(\sqrt[3]2,\sqrt[3]3,\zeta_3)$, the splitting field of $f$ over $k=\mathbb Q$, where $\zeta_3$ denotes a 
primitive cubic root of unity. Note that $[L:\mathbb Q]=18$. In fact, we can quickly check that $G$ is the generalized dihedral group for the elementary abelian group of order 
9, that is, 
 $$ \operatorname{Gal}(L/\mathbb Q)\cong \{-1,-1\}\ltimes(\mathbb Z/3\mathbb Z)^2: $$
any automorphism in $G$ is determined by its action on $\sqrt[3]2$, $\sqrt[3]3$, and $\zeta_3$; the former two correspond to independent copies of $\mathbb Z/3\mathbb Z$, 
while the latter corresponds to the abelian group of order 2, which acts on $(\mathbb Z/3\mathbb Z)^2$ via the inverse map. We can list $G$ as 
 $$ G=\{\pi_{a,b,c}:a,b=-1,0,1; c=-1,1\}, $$
where $\pi_{a,b,c}$ is the field automorphism of $L$ that maps $\sqrt[3]2$ to $\zeta_3^a\sqrt[3]2$, $\sqrt[3]3$ to $\zeta_3^b\sqrt[3]3$, and $\zeta_3$ to $\zeta_3^c$.

We concentrate on those primes $p$ that do not ramify over $L$. For this, note that $L$ is the compositum of $\mathbb Q(\sqrt[3]2),\mathbb Q(\sqrt[3]3),\mathbb Q(\zeta_3)$, and
therefore an integral prime ramifies over $L$ if and only if it ramifies in one of these fields, so the only such primes are $2,3$. In particular, all strong representatives $p$ of order 
3 are unramified in $L$. 

Although we do not need it explicitly, we briefly explain what remains undescribed from the statement of theorem \ref{thm:chebotarev}, namely the Frobenius, or Artin symbol. 
Suppose an integral prime $p$ is unramified in $L$ and $\mathfrak q$ is a prime of $L$ lying over $p$. The Frobenius is the unique $\sigma\in G$ such that 
 $$ \sigma(\alpha)\equiv\alpha^{{\mathrm N}_{L/\mathbb Q}(p)}\pmod{\mathfrak q} $$ 
for all $\alpha\in L$, where ${\mathrm N}_{L/\mathbb Q}(\cdot)$ is the norm. The choice of $\sigma=\sigma_p$ depends on $\mathfrak q$, but any two such choices are 
conjugate, which explains why we consider conjugacy classes rather than individual members of the Galois group. 

What we really need of Chebotar\"ev's result is the following application, actually due to Frobenius, see \cite{StevenhagenLenstra96}: suppose $g\in\mathbb Q[x]$ is monic and 
$K$ is its splitting field. Given an integral prime $p$, denote by $\mathbb F_p$ the field of $p$ elements. In $\mathbb F_p[x]$, $g$ factors into irreducible polynomials, say 
$g=g_1\cdots g_m$. Letting $n_i$ denote the degree of $g_i$ for $i\in[m]$, we can associate to $g$ the partition 
 $$ \Pi_p=\Pi_p(g)=(n_1,\dots,n_m) $$ 
of $\deg(g)$. Letting $\mathcal G=\operatorname{Gal}(K/\mathbb Q)$, we can identify $\mathcal G$ with a group of permutations of the roots of $g$ in $K$. Fixing an ordering 
of the roots, we can write each $\sigma\in \mathcal G$ as a product of disjoint cycles, say $\sigma=\tau_1\cdots\tau_l$. Letting $t_i$ denote the length of $\tau_i$, we can 
associate to $\sigma$ its cycle pattern 
 $$ \Lambda_\sigma=(t_1,\dots,t_l). $$ 
Frobenius's theorem states that the set of $p$ unramified in $K$ with associated partition $\Pi_p$ has natural density in the set of primes equal to the fraction of 
$\sigma\in\mathcal G$ whose associated cycle pattern $\Lambda_\sigma$ coincides with $\Pi_p$.

In the case under consideration, the condition on $2,3,12$ means that we are looking at those integral primes $p$ with $\Pi_p(f)=(3,3,3)$. Since each $\pi_{a,b,c}$ fixes (setwise) 
the sets $\{\zeta_3^s\sqrt[3]{r}:s\in[3]\}$ for $r=2,3,12$, what we need to count is those automorphisms that do not fix any of the $\zeta_3^s\sqrt[3]{r}$. 

Very explicitly: fix $a,b\in\{-1,0,1\}$ and $c\in\{-1,1\}$. We see that $\pi_{a,b,c}$ maps each $\zeta_3^j\sqrt[3]2$ to $\zeta_3^{jc+a}\sqrt[3]2$, each $\zeta_3^k\sqrt[3]3$ to 
$\zeta_3^{kc+b}\sqrt[3]3$, and each $\zeta_3^l\sqrt[3]{12}$ to $\zeta_3^{lc+2a+b}\sqrt[3]{12}$, and we need that $jc+a\ne j$, $kc+b\ne k$ and $lc+2a+b\ne l$ for any $j,k,l$, 
where the inequalities are all modulo 3. 

There are two cases, depending on $c$. First, if $c=-1$, the first condition says that $a-j\ne j$, or $a\ne 2j$, but this is impossible to 
satisfy simultaneously for all $j$. Second, if $c=1$, what we need is that $a+j\ne j$, $k+b\ne k$ and $l+2a+b\ne l$, that is, $a$, $b$ and $2a+b$ should all be different from 0, 
and the last requirement is equivalent to asking that $a\ne b$. There are precisely two members of $G$ that satisfy all these conditions, namely $\pi_{1,-1,1}$ and $\pi_{-1,1,1}$. 

This means that the set of strong representatives of order 3 has natural density $2/18=1/9$.
\end{proof}

\begin{remark}
Note that the value $1/9$ was to be expected: with notation as in the proof above, since  $|G|=18$, if the set of strong representatives of order 3 was to have a natural density 
$r$ at all, $r$ would have to be a rational number of the form $a/18$ for some $a\in[18]$, and table \ref{table:con-10} strongly suggests that, indeed, $r=1/9=0.\overline1$. 

Note also that our argument in particular established the existence of strong representatives of order 3 (although, of course, there are much simpler proofs of this assertion); the 
point is that this is a benefit that does not automatically generalize, as there are primes $n$ for which there are no strong representatives of order $n$, such as $n=211$, see 
table \ref{table:groupless} (that is, we cannot remove in item (1) of question \ref{qu:natdensity} the hypothesis that strong representatives of order $n$ exist).
\end{remark}

The same approach works in general: given $n$, that a prime $p=nk+1$ is a strong representative of order $n$ means that $i^k\not\equiv j^k\pmod p$ whenever $i<j$ 
are in $[n]$, that is, $(j/i)^k\not\equiv 1\pmod p$, which means that $j\cdot i^{n-1}$ is not an $n^{\mathrm{th}}$ power modulo $p$, or, what is the same, that the polynomial 
$x^n-j\cdot i^{n-1}$ has no roots modulo $p$. This translates, just as in the example above, into a condition on a conjugacy class in the Galois group of certain Galois extension 
of $\mathbb Q$, namely $L=\mathbb Q(\zeta_n,\root n \of j:j\in[n]\cap\mathbb P)$, the splitting field over $\mathbb Q$ of the polynomial 
 $$ f(x)=\prod_{1\le i<j\le n}(x^n-j\cdot i^{n-1}), $$ 
where $\zeta_n$ denotes a primitive $n^{\mathrm{th}}$ root of unity. The condition is in general messier than in the case $n=3$, since many different factorization patterns may 
occur for $f$ in $\mathbb F_p[x]$ that are compatible with $f$ not having roots in $\mathbb F_p$. 

As before, $L$ is the compositum of the fields $\mathbb Q(\sqrt[n] j)$, $\mathbb Q(\zeta_n)$, for $j\in[n]\cap\mathbb P$, so any integral prime that ramifies in $L$ divides $n!$, 
and in particular any strong representative of order $n$ is unramified in $L$. If $n$ is prime, it follows from lemma \ref{lemma:allnpowers} that any $p$ for which $f$ has no roots 
modulo $p$ is automatically congruent to 1 modulo $n$. We also expect that, if $n$ itself is prime, the Galois group of the extension should be given by 
 $$ G=\operatorname{Gal}(L/\mathbb Q)\cong(\mathbb Z/n\mathbb Z)^*\ltimes(\mathbb Z/n\mathbb Z)^{\pi(n)}, $$
where $(\mathbb Z/n\mathbb Z)^*$ denotes the group of units modulo $n$, a group of order $\phi(n)=n-1$. We have verified this by direct computation for small values of $n$. If 
$n$ is not prime, however, there may be unexpected relations between the various $\sqrt[n]j$ and $\zeta_n$. For instance, for $n=8$ and $\zeta_8=e^{2\pi \sqrt{-1}/8}$, we have 
$\zeta_8+\zeta_8^{-1}=\sqrt2$ or, if $n=2p$ where $p$ is prime and $p\equiv 1\pmod 4$, then $\sqrt p\in\mathbb Q(\zeta_n)$. Thus in general the Galois group may 
be a proper subgroup of the semidirect product indicated above (determining whether this is indeed the case involves Kummer theory). Still, we can ensure that the primes we 
consider are congruent to 1 modulo $n$: note that there is a natural projection from the Galois group of the extension to $\operatorname{Gal}(\mathbb Q(\zeta_n)/\mathbb Q)$, 
and the congruence condition means that this projection maps the automorphisms we are interested in counting to the identity of this smaller Galois group.

Via Chebotar\"ev's theorem (or, rather, Frobenius's theorem), if the requirements imposed on the automorphisms in $G$ are at all satisfiable, then a positive proportion of all primes
are strong representatives modulo $n$. But that the requirements are satisfiable is precisely the claim that there is at least one such prime. We have proved:

\begin{theorem} \label{thm:positivenatural}
If there is a strong representative of order $n$, then the set of such primes has positive natural density in the set of all primes.
\end{theorem}

Given an $n$-satisfactory coloring $c$, this analysis can be further extended to capture in addition that $p$ is a strong representative for $c$. This is slightly more delicate, but 
the point is that if $c$ admits a strong representative, then it is \emph{multiplicative}, in the sense of section \ref{sec:multiplicative}, and such a $c$ is completely determined by 
the tuple of values $(c(ij):i\le j\in[n])$, see for instance corollary \ref{cor:finitelymultiplicative}. This translates into yet a further condition on a conjugacy class, and the problem of 
computing the natural densities becomes a purely group-theoretic question.

Note that as long as no hidden relations are present (in particular, we expect this to be the case for $n$ prime), the extension has degree $[L:\mathbb Q]=n^{\pi(n)}\cdot\phi(n)$. 
For instance, for $n=2$, the extension is of degree 2, and thus table \ref{table:con-10} suggests the density is $1/2$, as we indeed verified in fact \ref{fact:den-2}. Table 
\ref{table:suggested} shows the degree of the corresponding extension for each prime $n$ with $2\le n\le 10$ and the density that table \ref{table:con-10} suggests accordingly. 
For $n=4$ no additional relations occur either. In that case, the extension has degree $4^2\cdot 2=32$, and the expected density is $1/16=0.0625$.\footnote{The discussion 
here incorporates suggestions of Felipe Voloch to the first-named author at \href{https://mathoverflow.net/q/141993}{https://mathoverflow.net/q/141993} and through private 
communication. Thanks are also due to David E Speyer.}
 
\begin{table}
\begin{tabular}{|c||c|c|}
\hline 
$n$&degree&expected density\\
\hline\hline 
2&$2^1\cdot 1=2$&$1/2=0.5$\\
\hline
3&$3^2\cdot 2=18$&$1/9=0.\overline1$\\
\hline
5&$5^3\cdot4=500$&$2/125=0.016$\\
\hline
7&$7^4\cdot 6=14406$&$6/2401=0.00249\dots$\\
\hline
\end{tabular}
\vspace{3mm}
\caption{Degree of $\mathbb Q(\zeta_n,\sqrt[n] j:j\in[n]\cap\mathbb P):\mathbb Q$ and expected density of the set of strong representatives of order $n$ in the set of all 
primes for $n$ prime below 10.}
\vspace{-5mm}
\label{table:suggested}
\end{table}

\begin{remark}
Chebotar\"ev's theorem admits an effective version. It follows that the expected densities can be verified not just by a combinatorial analysis of the relevant Galois groups, but 
simply by determining the sizes of these groups, and extending the entries in table \ref{table:con-10} to a sufficiently large number, allowing us to compare the results with 
rationals of the form $a/m$ where $m$ is the size of the corresponding group. 
\end{remark}

\subsection{Asymptotics of coincidences} \label{subs:asymptotics}

Fix $k>2$. For primes $p=kn+1$ sufficiently large, theorem \ref{thm:k-representative} shows the existence of coincidences 
 $$ a^k\equiv b^k\pmod p $$
with $1\le a<b\le n$. We close this section by showing that, in fact, the number of such coincidences is asymptotically proportional to $p$.

The result is due to Noam D. Elkies\footnote{\label{footnote:elkies}See \href{https://mathoverflow.net/q/78270/}{https://mathoverflow.net/q/78270/}}, and what follows is closely 
based on his argument.

\begin{theorem}[Elkies] \label{thm:elkies}
For $k>2$, the number of coincidences $a^k\equiv b^k\pmod p$ for $p$ of the form $kn+1$ and sufficiently large, and distinct $a,b\in[n]$ is 
 $$ C_k p + O_k(p^{1-\epsilon(k)}), $$ 
where 
 $$ C_k = \begin{cases}
 \vspace{1mm}\displaystyle \frac{k - 1}{2k^2} & \mbox{ if $k$ is odd, and}\\ 
 \displaystyle \frac{k - 2}{2k^2} & \mbox{ if $k$ is even,}
 \end{cases}$$ 
and $\epsilon(k) = 1/\phi(k)$. 
\end{theorem}

\begin{proof}
First, for $a,b$ nonzero and distinct modulo $p$, that 
 $$ a^k \equiv b^k\pmod p$$
is equivalent to saying that $b\equiv ma \pmod p$ where $m\ne 1$ is a $k^{\mathrm{th}}$ root of unity: 
$m^k\equiv 1\pmod p$. Since we are only interested in the case where $a,b\in[n]$, for $k$ even we further exclude $m=-1$. Fix $m$, and consider the nonzero vectors $(a,b)$ 
in $\mathbb Z^2$ defined by the relation $b \equiv ma \pmod p$, $a,b\in[n]$. Note that for any such vector, $p\mid a^k-b^k$, and the latter factors into homogeneous 
polynomials in $a, b$ of degree at most $\phi(k)$, none of which is zero, and therefore the length of the vector is $\Omega(p^{\epsilon(k)})$. 

This means that the solutions to the equation $b \equiv ma \pmod p$ with $a, b \in [n]$ are the lattice points in the square with sides parallel to the axis of side length 
$n\approx p/k$ and bottom left corner at the origin. This number can be readily estimated as $p^{-1}(p/ k)^2 = p/ k^2$, with an error bound proportional to the fraction 
 $$ \frac{\text{side length}}{\text{length of smallest such vector}}=O(p^{1-\epsilon(k)}). $$ 
 
The total of such coincidences is now obtained by summing these estimates over all $k-1$ or $k-2$ possible values of $m$, and then dividing by 2 (since each coincidence has 
been counted twice in the above, as both $(a,b)$ and $(b,a)$).
\end{proof}

The argument can be strengthened to estimate for $k,n,p$ as before the proportion of distinct $k^{\mathrm{th}}$ powers of members of $n$. A quick computation verifies that the 
fractions 
 $$ \frac{|\{(i^k\bmod p) :i\in[n]\}|} n $$ 
stay rather close to $2/3$ for $k=3$, and to $84/125$ for $k=5$. For instance, for $k = 3$, $n = 387,642$, and $p = 1,162,927$, the fraction is 
 $$ 258429/387642= 0.6666692464\dots, $$
while for $k=5$, $n=35,804$ and $p=179,021$, the fraction is 
 $$ 24065/35804=0.6721316054\dots\,.$$

The result, also due to Elkies, shows that these values are to be expected.

\begin{theorem}[Elkies] \label{thm:elkies2}
For $k>2$ and $p$ of the form $kn+1$ and sufficiently large, the fraction $|\{(i^k\bmod p) :i\in[n]\}|/n$ of distinct $k^{\mathrm{th}}$ powers of members of $[n]$ is asymptotic to 
$1-((k -1)^k + 1) /k^k$.
\end{theorem}

In particular, for $k=3$ the fraction approaches $1-\frac{2^3+1}{3^3}=2/3$ and for $k = 5$ it approaches $1-\frac{4^5+1}{5^5}=84 /125$, as expected, and, as 
$k\to\infty$, the proportion of $k^{\mathrm{th}}$ powers with small $k^{\mathrm{th}}$ roots approaches $1 - (1 /e)$.

As Elkies remarks (at the post linked to in footnote \ref{footnote:elkies}), the same approach as for the previous theorem allows one to estimate the number of coincidental triples, 
or quadruples, etc. Care must be taken ``with subsets of the $k^{\mathrm{th}}$ roots of unity that have integer dependencies, but at least when $k$ is prime there are no 
dependencies except that all $k$ of them sum to zero''. Elkies further indicates that for $j<k$ the number of $j$-element subsets of $n$ with the same $k^{\mathrm{th}}$ power is 
asymptotic to  
 $$ \binom kj p /k^{j+1}, $$ 
while there are no such subsets with $j = k$ because the sum of all $k$ solutions of $a^k \equiv c \pmod p$ vanishes (for any $c$). By an inclusion-exclusion 
argument one then obtains the estimate indicated in theorem \ref{thm:elkies2}.

\section{Multiplicative colorings} \label{sec:multiplicative}

\subsection{Multiplicativity} \label{subs:multiplicative}

As shown in the previous section, any $n$-satisfactory coloring for $n\le 5$ admits strong representatives. Colorings with strong representatives are very special: fix some $n$, 
and suppose that $c$ is a satisfactory coloring of $K_n$ admitting a strong representative $p=kn+1$. Let 
 $$ G=\{(a^k\bmod p): a\in[n]\}\le(\mathbb Z/p\mathbb Z)^*. $$ 
The group $G$ is isomorphic to $\mathbb Z/n\mathbb Z$. The map $h\!:K_n\to G$ given by 
 $$ h(a)=(a^k\bmod p) $$
satisfies
 $$ h(ab) = h(a) \cdot h(b) $$
for any $a,b \in K_n$, where $ab$ is the usual product of $a$ and $b$ and $h(a)\cdot h(b)$ is the
product in $G$. We generalize this setting in the following definition.

\begin{definition} \label{def:multiplicative}
A satisfactory coloring $c$ of $K_n$ is \emph{multiplicative} if and only if there exists a group $(G,\cdot)$ of order $n$ and a bijection $\varphi\!:[n]\to G$ such that, thinking of 
$c$ as a map $c\!:K_n\to[n]$ with $c(i)=i$ for all $i\in[n]$, and letting $h=\varphi\circ c$, we have that
\begin{equation} \label{eq:mult}
h(ab) = h(a) \cdot h(b) 
\end{equation}
for all $a, b \in K_n$. In this case, we say that $c$ is a $G$-coloring.

Multiplicative colorings of $\mathbb Z^+$ are defined the same way, only requiring that the domain of $c$ be $\mathbb Z^+$ and that equation \ref{eq:mult} holds for all positive 
integers.
\end{definition}

The usefulness of the notion is stated explicitly in theorem \ref{thm:partialisomorphism} below, the point is that to describe a multiplicative coloring it is enough to describe what we 
call a partial isomorphism, see \S\,\ref{subs:partialisomorphism}, which reduces the problem of searching for a multiplicative coloring to a finite question. 

The following observation should be immediate.

\begin{fact}
If a satisfactory coloring of $K_n$ is both a $G_1$-coloring and a $G_2$-coloring, then $G_1 \cong G_2$.
\end{fact}

Note that if $G$ is as in definition \ref{def:multiplicative}, then $G$ is abelian, and consequently we adopt additive notation in what follows, so $h$ is a kind of discrete logarithm 
but, rather than referring to it this way, we also say that $h$ is multiplicative.

\begin{definition} 
If $(G, +)$ is an abelian group (of order $n$) and the map $h\!:K_n \to G$ satisfies that $h(ab) = h(a) + h(b)$ for any $a, b \in K_n$, we say that $h$ is \emph{multiplicative}.
\end{definition}

\begin{corollary} \label{cor:finitelymultiplicative}
For any $n$, there are only finitely many multiplicative colorings of $K_n$.
\end{corollary}

\begin{proof}
Suppose $c$ is multiplicative as witnessed by $(G, +), \varphi$. Let $h =\varphi\circ  c$, where as before, $c$ is interpreted as a map $c\!:K_n\to[n]$ with $c(i)=i$ for $i\in[n]$, so 
$h(ab) = h(a)+h(b)$ for all $a, b \in K_n$. Note that this induces a group structure $\oplus$ on $[n]$ isomorphic to $G$ because $c$ is the identity on $[n]$, so if $a \in[n]$, then 
 $$ h(a) = \varphi(c(a)) = \varphi(a), $$ 
and we are setting $a \oplus b = d$ for $a,b,d \in [n]$ if and only if $\varphi(d) = \varphi(a) + \varphi(b)$. By identifying $(G, +)$ with 
$([n], \oplus)$, it follows that we may assume that $\varphi$ is the identity so $h = c$. But now we see that $([n],\oplus)$ completely determines $c$. In effect, if 
$p_1<\dots<p_{\pi(n)}$ are the primes less than or equal to $n$ and $s=\pi(n)$, then the multiplicity requirement gives us
\begin{equation} \label{eq:multformula}
c(p_1^{\alpha_1}\cdots p_s^{\alpha_s} ) = \alpha_1c(p_1) \oplus\cdots\oplus\alpha_sc(p_s),
\end{equation}
where $\alpha_i c(p_i)$ is the result of adding $c(p_i)$ to itself $\alpha_i$ times in $([n],\oplus)$. 

Since there are only finitely many group structures on $[n]$, we are done. In fact, all these group structures can be efficiently identified, from the classification theorem for finite 
abelian groups.
\end{proof}

\begin{remark} \label{rmk:periodic}
Note that whenever an $n$-satisfactory coloring is multiplicative as witnessed by a group $(G,\cdot)$, the corresponding tiling of $\mathbb O_n$ by unit blocks of $n$ colors is 
periodic, which explains the patterns observed in figures \ref{figure:3tiling-b} and \ref{figure:4tiling-b}. Indeed, this periodicity is simply a consequence of the fact that $x^n$ is the 
identity of $G$ for any $x\in G$. In turn, this also gives periodicity of the tiling by $T_n$, see \S\,\ref{subs:translation}.
\end{remark}

In what follows, given an abelian group $(G, \oplus)$, we will denote $\alpha$-fold sums of the form $\underbrace{g\oplus\cdots\oplus g}_{\text{$\alpha$ times}}$ by 
$g^{\oplus\alpha}$ rather than $\alpha g$ as above. We extend the notation to include negative values of $\alpha$ (including $\alpha=-1$).

\begin{remark}
Note that not every abelian group structure on $[n]$ gives rise to a multiplicative coloring. For example, if $n = 4$, then $\oplus$ is given by 
 $$ a\oplus b = ab \pmod 5 $$ 
and, in 
particular, $([4],\oplus)$ is isomorphic to $\mathbb Z/4\mathbb Z$ and not to $\mathbb Z/2\mathbb Z \times \mathbb Z/2\mathbb Z$. 
\end{remark}

The case when $G \cong \mathbb Z/n\mathbb Z$, as in the case of a strong representation, deserves special attention. 

\begin{question} \label{q:zncoloringrepresentation}
Does every $\mathbb Z/n\mathbb Z$-coloring admit a strong representation?
\end{question}

This is a good point to reiterate what we mentioned in remark \ref{rmk:variety}. Perhaps surprisingly, the answer to question \ref{q:zncoloringrepresentation} is negative, as we 
show below, see the analysis of multiplicative 8-colorings in \S\,\ref{sub:nle8} and in particular the coloring described in table \ref{table:z81537}. Nevertheless, we can answer 
the question affirmatively at the cost of replacing strong representations with a weak variant, see remark \ref{rem:weak}. Although so far our examples and results have only 
exhibited multiplicative colorings, it should be pointed out that not every satisfactory coloring is multiplicative, see \S\,\ref{subs:n6} and \S\,\ref{subs:n8} for dramatic examples. 
Similarly, not every multiplicative coloring is a $\mathbb Z/n\mathbb Z$-coloring. Examples are presented in \S\,\ref{sub:nle8}, in particular see table \ref{table:3751} for $a=1$. 
Nonmultiplicative colorings seem more difficult to analyze, and we do not understand them well. In what follows, we restrict our attention to the multiplicative case except for 
\S\,\ref{subs:n6} and \S\,\ref{subs:n8}.

\subsection{Partial $G$-isomorphisms} \label{subs:partialisomorphism}

The following notion has appeared before in the literature, in particular in connection with Graham's conjecture, and goes back at least to Galovich and Stein 
\cite{GalovichStein81}, who talk of KM logarithms, for Kummer and Mills. In MathOverflow, Ewan 
Delanoy\footnote{See \href{https://mathoverflow.net/q/26358/}{https://mathoverflow.net/q/26358/}} considered the case 
$G = \mathbb Z/n\mathbb Z$. Though not identical, it is closely related to the concept of Freiman homomorphism in additive combinatorics, see \cite[definition 5.21]{TaoVu06}.

\begin{definition} \label{def:partialisomorphism}
Let $(G, +)$ be an abelian group of order $n$. A map $h\!:[n]\to G$ is a \emph{partial $G$-isomorphism} if and only if $h$ is a bijection and, whenever $a, b \in [n]$, if $ab \in[n]$, 
then $h(ab) = h(a)+h(b)$. If $G = \mathbb Z/n\mathbb Z$, we simply call $h$ a \emph{partial isomorphism}.
\end{definition}

\begin{remark} \label{rem:chandler}
We require $G$ to be abelian as our goal is to relate partial $G$-isomorphisms to satisfactory colorings. This is done via an explicit construction in theorem 
\ref{thm:partialisomorphism} below, and although the coloring we describe is perhaps the ``natural'' one, our formula requires that $G$ is abelian, and we do not see a way to 
proceed otherwise. But the question of whether there are partial $G$-isomorphisms where $G$ is not abelian is interesting in its own right. This seems to be open in 
general, but for $n$ odd the answer is negative, as shown by K. A. Chandler \cite{Chandler88}.
\end{remark}

Since $h$ is a bijection, it induces a group operation $\oplus$ on $[n]$ such that 
 $$ ([n],\oplus) \cong (G,+) $$ 
and $\oplus$ extends the partial graph of multiplication on $[n]$. Our use of the term isomorphism here is perhaps further justified by noting that if $h$ is a partial 
$G$-isomorphism, then $h(1) = h(1 \cdot 1) = h(1\oplus 1)=h(1) + h(1)$, and it follows that $h(1) = 0_G$. 

\begin{theorem} \label{thm:partialisomorphism}
If $h\!: [n]\to G$ is a partial $G$-isomorphism, then $h$ can be uniquely extended to a multiplicative map $\hat h\!:K_n \to G$. Moreover, $h^{-1}\circ\hat h\!: K_n \to [n]$ is a 
$G$-coloring of $K_n$.
\end{theorem}

\begin{proof}
Let $h\!:[n]\to G$ be a partial $G$-isomorphism. Letting $p_1,\dots, p_s$ be the primes less than or equal to $n$, a map $\hat h\!:K_n \to G$ extends $h$ and is multiplicative if 
and only if for any $a_1,\dots,a_s \in \mathbb N$, we have
 $$ \hat h(p_1^{a_1}\cdots p_s^{a_s}) =  \bigoplus_{i=1}^s h(p_i)^{\oplus a_i}. $$
This proves the existence and uniqueness of the extension $\hat h$. 

Moreover, if $1 \le i < j \le n$ and $a\in K_n$, then
 $$ \hat h(ia) = \hat h(i) + \hat h(a) \ne \hat h(j) + \hat h(a) = \hat h(ja) $$
because $\hat h \upharpoonright[n]=h$ is a bijection.

Letting $c=h^{-1}\circ\hat h$, this gives us that $c\!:K_n\to[n]$ is a $G$-coloring.
\end{proof}

\begin{remark}
Note the similarity between this argument and the proof of corollary \ref{cor:finitelymultiplicative}.
\end{remark}

Obviously, if $\hat h\!:K_n \to G$ is multiplicative and $h = \hat h\upharpoonright[n]$ is a bijection, then $h$ is a partial $G$-isomorphism. Therefore, if $c\!: K_n \to [n]$ is a 
$G$-coloring as witnessed by the bijection $\varphi\!:[n]\to G$, then $\varphi$ is a partial $G$-isomorphism as, by definition, $h = \varphi\circ c$ is multiplicative, and 
$\varphi = h \upharpoonright [n]$.

This shows that the problem of building $G$-colorings of $K_n$ is equivalent to the problem of building partial $G$-isomorphisms or, equivalently, $G$-satisfactory groups:

\begin{definition} \label{def:satisfactory}
Given an abelian group $(G,+)$ of order $n$, we say that an abelian group structure on $[n]$, $([n], \oplus)$, is a \emph{$G$-satisfactory group} if and only if 
 $$ ([n], \oplus) \cong (G, +) $$ 
and $a\oplus b = ab$ whenever $a, b, ab \in[n]$.

We say that the $G$-coloring resulting from extending $\oplus$ as in theorem \ref{thm:partialisomorphism} is \emph{associated} to $([n],\oplus)$.
\end{definition}

There is a two-fold advantage on building $G$-satisfactory groups rather than partial $G$-isomorphisms: first, the extension to a $G$-coloring is immediate. Second, and more 
significantly, different partial $G$-isomorphisms may give rise to the same $G$-coloring, as the notion is only uniquely determined up to automorphisms of $G$.

For example, if $h_1\!:[6]\to\mathbb Z/6\mathbb Z$ and $h_2\!:[6]\to\mathbb Z/6\mathbb Z$ are the maps
 $$ (1,2,3,4,5,6) \overset{h_1}\longmapsto (0,2,1,4,5,3) $$
and  
 $$ (1,2,3,4,5,6)\overset{h_2}\longmapsto (0,4,5,2,1,3), $$
then both give rise to the $\mathbb Z/6\mathbb Z$-coloring strongly represented by $7 = 1\cdot 6 + 1$, and this coloring is associated to the $G$-satisfactory group shown in 
table \ref{table:z6satisfactory}.

\begin{table}[ht]
\begin{tabular}{c||c|c|c|c|c|c|} 
$\oplus$&1&2&3&4&5&6\\
\hline \hline
1&1&2&3&4&5&6\\
2&2&4&6&1&3&5\\
3&3&6&2&5&1&4\\
4&4&1&5&2&6&3\\
5&5&3&1&6&4&2\\
6&6&5&4&3&2&1\\
\hline
\end{tabular}\vspace{3mm}
\caption{A $\mathbb Z/6\mathbb Z$-satisfactory group.}
\label{table:z6satisfactory}
\vspace{-7mm}
\end{table}

In \S\,\ref{sub:nle8}, we use systematically the notation of $G$-satisfactory groups to identify all multiplicative colorings with at most eight colors.

\begin{remark} \label{rem:weak} 
We are now in a position to explain how $\mathbb Z/n\mathbb Z$-colorings or, equivalently, partial isomorphisms are closely related to strong representations. In fact, we can 
prove that any $\mathbb Z/n\mathbb Z$-coloring admits a ``weak'' representation. Let $h\!:[n]\to \mathbb Z/n\mathbb Z$ be a partial isomorphism. As before, let $p_1,\dots,p_s$ 
be the primes less than or equal to $n$. Extend $h$ to a map from $K_n$ to $\mathbb Z/n\mathbb Z$ as in the proof of theorem \ref{thm:partialisomorphism}. Denote the extension 
again by $h$.

By Dirichlet's theorem, there are primes $P$ of the form $kn + 1$. For any such $P$, let $g$ be a primitive root modulo $P$, i.e., a generator of $(\mathbb Z/P\mathbb Z)^*$. In 
other words, the powers $g^{ki}$ are precisely the $k^{\mathrm{th}}$ power residues modulo $P$. Invoking again Dirichlet's theorem, for each $p_i$ we can find a prime $q_i$ such 
that 
 $$ q_i\equiv g^{h(p_i)} \pmod P. $$
Now for $x\in K_n$ define $d\!: K_n \to (\mathbb Z/P\mathbb Z)^*$ by $d(x) = g^{kh(x)}$.

If $x=\prod_{i=1}^s p_i^{a_i}$, then $h(x)=\sum_i a_ih(p_i)$ and 
\begin{equation} \label{eq:prod}
d(x)=\prod_i(g^{h(p_i)})^{ka_i}=\left(\prod_i q_i^{a_i}\right)^k
\end{equation}
where of course the products are computed modulo $P$.

The point is that if $i,j \in[n]$, then $d(i) \ne d(j)$ because $h(i) \ne h(j)$, $h$ being a bijection. If $0 \le h(i) < h(j) < n$, then $0 \le kh(i) < kh(j) < kn$, and $g^{kh(i)} \ne g^{kh(j)}$, 
since $g$ is a primitive root. It follows that $d(ix) = d(i)d(x) \ne d(j)d(x) = d(jx)$, and $d$ defines a $\mathbb Z/n\mathbb Z$-coloring.

Note how close the coloring given by equation (\ref{eq:prod}) is to the colorings described in definition \ref{def:strongrepresentatives}. Strong representations are the particular 
case where we can choose $P$ for which we can take $q_i = p_i$ for all $i$.
\end{remark}

Partial isomorphisms are easy to construct ``by hand'' for small values of $n$. Examples of partial isomorphisms for all $n\le 31$ are given in table \ref{table:linear}. In 
appendix B of the second-named author's master's thesis\footnote{See \href{http://scholarworks.boisestate.edu/td/231/}{http://scholarworks.boisestate.edu/td/231/}}, this is 
extended to all $n\le 54$. The authors of \cite{ForcadePollington90} have verified their existence for all $n<195$.

Given $n$, define $M$ and $M_{K_n}$ as the sets of multiplicative colorings of $\mathbb Z^+$ and of $K_n$, respectively. In corollary \ref{cor:finitelymultiplicative} we 
showed that $M_{K_n}$ is finite. We now show that restricting attention to colorings in $M$ does not affect the computation of the number of satisfactory colorings (corollary 
\ref{cor:many}).

\begin{theorem} \label{thm:manymultiplicative}
If $n>1$ and $M_{K_n}\ne\emptyset$, then $|M| = \mathfrak c$.
\end{theorem}

\begin{proof}
As in corollary \ref{cor:many}, it is enough to show that $n^{\aleph_0}\le |M|$. Let $c\!:K_n\to[n]$ be a multiplicative coloring associated to the $G$-satisfactory group 
$([n],\oplus)$. To each prime $p$ assign a number $a_p\in[n]$ with the only restriction that $a_p = p$ if $p \in n$. Now define $c'\!:\mathbb Z^+\to[n]$ as follows: if 
$m \in \mathbb Z^+$, let $\prod_i p_i^{b_i}$ be its prime factorization, and set
 $$ c'(m)=\bigoplus_i{a_{p_i}}^{\oplus b_i}. $$
It is immediate that any $c'$ defined this way is multiplicative and extends $c$, and that different sequences $(a_p: p$ prime$)$ give rise to different $c'$, and therefore we have 
associated $n^{\aleph_0}$ colorings in $M$ to each $c\in M_{K_n}$.
\end{proof}

\subsection{Translation invariance} \label{subs:translation}

In this subsection we show that multiplicativity of a coloring, an algebraic condition, is equivalent to translation invariance, a geometric condition. This helps elucidate the relation 
between multiplicativity of colorings and periodicity of the corresponding tilings. We have organized the presentation to highlight how far the assumption of translation invariance 
alone takes us, with the equivalence itself established at the end.

Recall that if $c$ is a coloring of $K_n$ and $k\in K_n$, then $c_k$ is the coloring where two numbers $m,m'\in K_n$ receive the same color precisely when $c(km)=c(km')$.

\begin{definition} \label{def:trans}
A coloring $c$ of $K_n$ is \emph{translation invariant} if and only if $c_k=c$ for all $k\in K_n$.
\end{definition}

For any $k\in K_n$, we can naturally identify $\mathbb O_n$ and $\mathbb O_n+t(k)$. This induces a coloring of $\mathbb O_n$ from one of $\mathbb O_n+t(k)$. If we start 
with $c$, the resulting coloring is precisely $c_k$. That $c$ is translation invariant means that, for any $k\in K_n$, this coloring is again $c$. Clearly, multiplicative colorings are 
translation invariant. Notice that translation invariance is a strong requirement on a coloring: the color classes must all look the same, no matter from where we start to look at 
them. We illustrate this with figure \ref{fig:4classes}, showing the four color classes of the 4-satisfactory coloring depicted in figure \ref{figure:4tiling}. In the figure, we have also 
indicated the axes of an orthant $\mathbb O_4+t(k)$, and the reader can see that the four color classes of the induced coloring of this orthant look precisely like the original ones.

\begin{figure}[ht]
\begin{tikzpicture}[scale=.3]

\draw[step=1cm,gray,very thin] (0,0) grid (6.1,6.1); 
\draw[thick,->] (0,0) -- (6.1,0);
\draw[thick,->] (0,0) -- (0,6.1);

\draw[step=1cm,gray,very thin] (8,0) grid (14.1,6.1);
\draw[thick,->] (8,0) -- (14.1,0);
\draw[thick,->] (8,0) -- (8,6.1);

\draw[step=1cm,gray,very thin] (16,0) grid (22.1,6.1);
\draw[thick,->] (16,0) -- (22.1,0);
\draw[thick,->] (16,0) -- (16,6.1);

\draw[step=1cm,gray,very thin] (24,0) grid (30.1,6.1); 
\draw[thick,->] (24,0) -- (30.1,0);
\draw[thick,->] (24,0) -- (24,6.1);

\draw[very thick,densely dotted,->] (3,1) -- (6.3,1);
\draw[very thick,densely dotted,->] (3,1) -- (3,6.3);

\draw[very thick,densely dotted,->] (11,1) -- (14.3,1);
\draw[very thick,densely dotted,->] (11,1) -- (11,6.3);

\draw[very thick,densely dotted,->] (19,1) -- (22.3,1);
\draw[very thick,densely dotted,->] (19,1) -- (19,6.3);

\draw[very thick,densely dotted,->] (27,1) -- (30.3,1);
\draw[very thick,densely dotted,->] (27,1) -- (27,6.3);

\path[fill=red] (0,0) -- (1,0) -- (1,2) -- (3,2) -- (3,4) -- (5,4) -- (5,6) -- (6,6) -- (6,5) -- (4,5) -- (4,3) -- (2,3) -- (2,1) -- (0,1) -- (0,0);
\path[fill=red] (4,0) -- (5,0) -- (5,2) -- (6,2) -- (6,1) -- (4,1) -- (4,0);
\path[fill=red] (0,4) -- (1,4) -- (1,6) -- (2,6) -- (2,5) -- (0,5) -- (0,4);

\path[fill=blue] (9,0) -- (10,0) -- (10,2) -- (12,2) -- (12,4) -- (14,4) -- (14,5) -- (13,5) -- (13,3) -- (11,3) -- (11,1) -- (9,1) -- (9,0);
\path[fill=blue] (13,0) -- (14,0) -- (14,1) -- (13,1) -- (13,0);
\path[fill=blue] (8,3) -- (9,3) -- (9,5) -- (11,5) -- (11,6) -- (10,6) -- (10,4) -- (8,4) -- (8,3);

\path[fill=black] (16,1) -- (17,1) -- (17,3) -- (19,3) -- (19,5) -- (21,5) -- (21,6) -- (20,6) -- (20,4) -- (18,4) -- (18,2) -- (16,2) -- (16,1);
\path[fill=black] (19,0) -- (20,0) -- (20,2) -- (22,2) -- (22,3) -- (21,3) -- (21,1) -- (19,1) -- (19,0);
\path[fill=black] (16,5) -- (17,5) -- (17,6) -- (16,6) -- (16,5); 

\path[fill=purple] (26,0) -- (27,0) -- (27,2) -- (29,2) -- (29,4) -- (30, 4) -- (30,3) -- (28,3) -- (28,1) -- (26,1) -- (26,0);
\path[fill=purple] (24,2) -- (25,2) -- (25,4) -- (27,4) -- (27,6) -- (28,6) -- (28,5) -- (26,5) -- (26,3) -- (24,3) -- (24,2); 

\end{tikzpicture}
\caption{The four color classes of the unique 4-satisfactory coloring.}
\label{fig:4classes}
\end{figure}
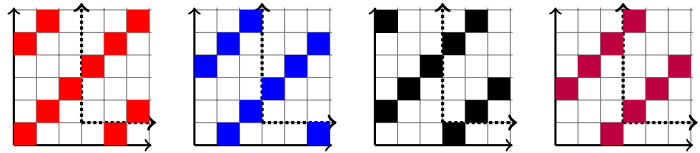

\begin{fact}
Let $c$ be a translation invariant finite coloring of $K_n$. For any $k\in K_n$ there is a least positive integer $o(k)$ such that $c(k^{o(k)}m)=c(m)$ for all $m\in K_n$. 

If $N$ is the number of colors used by $c$, then $o(k)\mid N$. 
\end{fact}

\begin{proof}
We simply use the standard argument for Lagrange's theorem: the list of colors $c(1),c(k),c(k^2),\dots$ must eventually have repetitions. If $i<j$ and 
 $$ c(k^i)=c(k^j), $$ 
then, since $c_{k^i}=c$, we see that $c(1)=c(k^{j-i})$, and it follows that there is a least $o(k)>0$ with $c(k^{o(k)})=c(1)$. By translation invariance, in fact $c(k^{o(k)}m)=c(m)$ 
for all $m\in K_n$. 

For any $m\in K_n$, the colors $c(m),c(km),\dots,c(k^{o(k)-1}m)$ are all distinct, since a coincidence $c(k^am)=c(k^bm)$ with $a<b$ implies $c(1)=c(k^{b-a})$ by translation 
invariance. Let $L_m=\{c(m),c(km),\dots,c(k^{o(k)-1}m)\}$, and note that any two such sets $L_m,L_{m'}$ are either disjoint or coincide, since if $c(mk^a)=c(m'k^b)$, then, 
letting $l=b-a$ if $b\ge a$ or $b+o(k)-a$ otherwise, we see that $c(m)=c(m'k^l)$, from which $L_m=L_{m'}$ follows. This shows that the sets $L_m$ partition the set of colors
into classes of the same size, which completes the proof.
\end{proof}

Arguably, a coloring $c$ of $K_n$ deserves to be called periodic if there is a $k\in K_n$ larger than 1 such that $c(km)=c(m)$ for all $m\in K_n$. Our actual definition is 
somewhat more stringent.

\begin{definition}
Let $c$ be a coloring of $K_n$, $k\in K_n$ be larger than 1, and $l\in\mathbb Z^+$. Say that $c$ is \emph{periodic in the direction of $k$ with period $l$} if and only if 
$c(k^lm)=c(m)$ for all $m\in K_n$.

Say that $c$ is \emph{periodic} if and only if it is periodic in every direction. 
\end{definition}

The following result is an immediate consequence of the existence of the \emph{orders} $o(k)$, $k\in K_n$, for translation invariant colorings.

\begin{corollary}
If $c$ is a translation invariant $n$-coloring of $K_n$, then $c$ is periodic, with period $n$ in every direction.
\end{corollary} 

Translation invariant $n$-satisfactory colorings are particularly well-behaved. 

\begin{lemma} \label{lem:extend}
For any $N$, any translation invariant $N$-coloring of $K_n$ admits a unique extension to such a coloring of $\hat K_n$; moreover, the extension $c$ is translation invariant in 
the strong sense that $c_k=c$ for all $k\in\hat K_n$. If the original coloring is in addition $n$-satisfactory, then so is the extension. 
\end{lemma}

\begin{proof}
Let $c$ be translation invariant. We define an extension, that we also denote by $c$, in the natural way: given $m,m'\in K_n$, let 
 $$ c(m'/m)\coloneqq c(m'm^{o(m)-1}). $$ 
This is well-defined, in the sense that if $m'/m=s'/s$ for $s,s'\in K_n$, then 
 $$ c(m'm^{o(m)-1})=c(s's^{o(s)-1}), $$ 
because
 $$ c(m'm^{o(m)-1}ms)=c(m's)=c(ms')=c(s's^{o(s)-1}ms), $$ 
and therefore  $c(m'm^{o(m)-1})=c(s's^{o(s)-1})$, by translation invariance.

Similarly, if $a,m_1,m_2\in K_n$, then 
\begin{equation} \label{eq:product}
c\left(\frac a{m_1m_2}\right)=c(a\, m_1^{o(m_1)-1}m_2^{o(m_2)-1}),
\end{equation} 
because 
 $$ c(a\, m_1^{o(m_1)-1}m_2^{o(m_2)-1}m_1m_2)=c(a)=c(a (m_1m_2)^{o(m_1m_2)-1}m_1m_2). $$
 
We argue that the extension is translation invariant in the strong sense indicated above: let $k\in \hat K_n$. We must show that $c_k=c$, that is, that if $a,b\in\hat K_n$, then 
$c_k(a)=c_k(b)$ if and only if $c(a)=c(b)$. For this, let $m_1,\dots,m_6\in K_n$ be such that $a=m_1/m_2$, $b=m_3/m_4$ and $k=m_5/m_6$, and note that $c_k(a)=c_k(b)$ 
if and only if 
 $$ c\left(\frac{m_1m_5}{m_2m_6}\right)=c\left(\frac{m_3m_5}{m_4m_6}\right) $$ 
or, equivalently, 
 $$ c(m_1m_5m_2^{o(m_2)-1}m_6^{o(m_6)-1})=c(m_3m_5m_4^{o(m_4)-1}m_6^{o(m_6)-1}), $$ 
which, in turn, is equivalent to 
 $$ c(m_1m_2^{o(m_2)-1})=c(m_3m_4^{o(m_4)-1}), $$ 
that is, to $c(a)=c(b)$, as wanted.

Suppose now that $c$ is in addition $n$-satisfactoty. To see that the extension is again $n$-satisfactory, note that if $i,j\in[n]$ and $c(im'/m)=c(jm'/m)$, then 
$c(im'm^{o(s)-1})=c(jm'm^{o(s)-1})$, and it follows that $i=j$. 

To see that the extension we defined is the only possible translation invariant extension using the same colors, suppose $c'$ is such an extension of $c$, and that $a,b,i\in K_n$
are such that $c'(a/b)=c'(i)=c(i)$. By translation invariance of $c'$, this is equivalent to asserting that $c'(bi)=c'(a)$, that is $c(bi)=c(a)=c(ab^{o(b)})$ which, again by translation 
invariance, is in turn equivalent to $c'(i)=c(i)=c(ab^{o(b)-1})$. This shows that the extension defined above is indeed the only possible one.
\end{proof}

For $c$ a translation invariant coloring of $K_n$, call its extension to $\hat K_n$ constructed in the proof of lemma \ref{lem:extend} the \emph{canonical extension} of $c$. The 
resemblance between classes in translation invariant colorings, mentioned above and illustrated in figure \ref{fig:4classes}, is even stronger once we pass from the coloring to 
its canonical extension. Doing so eliminates the ``boundary'' of $\mathbb O_n$ given by the coordinate axes. In the absence of such a frame of reference, the color classes are 
entirely indistinguishable from one another.

\begin{fact} \label{fact:extmult}
The canonical extension of a multiplicative coloring is again multiplicative.
\end{fact}

\begin{proof}
Suppose $c$ is the multiplicative $n$-satisfactory coloring determined by the abelian group $([n],\oplus)$. For $a,b\in\hat K_n$, let $m_1,m_2,m_3,m_4\in K_n$ be such 
that $a=m_1/m_2$ and $b=m_3/m_4$. We have that 
 $$ c(ab)=c\left(\frac {m_1m_2}{m_3m_4}\right)=c(m_1 m_2^{o(m_2)-1}m_3 m_4^{o(m_4)-1}), $$ 
by equation (\ref{eq:product}). Since $c$ is multiplicative on $K_n$, the last expression equals $c(m_1m_2^{o(m_2)-1})\oplus c(m_3m_4^{o(m_4)-1})=c(a)\oplus c(b)$, 
as wanted.
\end{proof}

\begin{remark}
With notation as in the proof of fact \ref{fact:extmult}, let $k\in \hat K_n$, and write its prime factorization as 
 $$ k=\prod_{p_i\in P}p_i^{\alpha_i}\cdot\prod_{p_i\in N}p_i^{\alpha_i}, $$
where the $p_i$ are the primes in $[n]$, listed in increasing order, $P$ is the set of primes $p_i$ that appear in $k$ with positive exponent $\alpha_i$, while $N$ is the set of 
such primes present in $k$ with negative exponent. Since $c$ is multiplicative, and using the convention that $c(i)=i$ for $i\in[n]$, we have that 
 $$ c(k)=\bigoplus_{p_i\in P}p_i^{\oplus\alpha_i}\oplus\bigoplus_{p_i\in N}p_i^{\oplus(-\alpha_i)(o(p_i)-1)}=
 \bigoplus_{p_i\in P}p_i^{\oplus\alpha_i}\oplus\bigoplus_{p_i\in N}p_i^{\oplus\alpha_i}=\bigoplus_i p_i^{\oplus\alpha_i}, $$
that is, equation (\ref{eq:multformula}) still holds, regardless of the sign of the exponents.  
\end{remark}

For $l$ finite, a tiling  of $\mathbb Z^l$ by $T$, say $T+B$ where the sum is direct, is \emph{periodic} if and only if there is a finite index subgroup $\Lambda$ of $\mathbb Z^l$ 
such that $B+\Lambda=B$.

In \S\,\ref{subs:tilings} we defined the natural bijective map $t\!:K_n\to\mathbb O_n$ associating to a point $k$ in $K_n$ the point in $\mathbb O_n$ whose coordinates are the 
exponents of the prime factorization of $k$. The same definition gives us an extension of this map to $\hat K_n$ that we again denote by $t$ and is now a bijection with 
$\mathbb Z^{\pi(n)}$. The proof of proposition \ref{proposition:tiling} gives us that if $B$ is the image under $t$ of any of the color classes of the canonical extension of 
$c$, then the sum $T_n+B$ is direct and tiles $\mathbb Z^{\pi(n)}$. 

\begin{lemma} \label{lem:lattice}
Let $c$ be a translation invariant $n$-satisfactory coloring and let $B$ be the image under $t$ of a color class of the canonical extension of $c$. The tiling $T_n+B$ of 
$\mathbb Z^{\pi(n)}$ by $T_n$ is periodic. In particular, this holds for multiplicative colorings.
\end{lemma}

\begin{proof}
Let $p_1<\dots<p_{\pi(n)}$ be the primes less than or equal to $n$. For $i\in[\pi(n)]$ let $\mathbf x_i=t(p_i^{o(p_i)})$, and let $\Lambda=\langle x_i:i\in[\pi(n)]\rangle$, so that 
$\Lambda$ has finite index in $\mathbb Z^{\pi(n)}$. We claim that $B+\Lambda=B$. The point is that if $k\in t^{-1}(\Lambda)$, then $c(k)=c(1)$ and if $k'\in t^{-1}(B)$, then 
$c(kk')=c(k')$, so that $t(k)+t(k')\in B$, that is, $B+\Lambda\subseteq B$. But, since $\mathbf0\in\Lambda$, clearly $B\subseteq B+\Lambda$.
\end{proof}

In fact, we can prove a bit more for multiplicative colorings.

\begin{lemma} \label{lem:lattice2}
Let $c$ be a multiplicative $n$-satisfactory coloring as witnessed by the group $(G,\oplus)$, let $B$ be the image under $t$ of a color class of the canonical extension of $c$, and 
let $\Lambda$ be the image under $t$ of the color class of 1. We have that $\Lambda$ is a finite index subgroup of $\mathbb Z^{\pi(n)}$ and that $B+\Lambda=\Lambda$. In fact, 
$\mathbb Z^{\pi(n)}/\Lambda\cong G$ and the images of the color classes are the cosets of $\Lambda$ in $\mathbb Z^{\pi(n)}$.
\end{lemma}

Note that, in particular, $\mathbb Z^{\pi(n)}=T_n+\Lambda$, and the sum is direct.

\begin{proof}
That $\Lambda$ is a group is clear from the fact that $c$ is multiplicative: first, 
 $$ \mathbf 0=t(1)\in\Lambda; $$ 
if $c(k)=c(1)$, then $c(k^{-1})=c(k)^{\oplus(-1)}=c(1)$, so $\Lambda$ is closed under inverses; and since $c(k_1k_2)=c(k_1)\oplus c(k_2)$, then $\Lambda$ is closed under 
products. That it has finite index in $\mathbb Z^{\pi(n)}$ follows from the fact that it contains the group we denoted by $\Lambda$ in the proof of lemma \ref{lem:lattice}. That 
$B+\Lambda=B$ is exactly as in that proof. 

In fact, let $B$ be the image under $t$ of the color class of $i\in[n]$, and note that $T_n\cap B=\{t(i)\}$. We claim that $B=t(i)+\Lambda$, from which it follows that 
$\mathbb Z^{\pi(n)}/\Lambda\cong G$ and that the images of the color classes are the cosets of $\Lambda$. Clearly $B\supseteq t(i)+\Lambda$, and if $t(k)\in B$, then 
 $$ c(k/i)=c(k)\oplus c(i)^{\oplus(-1)}=c(1), $$ 
so that 
 $$ t(k)=t(i)+(t(k)-t(i))\in t(i)+\Lambda. $$ 
This completes the proof.
\end{proof}

Figure \ref{fig:lattice} illustrates the group $\Lambda$ for the unique 4-satisfactory coloring. 

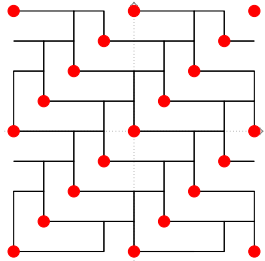
\begin{figure}[ht]
\begin{tikzpicture}[scale=.4,%
  every node/.style={draw,color=red,fill,circle,minimum size=0.5pt,inner sep=1.5pt}]

\draw[very thin,densely dotted,gray,->] (-4.3,0) -- (4.3,0);
\draw[very thin,densely dotted,gray,->] (0,-4.3) -- (0,4.3);

\draw (0,0) -- (3,0) -- (3,1) -- (1,1) -- (1,2) -- (0,2) -- (0,0);
\draw (1,1) -- (4,1) -- (4,2) -- (2,2) -- (2,3) -- (1,3) -- (1,1);
\draw (4,2) -- (2,2) -- (2,4) -- (3,4) -- (3,3) -- (4,3);
\draw (3,4) -- (3,3) -- (4,3);
\draw (-1,-1) -- (2,-1) -- (2,0) -- (0,0) -- (0,1) -- (-1,1) -- (-1,-1);
\draw (-2,-2) -- (1,-2) -- (1,-1) -- (-1,-1) -- (-1,0) -- (-2,0) -- (-2,-2);
\draw (-3,-3) -- (0,-3) -- (0,-2) -- (-2,-2) -- (-2,-1) -- (-3,-1) -- (-3,-3);
\draw (-4,-4) -- (-1,-4) -- (-1,-3) -- (-3,-3) -- (-3,-2) -- (-4,-2) -- (-4,-4);
\draw (-1,-4) -- (-1,-3) -- (0,-3) -- (0,-4) -- (2,-4);
\draw (0,-4) -- (3,-4) -- (3,-3) -- (1,-3) -- (1,-2) -- (0,-2) -- (0,-4);
\draw (1,-3) -- (4,-3) -- (4,-2) -- (2,-2) -- (2,-1) -- (1,-1) -- (1,-3);
\draw (3,-4) -- (3,-3) -- (4,-3) -- (4,-4);
\draw (4,-2) -- (2,-2) -- (2,0) -- (3,0) -- (3,-1) -- (4,-1);
\draw (4,0) -- (4,1) -- (3,1) -- (3,-1) -- (4,-1);
\draw (-4,0) -- (-1,0) -- (-1,1) -- (-3,1) -- (-3,2) -- (-4,2) -- (-4,0);
\draw (-3,1) -- (0,1) -- (0,2) -- (-2,2) -- (-2,3) -- (-3,3) -- (-3,1);
\draw (-2,2) -- (1,2) -- (1,3) -- (-1,3) -- (-1,4) -- (-2,4) -- (-2,2);
\draw (-1,4) -- (-1,3) -- (2,3) -- (2,4) -- (0,4);
\draw (-4,3) -- (-2,3) -- (-2,4) -- (-4,4);
\draw (-4,2) -- (-3,2) -- (-3,3) -- (-4,3);
\draw (-4,-1) -- (-2,-1) -- (-2,0) -- (-4,0);
\draw (-4,-2) -- (-3,-2) -- (-3,-1) -- (-4,-1); 

\node at (-4,-4) {};
\node at (-3,-3) {};
\node at (-2,-2) {};
\node at (-1,-1) {};
\node at (0,0) {};
\node at (1,1) {};
\node at (2,2) {};
\node at (3,3) {};
\node at (4,4) {};
\node at (-4,0) {};
\node at (-3,1) {};
\node at (-2,2) {};
\node at (-1,3) {};
\node at (0,4) {};
\node at (0,-4) {}; 
\node at (1,-3) {};
\node at (2,-2) {};
\node at (3,-1) {};
\node at (4,0) {};
\node at (-4,4) {};
\node at (4,-4) {};

\end{tikzpicture}
\caption{Tiling corresponding to the unique 4-satisfactory coloring: $\mathbb Z^2=T_4+\Lambda$, the sum being direct. The dots indicate the members of the group $\Lambda$.}
\label{fig:lattice}
\end{figure}

\begin{remark} \label{rmk:lattice2}
We can easily address question \ref{q:inverse} for multiplicative colorings via lemma \ref{lem:lattice2}. Suppose $c$ is a multiplicative $n$-satisfactory coloring. Let $\Lambda$ 
be the image under $t$ of the color class of 1, so any color class $B$ is a coset $t(i_0)+\Lambda$. The coloring $c'$ derived from the tiling $T_n+B$ has color classes 
$t(i)+(t(i_0)+\Lambda)$ for $i\in[n]$, but these are precisely the color classes of $c$, as $(c(ii_0):i\in[n])$ is just a permutation of $[n]$.
\end{remark}

We now establish the main result of this subsection.

\begin{theorem} \label{thm:inv}
A translation invariant $n$-satisfactory coloring is multiplicative.
\end{theorem}

\begin{proof}
Let $c$ be translation invariant and $n$-satisfactory. Using the convention that $c(i)=i$ for $i\in[n]$, we need to verify that the map $\oplus\!:[n]\times[n]\to[n]$ given by 
$i\oplus j=c(ij)$ defines a group operation on $[n]$. Clearly, $\oplus$ is commutative and, for any $j$, the sequence $(i\oplus j:i\in[n])$ is a permutation of $[n]$. The issue is 
whether $\oplus$ is associative, that is, whether for $i,j,k\in[n]$ we have 
 $$ c(c(ij)k)=c(ic(jk)). $$

To see this, note that $c(ij)=c(c(ij))$ and therefore, by translation invariance, $c(ijk)=c(c(ij)k)$. Similarly, $c(ijk)=c(ic(jk))$, and we are done.

Finally, we must check that for any $a,b\in K_n$, $c(ab)=c(a)\oplus c(b)$. For this, note that $c(a)=c(c(a))$ and $c(b)=c(c(b))$ so, by translation invariance, 
 $$ c(ab)=c(c(a)b)=c(c(a)c(b))=c(a)\oplus c(b), $$
where the last equality is by definition.
\end{proof}

We close by reminding the reader of the periodic tiling conjecture of Jeffrey Lagarias and Yang Wang, see \cite{LagariasWang96}, the relevant version of which 
in our setting asks whether, given any finite set $T\subseteq \mathbb Z^l$, if $T$ tiles $\mathbb Z^l$, then it also does so periodically. With $l=\pi(n)$ and $T=T_n$, this asks 
whether the existence of an $n$-satisfactory coloring implies the existence of a periodic one. We remark that many of the examples of 6-satisfactory colorings constructed in 
theorem \ref{thm:n6} below are periodic, which shows that a periodic coloring needs not be translation invariant. However, note that in general, no period in the direction of 3 
of the periodic examples exhibited in that result is a factor of 6.

\subsection{A multiplicative coloring of $p^2-p$}

In this short subsection we argue that question \ref{q:problem} has a positive answer for $n=p^2-p$ with $p$ prime, by exhibiting a multiplicative coloring of $n$.

\begin{theorem} \label{thm:p2-p}
For any prime $p$, there is a $G$-coloring of $K_{p^2-p}$, where 
 $$ G=(\mathbb Z/p^2\mathbb Z)^*. $$
\end{theorem}

The group $G$ in the theorem is isomorphic to $\mathbb Z/(p^2-p)\mathbb Z$, but visualizing it as indicated was essential to identifying this coloring.

\begin{proof}
For $i\in[p^2-p]$ such that $p\nmid i$, let $g_i\equiv i \pmod{p^2}$, and for $j\in[p-1]$, let $g_{pj}\equiv p^2-j \pmod {p^2}$, so that $g_1,g_2,\dots,g_{p^2-p}$ lists the 
elements of a reduced residue system modulo $p^2$.

We claim that $g_{ij}=g_ig_j$, whenever $ij\leq n$. If $p\nmid i,j$, then the statement is clear. If $p\nmid i$, but $j=pj'$, then 
 $$ g_ig_j\equiv i(p^2-j')\equiv -ij'\equiv g_{ij}\pmod{p^2}. $$ 
Finally, if $p\mid i,j$, then $ij$ is too large. For $G=(\mathbb Z/p^2\mathbb Z)^*$, this shows that the map 
$g\!:[p^2-p]\to G$ given by $i\mapsto g_i$ is a partial $G$-isomorphism, and therefore there is a $G$-coloring of $K_{p^2-p}$, by theorem \ref{thm:partialisomorphism}.
\end{proof}

\subsection{Multiplicative colorings for $n\le8$} \label{sub:nle8}

In this subsection we list all $G$-sat\-is\-fac\-to\-ry groups for $6 \le |G|\le 8$, thus determining all multiplicative colorings with at most eight colors. The construction is relatively 
simple in each case and proceeds by explicitly exhibiting the multiplication table of all possible $G$-satisfactory groups $([n],\oplus)$ for $6\le n\le 8$. For instance, we use the 
fact that if $([n],\oplus)$ is $G$-satisfactory for some $G$, then 
 $$ \{2\oplus a :  a \in [n], 2a > n\} $$ 
coincides with the set of odd integers in $[n]$, and repeatedly make use of the fact that $\oplus$ must be associative and commutative, that each row and column of the 
multiplication table must be a permutation of $[n]$, and that the order in $([n],\oplus)$ of any element $a$ must divide $n$. In a sense, identifying these colorings is akin to solving 
a Sudoku puzzle. The complete analysis is perhaps a bit too tedious to present in full. We provide essentially all details for $n=8$, the most involved case, and 
sketch the cases $n=6,7$; these sketches can be fleshed out along the same lines as for $n=8$ but more straightforwardly. 

\begin{itemize}
\item
To begin with, the only abelian $G$ of size 6 is $\mathbb Z/6\mathbb Z$, and there are precisely five $\mathbb Z/6\mathbb Z$-colorings of $K_6$.
\end{itemize}

To see this, begin by building the partial multiplication table of a putative $G$-satisfactory group $([6],\oplus)$. The only entries we know originally are those of the form 
$a\oplus b$ with $ab\le 6$ (in which case $a\oplus b=ab$). Note that $\{2\oplus a:a>3\}=\{1,3,5\}$. We consider three cases, according to the value of $a$ with $2\oplus a=1$.

If $2\oplus 5=1$, then $4\oplus 5=2$ and all values of the table are completely determined from the elementary observations two paragraphs above; the result is table 
\ref{table:z6547} below. 

If $2\oplus 6=1$, similarly all values are completely determined; the result is table \ref{table:z613}. 

If $2\oplus 4=1$, then $2\oplus 5=3$ (it cannot be 5 since already $1\oplus5=5$) and $2\oplus 6=5$. Thus, $4\oplus2=2\oplus4=1$, $4\oplus 3=2\oplus 6=5$, 
$4\oplus4=2\oplus(2\oplus4)=2\oplus1=2$, and $4\oplus6\ne6$, so $4\oplus 5=6$ and $4\oplus6=3$. We cannot complete the table just yet, but we do as soon as we choose the 
value of $3\oplus 3$, which must be one of 1, 2, or 4; the results are shown in tables \ref{table:z6103}, \ref{table:z6satisfactory} and \ref{table:z6487}, respectively.

The resulting five colorings are all strongly representable, namely, by $7=1\cdot 6+1$, $13=2\cdot 6+1$, $103=17\cdot 6+1$, $487=81\cdot 6+1$, and $547=91\cdot 6+1$, and 
the tables are listed in the increasing order of these primes. Note that the construction given in the proof of theorem \ref{thm:p2-p} for $p=3$ results in table \ref{table:z6547}.

\begin{table}[ht]
\begin{tabular}{c||c|c|c|c|c|c|} 
$\oplus$&1&2&3&4&5&6\\
\hline \hline
1&1&2&3&4&5&6\\
2&2&4&6&5&3&1\\
3&3&6&4&1&2&5\\
4&4&5&1&3&6&2\\
5&5&3&2&6&1&4\\
6&6&1&5&2&4&3\\
\hline
\end{tabular}\vspace{3mm}
\caption{$\mathbb Z/6\mathbb Z$-coloring strongly represented by $13=2\cdot 6+1$.}
\label{table:z613}
\vspace{-7mm}
\end{table}

\begin{table}[ht]
\begin{tabular}{c||c|c|c|c|c|c|} 
$\oplus$&1&2&3&4&5&6\\
\hline \hline
1&1&2&3&4&5&6\\
2&2&4&6&1&3&5\\
3&3&6&1&5&4&2\\
4&4&1&5&2&6&3\\
5&5&3&4&6&2&1\\
6&6&5&2&3&1&4\\
\hline
\end{tabular}\vspace{3mm}
\caption{$\mathbb Z/6\mathbb Z$-coloring strongly represented by $103=17\cdot 6+1$.}
\label{table:z6103}
\vspace{-7mm}
\end{table}

\begin{table}[ht]
\begin{tabular}{c||c|c|c|c|c|c|} 
$\oplus$&1&2&3&4&5&6\\
\hline \hline
1&1&2&3&4&5&6\\
2&2&4&6&1&3&5\\
3&3&6&4&5&2&1\\
4&4&1&5&2&6&3\\
5&5&3&2&6&1&4\\
6&6&5&1&3&4&2\\
\hline
\end{tabular}\vspace{3mm}
\caption{$\mathbb Z/6\mathbb Z$-coloring strongly represented by $487=81\cdot 6+1$.}
\label{table:z6487}
\vspace{-7mm}
\end{table}

\begin{table}[ht]
\begin{tabular}{c||c|c|c|c|c|c|} 
$\oplus$&1&2&3&4&5&6\\
\hline \hline
1&1&2&3&4&5&6\\
2&2&4&6&3&1&5\\
3&3&6&1&5&4&2\\
4&4&3&5&6&2&1\\
5&5&1&4&2&6&3\\
6&6&5&2&1&3&4\\
\hline
\end{tabular}\vspace{3mm}
\caption{$\mathbb Z/6\mathbb Z$-coloring strongly represented by $547=91\cdot 6+1$; this is also the $(\mathbb Z/9\mathbb Z)^*$-coloring given by the proof of theorem 
\ref{thm:p2-p}.}
\label{table:z6547}
\vspace{-7mm}
\end{table}

\begin{itemize}
\item
The only abelian $G$ of size 7 is $\mathbb Z/7\mathbb Z$, and there are precisely six $\mathbb Z/7\mathbb Z$-colorings of $K_7$. 
\end{itemize}

We argue as before, by starting with the partial multiplication table of a putative $G$-satisfactory group $([7],\oplus)$ and analyzing cases. Note that $2\oplus 4 \ne 1$, since 1 is 
the identity of $([7],\oplus)$ and every member of $\mathbb Z/7\mathbb Z$ other than the identity has order~7. Also, $2\oplus 6 = 4\oplus 3 \ne 3$. Similarly, $2\oplus a \ne a$ for 
$a = 5,7$. Hence, the sequence $(2\oplus a : 4 \le a \le 7)$ is a permutation of the numbers $1,3,5,7$ that does not begin with 1, does not have 5 as its second element, does not 
have 3 as its third element, and does not end with 7.

Although nine permutations satisfy these requirements, $(2\oplus a : 4 \le a \le 7)$ cannot be any of the sequences $(3,7,1,5)$, $(5,1,7,3)$, and $(7,3,5,1)$: the first would imply that 
2 has order 5, while the other two would imply that it has order 4. For instance, if $(2\oplus a : 4 \le a \le 7) = (3,7,1,5)$, then $4\oplus 2=3$ and $6\oplus 2=1$, so $4\oplus 4=6$ and 
$6\oplus 2=2^{\oplus 5}$.

The remaining six sequences lead indeed to the multiplication table of a $\mathbb Z/7\mathbb Z$-satisfactory group. 

They are all strongly representable, by 
$659=94\cdot 7+1$, $1429=204\cdot 7+1$, $2087=298\cdot 7+1$, $3557=508\cdot7+1$, $17431=2490\cdot7+1$, and $21911=3130\cdot 7+1$, and are described by tables 
\ref{table:z7659}, \ref{table:z71429}, \ref{table:z72087}, \ref{table:z73557}, \ref{table:z717431}, and \ref{table:z721911}, respectively.

\begin{table}[ht]
\begin{tabular}{c||c|c|c|c|c|c|c|} 
$\oplus$&1&2&3&4&5&6&7\\
\hline \hline
1&1&2&3&4&5&6&7\\
2&2&4&6&5&3&7&1\\
3&3&6&2&7&1&4&5\\
4&4&5&7&3&6&1&2\\
5&5&3&1&6&7&2&4\\
6&6&7&4&1&2&5&3\\
7&7&1&5&2&4&3&6\\
\hline
\end{tabular}\vspace{3mm}
\caption{$\mathbb Z/7\mathbb Z$-coloring strongly represented by $659$.}
\label{table:z7659}
\vspace{-7mm}
\end{table}

\begin{table}[ht]
\begin{tabular}{c||c|c|c|c|c|c|c|} 
$\oplus$&1&2&3&4&5&6&7\\
\hline \hline
1&1&2&3&4&5&6&7\\
2&2&4&6&3&7&5&1\\
3&3&6&7&5&2&1&4\\
4&4&3&5&6&1&7&2\\
5&5&7&2&1&3&4&6\\
6&6&5&1&7&4&2&3\\
7&7&1&4&2&6&3&5\\
\hline
\end{tabular}\vspace{3mm}
\caption{$\mathbb Z/7\mathbb Z$-coloring strongly represented by $1429$.}
\label{table:z71429}
\vspace{-7mm}
\end{table}

\begin{table}[ht]
\begin{tabular}{c||c|c|c|c|c|c|c|} 
$\oplus$&1&2&3&4&5&6&7\\
\hline \hline
1&1&2&3&4&5&6&7\\
2&2&4&6&3&1&7&5\\
3&3&6&5&7&4&1&2\\
4&4&3&7&6&2&5&1\\
5&5&1&4&2&7&3&6\\
6&6&7&1&5&3&2&4\\
7&7&5&2&1&6&4&3\\
\hline
\end{tabular}\vspace{3mm}
\caption{$\mathbb Z/7\mathbb Z$-coloring strongly represented by $2087$.}
\label{table:z72087}
\vspace{-7mm}
\end{table}

\begin{table}[ht]
\begin{tabular}{c||c|c|c|c|c|c|c|} 
$\oplus$&1&2&3&4&5&6&7\\
\hline \hline
1&1&2&3&4&5&6&7\\
2&2&4&6&7&1&5&3\\
3&3&6&2&5&7&4&1\\
4&4&7&5&3&2&1&6\\
5&5&1&7&2&6&3&4\\
6&6&5&4&1&3&7&2\\
7&7&3&1&6&4&2&5\\
\hline
\end{tabular}\vspace{3mm}
\caption{$\mathbb Z/7\mathbb Z$-coloring strongly represented by $3557$.}
\label{table:z73557}
\vspace{-7mm}
\end{table}

\begin{table}[ht]
\begin{tabular}{c||c|c|c|c|c|c|c|} 
$\oplus$&1&2&3&4&5&6&7\\
\hline \hline
1&1&2&3&4&5&6&7\\
2&2&4&6&7&3&1&5\\
3&3&6&7&1&4&5&2\\
4&4&7&1&5&6&2&3\\
5&5&3&4&6&2&7&1\\
6&6&1&5&2&7&3&4\\
7&7&5&2&3&1&4&6\\
\hline
\end{tabular}\vspace{3mm}
\caption{$\mathbb Z/7\mathbb Z$-coloring strongly represented by $17431$.}
\label{table:z717431}
\vspace{-7mm}
\end{table}

\begin{table}[ht]
\begin{tabular}{c||c|c|c|c|c|c|c|} 
$\oplus$&1&2&3&4&5&6&7\\
\hline \hline
1&1&2&3&4&5&6&7\\
2&2&4&6&5&7&1&3\\
3&3&6&5&1&2&7&4\\
4&4&5&1&7&3&2&6\\
5&5&7&2&3&6&4&1\\
6&6&1&7&2&4&3&5\\
7&7&3&4&6&1&5&2\\
\hline
\end{tabular}\vspace{3mm}
\caption{$\mathbb Z/7\mathbb Z$-coloring strongly represented by $21911$.}
\label{table:z721911}
\vspace{-7mm}
\end{table}

\begin{itemize}
\item
There are three abelian groups of order 8, namely $\mathbb Z/8\mathbb Z$, $\mathbb Z/2\mathbb Z\times\mathbb Z/4\mathbb Z$, and 
$\mathbb Z/2\mathbb Z\times\mathbb Z/2\mathbb Z\times\mathbb Z/2\mathbb Z$. There are no 
$\mathbb Z/2\mathbb Z\times\mathbb Z/2\mathbb Z\times\mathbb Z/2\mathbb Z$-satisfactory groups, for the simple reason that $2\oplus 2=4\ne1$. There are precisely four 
$\mathbb Z/2\mathbb Z\times\mathbb Z/4\mathbb Z$-colorings of $K_8$, and four $\mathbb Z/8\mathbb Z$-colorings admitting strong representatives. There are also six 
additional $\mathbb Z/8\mathbb Z$-colorings that do not admit strong representatives.
\end{itemize}

We provide some details. Consider first the sequence $(2\oplus a: 5\le a\le 8)$, noting that it must be a permutation of the numbers $1,3,5,7$ that does not begin with 5 and does 
not have 7 as a third element. Moreover, 3 cannot be the second element, since $2\oplus 6 = 3$ would imply that $4\oplus 6=6$. This means that the sequence must be one of the 
following: $(1,5,3,7)$, $(1,7,3,5)$, $(1,7,5,3)$, $(3,1,5,7)$, $(3,5,1,7)$, $(3,7,1,5)$, $(3,7,5,1)$, $(7,1,3,5)$, $(7,1,5,3)$, $(7,5,1,3)$, or $(7,5,3,1)$.

However, $(1,7,3,5)$, $(3,5,1,7)$, and $(7,1,5,3)$ are not possible.

\begin{enumerate}
\item[a.] 
Consider $(1,7,3,5)$: if $2\oplus 6=7$ and $2\oplus 7=3$, then $8\oplus 3=3$.
\item[b.]
Consider $(3,5,1,7)$: if $2\oplus 5=3$ and $2\oplus 6=5$, then again $8\oplus 3=3$.
\item[c.]
Consider $(7,1,5,3)$: if $2\oplus 6=1$ and $2\oplus 8=3=4\oplus 4$, then $4^{\oplus 3} =3\oplus 4=1$, against Lagrange's theorem.
\end{enumerate}

Of the remaining eight sequences, six of them determine $\oplus$ uniquely as shown below. In all cases, the resulting group is $\mathbb Z/8\mathbb Z$-satisfactory and 2 is a 
generator. As we will see below, none of the associated colorings is strongly representable.

For $(1, 5, 3, 7)$, see table \ref{table:z81537}; for $(1, 7, 5, 3)$, see table \ref{table:z81753}; for $(3, 1, 5, 7)$, see table \ref{table:z83157}; for $(3, 7, 1, 5)$, see table 
\ref{table:z83715}; for $(7, 1, 3, 5)$, see table \ref{table:z87135}; and for $(7, 5, 1, 3)$, see table \ref{table:z87513}.

\begin{table}[ht]
\begin{tabular}{c||c|c|c|c|c|c|c|c|} 
$\oplus$&1&2&3&4&5&6&7&8\\
\hline \hline
1&1&2&3&4&5&6&7&8\\
2&2&4&6&8&1&5&3&7\\
3&3&6&4&5&7&8&2&1\\
4&4&8&5&7&2&1&6&3\\
5&5&1&7&2&6&3&8&4\\
6&6&5&8&1&3&7&4&2\\
7&7&3&2&6&8&4&1&5\\
8&8&7&1&3&4&2&5&6\\
\hline
\end{tabular}\vspace{3mm}
\caption{$\mathbb Z/8\mathbb Z$-coloring corresponding to the sequence $(1,5,3,7)$.}
\label{table:z81537}
\vspace{-7mm}
\end{table}

\begin{table}[ht]
\begin{tabular}{c||c|c|c|c|c|c|c|c|} 
$\oplus$&1&2&3&4&5&6&7&8\\
\hline \hline
1&1&2&3&4&5&6&7&8\\
2&2&4&6&8&1&7&5&3\\
3&3&6&1&7&8&2&4&5\\
4&4&8&7&3&2&5&1&6\\
5&5&1&8&2&7&3&6&4\\
6&6&7&2&5&3&4&8&1\\
7&7&5&4&1&6&8&3&2\\
8&8&3&5&6&4&1&2&7\\
\hline
\end{tabular}\vspace{3mm}
\caption{$\mathbb Z/8\mathbb Z$-coloring corresponding to the sequence $(1,7,5,3)$.}
\label{table:z81753}
\vspace{-7mm}
\end{table}

\begin{table}[ht]
\begin{tabular}{c||c|c|c|c|c|c|c|c|} 
$\oplus$&1&2&3&4&5&6&7&8\\
\hline \hline
1&1&2&3&4&5&6&7&8\\
2&2&4&6&8&3&1&5&7\\
3&3&6&7&1&8&5&4&2\\
4&4&8&1&7&6&2&3&5\\
5&5&3&8&6&4&7&2&1\\
6&6&1&5&2&7&3&8&4\\
7&7&5&4&3&2&8&1&6\\
8&8&7&2&5&1&4&6&3\\
\hline
\end{tabular}\vspace{3mm}
\caption{$\mathbb Z/8\mathbb Z$-coloring corresponding to the sequence $(3,1,5,7)$.}
\label{table:z83157}
\vspace{-7mm}
\end{table}

\begin{table}[ht]
\begin{tabular}{c||c|c|c|c|c|c|c|c|} 
$\oplus$&1&2&3&4&5&6&7&8\\
\hline \hline
1&1&2&3&4&5&6&7&8\\
2&2&4&6&8&3&7&1&5\\
3&3&6&4&7&2&8&5&1\\
4&4&8&7&5&6&1&2&3\\
5&5&3&2&6&1&4&8&7\\
6&6&7&8&1&4&5&3&2\\
7&7&1&5&2&8&3&6&4\\
8&8&5&1&3&7&2&4&6\\
\hline
\end{tabular}\vspace{3mm}
\caption{$\mathbb Z/8\mathbb Z$-coloring corresponding to the sequence $(3,7,1,5)$.}
\label{table:z83715}
\vspace{-7mm}
\end{table}

\begin{table}[ht]
\begin{tabular}{c||c|c|c|c|c|c|c|c|} 
$\oplus$&1&2&3&4&5&6&7&8\\
\hline \hline
1&1&2&3&4&5&6&7&8\\
2&2&4&6&8&7&1&3&5\\
3&3&6&5&1&4&7&8&2\\
4&4&8&1&5&3&2&6&7\\
5&5&7&4&3&1&8&2&6\\
6&6&1&7&2&8&3&5&4\\
7&7&3&8&6&2&5&4&1\\
8&8&5&2&7&6&4&1&3\\
\hline
\end{tabular}\vspace{3mm}
\caption{$\mathbb Z/8\mathbb Z$-coloring corresponding to the sequence $(7,1,3,5)$.}
\label{table:z87135}
\vspace{-7mm}
\end{table}

\begin{table}[ht]
\begin{tabular}{c||c|c|c|c|c|c|c|c|} 
$\oplus$&1&2&3&4&5&6&7&8\\
\hline \hline
1&1&2&3&4&5&6&7&8\\
2&2&4&6&8&7&5&1&3\\
3&3&6&1&5&4&2&8&7\\
4&4&8&5&3&1&7&2&6\\
5&5&7&4&1&3&8&6&2\\
6&6&5&2&7&8&4&3&1\\
7&7&1&8&2&6&3&5&4\\
8&8&3&7&6&2&1&4&5\\
\hline
\end{tabular}\vspace{3mm}
\caption{$\mathbb Z/8\mathbb Z$-coloring corresponding to the sequence $(7,5,1,3)$.}
\label{table:z87513}
\vspace{-7mm}
\end{table}

The remaining two sequences do not contain sufficient information to determine $\oplus$. What they determine of the multiplication table is shown in table \ref{table:partial3751} 
for $(3,7,5,1)$ and in table \ref{table:partial7531} for $(7,5,3,1)$.

\begin{table}[ht]
\begin{tabular}{c||c|c|c|c|c|c|c|c|} 
$\oplus$&1&2&3&4&5&6&7&8\\
\hline \hline
1&1&2&3&4&5&6&7&8\\
2&2&4&6&8&3&7&5&1\\
3&3&6&  &7&  &  &  &5\\
4&4&8&7&1&6&5&3&2\\
5&5&3&  &6&  &  &  &7\\
6&6&7&  &5&  &  &  &3\\
7&7&5&  &3&  &  &  &6\\
8&8&1&5&2&7&3&6&4\\
\hline
\end{tabular}\vspace{3mm}
\caption{The partial multiplication table determined by the sequence $(3,7,5,1)$.}
\label{table:partial3751}
\vspace{-7mm}
\end{table}

\begin{table}[ht]
\begin{tabular}{c||c|c|c|c|c|c|c|c|} 
$\oplus$&1&2&3&4&5&6&7&8\\
\hline \hline
1&1&2&3&4&5&6&7&8\\
2&2&4&6&8&7&5&3&1\\
3&3&6&  &5&  &  &  &7\\
4&4&8&5&1&3&7&6&2\\
5&5&7&  &3&  &  &  &6\\
6&6&5&  &7&  &  &  &3\\
7&7&3&  &6&  &  &  &5\\
8&8&1&7&2&6&3&5&4\\
\hline
\end{tabular}\vspace{3mm}
\caption{The partial multiplication table determined by the sequence $(7,5,3,1)$.}
\label{table:partial7531}
\vspace{-7mm}
\end{table}

Note that in both cases we have $2^{\oplus 4} = 1$. We conclude by observing that the value of $3\oplus 3 = a$ completely determines the tables, and any of the four options for 
$a$ (namely, 1, 2, 4, or 8) is possible, see tables \ref{table:3751} and \ref{table:7531}.

\begin{table}[ht]
\begin{tabular}{c||c|c|c|c|c|c|c|c|} 
$\oplus$&1&2&3&4&5&6&7&8\\
\hline \hline
1&1&2&3&4&5&6&7&8\\
2&2&4&6&8&3&7&5&1\\
3&3&6&$a$&7&$8\oplus a$&$2\oplus a$&$4\oplus a$&5\\
4&4&8&7&1&6&5&3&2\\
5&5&3&$8\oplus a$&6&$4\oplus a$&$a$&$2\oplus a$&7\\
6&6&7&$2\oplus a$&5&$a$&$4\oplus a$&$8\oplus a$&3\\
7&7&5&$4\oplus a$&3&$2\oplus a$&$8\oplus a$&$a$&6\\
8&8&1&5&2&7&3&6&4\\
\hline
\end{tabular}\vspace{3mm}
\caption{Group corresponding to the sequence $(3,7,5,1)$ with $a=1$, 2, 4, or 8.}
\label{table:3751}
\vspace{-7mm}
\end{table}

\begin{table}[ht]
\begin{tabular}{c||c|c|c|c|c|c|c|c|} 
$\oplus$&1&2&3&4&5&6&7&8\\
\hline \hline
1&1&2&3&4&5&6&7&8\\
2&2&4&6&8&7&5&3&1\\
3&3&6&$a$&5&$4\oplus a$&$2\oplus a$&$8\oplus a$&7\\
4&4&8&5&1&3&7&6&2\\
5&5&7&$4\oplus a$&3&$a$&$8\oplus a$&$2\oplus a$&6\\
6&6&5&$2\oplus a$&7&$8\oplus a$&$4\oplus a$&$a$&3\\
7&7&3&$8\oplus a$&6&$2\oplus a$&$a$&$4\oplus a$&5\\
8&8&1&7&2&6&3&5&4\\
\hline
\end{tabular}\vspace{3mm}
\caption{Group corresponding to the sequence $(7,5,3,1)$ with $a=1$, 2, 4, or 8.}
\label{table:7531}
\vspace{-7mm}
\end{table}

In both cases, we obtain $\mathbb Z/8\mathbb Z$-satisfactory groups if and only if $a = 2$ or 8. The associated colorings admit strong representatives, as follows: for the sequence 
$(3, 7, 5, 1)$, if $a = 2$, take $5417 = 677\cdot 8+1$, and if $a = 8$, take $117017 = 14627\cdot 8+1$. For the sequence $(7,5,3,1)$, if $a=2$, take $3617=452\cdot 8+1$, and if 
$a=8$, take $17=2\cdot 8+1$.

If instead we let $a = 1$ or 4, we obtain $\mathbb Z/2\mathbb Z \times\mathbb Z/4\mathbb Z$-satisfactory groups. If $a = 1$, in both cases, the unique group homomorphism that 
maps 2 to $(0,1)$ and 3 to $(1,0)$ is an isomorphism between $([8],\oplus)$ and $\mathbb Z/2\mathbb Z \times\mathbb Z/4\mathbb Z$. If $a = 4$ and the sequence is $(3, 7, 5, 1)$, 
the corresponding isomorphism is obtained by considering the homomorphism that maps 2 to $(0,1)$ and 5 to $(1,0)$. If $a = 4$ and the sequence is $(7,5,3,1)$, consider instead 
the homomorphism that maps 2 to $(0,1)$ and 7 to $(1,0)$.

Finally, we argue that the colorings associated with the first six $\mathbb Z/8\mathbb Z$-satisfactory groups we listed are not strongly representable. For this, simply note that if they 
were, any strong representative must be of the form $p = 8k + 1$, so $2^{4k}\equiv 1\pmod p$. But $2^{4k} = (2^4)^k$, so the corresponding coloring $c$ must satisfy 
 $$ c(2\oplus 8) = c(2^4) = 1 = c(1), $$ 
that is, we must have $2 \oplus 8 = 1$. 

\section{Groupless numbers and nonmultiplicative colorings} \label{sec:groupless}

In this section we recall results proving that not all numbers admit multiplicative colorings, and argue that not all satisfactory colorings are multiplicative.

\subsection{Groupless numbers} \label{subsec:groupless}

\begin{theorem}[Forcade-Pollington {\cite{ForcadePollington90}}] \label{thm:195}
There are positive integers $n$ for which no multiplicative colorings exist. The smallest such $n$ is $n=195$.
\end{theorem}

In particular, this refutes the natural conjectures that a careful probabilistic argument or even a more careful appeal to Chebator\"ev's theorem than the one in 
\S\,\ref{subsec:density} would prove the existence of multiplicative colorings for all $n$.\footnote{See for instance 
\href{https://mathoverflow.net/q/26358/}{https://mathoverflow.net/q/26358/}}

The motivation for this result was Graham's conjecture, discussed in \S\,\ref{subsec:graham}. The proof follows from the work initiated by R. W. Forcade, J. W. Lamoreaux, and A. D. 
Pollington when they posed the following question in 1986 \cite{ForcadeLamoreauxPollington86}.

\begin{question} \label{q:forcade}
Is it possible, changing only those products that exceed $n$, to make the set $[n]$ into a multiplicative group?
\end{question}

In our terminology, this is asking whether $G$-satisfactory groups exist for all values of $n$. In their article, they conjecture that the answer to question \ref{q:forcade} is affirmative. 
In their discussion, they also ask (in different terms) whether strong representatives exist for all values of $n$.

In 1990, perhaps surprisingly, Forcade and Pollington answered question \ref{q:forcade} negatively \cite{ForcadePollington90}. To do so they employed an exhaustive search 
algorithm that identified $195$ as the least value of $n$ for which there are no $G$-satisfactory groups.

Say that $n$ is \emph{groupless} if it admits no $G$-satisfactory group. Table \ref{table:groupless} lists all groupless $n \le 500$. The data for the table was supplied by Rodney 
Forcade. This is sequence OEIS A204811 in the Online Encyclopedia of Integer Se\-quen\-ces\footnote{See \href{http://oeis.org/A204811}{http://oeis.org/A204811}}.

\begin{table}[ht]
\begin{tabular}{cccccccccc} 
195 & 248 & 279 & 311 & 337 & 367 & 394 & 423 & 451 & 480 \\
205 & 252 & 283 & 313 & 339 & 368 & 395 & 424 & 452 & 481 \\
208 & 253 & 286 & 314 & 340 & 370 & 397 & 425 & 454 & 482 \\
211 & 255 & 287 & 317  & 343 & 373 & 399 & 427 & 457 & 484 \\
212 & 257 & 289 & 318 & 344 & 374 & 401 & 433 & 458 & 487 \\ 
214 & 258 & 290 & 319 & 347 & 376 & 402 & 434 & 461 & 489 \\
217 & 259 & 291 & 322 & 349 & 377 & 403 & 435 & 463 & 492 \\
218 & 263 & 294 & 324 & 351 & 379 & 406 & 436 & 465 & 493 \\
220 & 264 & 295 & 325 & 353 & 381 & 407 & 437 & 467 & 494 \\
227 & 265 & 297 & 327 & 355 & 383 & 409 & 439 & 469 & 496 \\
229 & 266 & 298 & 328 & 356 & 385 & 412 & 444 & 471 & 497 \\
235 & 267 & 301 & 331 & 357 & 387 & 415 & 445 & 472 & 499 \\
242 & 269 & 302 & 332 & 361 & 389 & 416 & 446 & 474 & 500 \\
244 & 271 & 304 & 333 & 362 & 390 & 417 & 447 & 475 \\
246 &274 & 305 & 334 & 364 & 391 & 421 & 449 & 477 \\
247 & 275 & 307 & 335 & 365 & 392 & 422 & 450 & 479
\end{tabular}\vspace{3mm}
\caption{Groupless $n\le 500$.}
\label{table:groupless}
\vspace{-7mm}
\end{table}

In \cite{BlackburnMcKee12}, S. R. Blackburn and J. F. McKee study partial $\mathbb Z/n\mathbb Z$-isomorphisms, in the context of constructing what they call $k$-radius 
sequences over a finite alphabet. In their paper, our partial isomorphisms are dubbed \emph{bijective logarithms of length $n$} or, simply, \emph{logarithms of length $n$}. 
Several references where they are studied are provided in their section 5.1. Their theorem 5.1, for which they further refer to \cite[theorem 3]{Mills63}, which makes essential 
use of Chebotar\"ev's theorem, seems particularly relevant to our question \ref{qu:natdensity}.

In \cite[section 5.2]{BlackburnMcKee12}, question \ref{q:problem} is considered (independently), in the language of tilings of powers of $\mathbb Z$. In 
\cite[section 9.2]{BlackburnMcKee12}, Blackburn and McKee discuss the number of partial isomorphisms for a given $n$ and present a table listing those $n \le 300$ that do 
not admit partial $\mathbb Z/n\mathbb Z$-isomorphisms (their table coincides with the beginning of our table \ref{table:groupless}). They further ask, motivated by numerical 
evidence, whether it is the case that if $n$ is large enough, then there is a partial $\mathbb Z/n\mathbb Z$-isomorphism if and only if either $n+1$ or $2n+1$ is prime. Note, 
however, that this turns out to be false, by theorem \ref{thm:p2-p}.

Nevertheless, the suggestion from the computations of \cite{BlackburnMcKee12} is that the set of groupless $n$ is large. It is thus natural to ask the following, as suggested by the 
referee.

\begin{question} \label{q:ref}
Is the set of groupless $n$ infinite, or even of natural density 1?
\end{question}

\subsection{Nonmultiplicative 6-satisfactory colorings} \label{subs:n6}

We finish the paper by proving that nonmultiplicative colorings exist, and in fact there may be many of them. Here we treat the case of $n=6$ colors. In \S\,\ref{subs:n8} we 
consider $n=8$.

Recall that question \ref{q:equation} asks, for a given $n$-satisfactory coloring $c$ and $k\in K_n$, to find all $n$-satisfactory colorings $d$ with $d_k=c$ (meaning that two 
numbers $m,m'$ receive the same color under $c$ if and only if $km$ and $km'$ receive the same color under $d$), that is, all extensions of a given coloring (or, in the sense 
introduced just before proposition \ref{proposition:tiling}, a given essential tiling) of $\mathbb O_n$ to one of $\mathbb O_n-t(k)$.

\begin{theorem} \label{thm:nonmultiplicative}
Let $c$ be the 6-satisfactory coloring with strong representative 7, that is, $c(m)=(m\bmod 7)$ for $m\in K_6$. There are exactly six 6-satisfactory colorings $d$ such that 
$d_5=c$. In particular, there are nonmultiplicative 6-satisfactory colorings. 
\end{theorem}

The coloring $c$ is particularly nicely behaved, which simplifies the analysis that follows. As a $\mathbb Z/6\mathbb Z$-coloring, it is determined by table 
\ref{table:z6satisfactory}. The reader may consider just as well the 6-satisfactory coloring $c'$ with strong representative 487, that is, $c'(m)=(m^{81}\bmod{487})$, for 
which the result also holds with a similar, but slightly more involved, geometric analysis than the one we suggest below for $c$. As a $\mathbb Z/6\mathbb Z$-coloring, 
$c'$ is determined by table \ref{table:z6487}.

\begin{proof}
Note that $c$ is multiplicative, $c(9)=2$ and $c(3^5)=5$, so that 
\begin{equation} \label{eqn:col-6}
c(2^\alpha3^\beta5^\gamma)=(3^{2\alpha+\beta+5\gamma\bmod 6}\bmod7)
\end{equation}
and, in particular, for fixed $\beta,\gamma$, $\{c(2^\alpha 3^\beta 5^\gamma):\alpha\in\mathbb N\}$ is either $\{1,2,4\}$ or $\{3,5,6\}$, the values alternating depending on the 
parity of $\beta+\gamma$. Indeed, $\beta+\gamma$ and $\beta+5\gamma$ have the same parity, and the powers of 3 are given modulo 7 by $1,3,2,6,4,5,1,3,\dots$. Moreover, 
since $c(2^3)=1$ while $c(i)=i$ for $i=1,2,4$, the value of $c(2^\alpha 3^\beta 5^\gamma)$ is periodic in $\alpha$ with period 3, see figure \ref{figure:coloringc}.

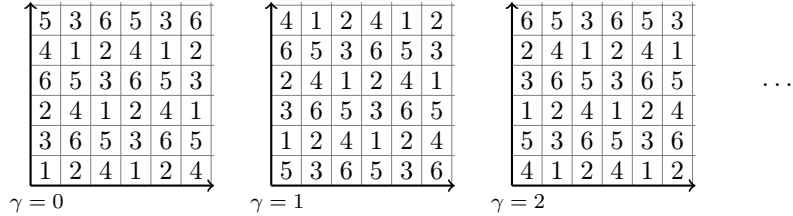
\begin{figure}[ht]
\begin{tikzpicture}[scale=.4]

\draw[step=1cm,gray,very thin] (0,0) grid (6.1,6.1); 
\draw[thick,->] (0,0) -- (6.1,0);
\draw[thick,->] (0,0) -- (0,6.1);
\draw (-1,-1.2) node[anchor=south west]{\footnotesize$\gamma=0$};

\draw[step=1cm,gray,very thin] (8,0) grid (14.1,6.1);
\draw[thick,->] (8,0) -- (14.1,0);
\draw[thick,->] (8,0) -- (8,6.1);
\draw (7,-1.2) node[anchor=south west]{\footnotesize$\gamma=1$};

\draw[step=1cm,gray,very thin] (16,0) grid (22.1,6.1);
\draw[thick,->] (16,0) -- (22.1,0);
\draw[thick,->] (16,0) -- (16,6.1);
\draw (15,-1.2) node[anchor=south west]{\footnotesize$\gamma=2$};

\draw (24,3) node[anchor=south west]{$\dots$};

    \setcounter{row}{1}
    \setrow {5}{3}{6}{5}{3}{6}  
    \setrow {4}{1}{2}{4}{1}{2}  
    \setrow {6}{5}{3}{6}{5}{3}
    \setrow {2}{4}{1}{2}{4}{1}  
    \setrow {3}{6}{5}{3}{6}{5}  
    \setrow {1}{2}{4}{1}{2}{4}
    
\setcounter{row}{1}
\setrowm {4}{1}{2}{4}{1}{2}{9}
\setrowm{6}{5}{3}{6}{5}{3}{9}
\setrowm{2}{4}{1}{2}{4}{1}{9}
\setrowm{3}{6}{5}{3}{6}{5}{9}
\setrowm{1}{2}{4}{1}{2}{4}{9}
\setrowm{5}{3}{6}{5}{3}{6}{9}

\setcounter{row}{1}
\setrowm{6}{5}{3}{6}{5}{3}{17}
\setrowm{2}{4}{1}{2}{4}{1}{17}
\setrowm{3}{6}{5}{3}{6}{5}{17}
\setrowm{1}{2}{4}{1}{2}{4}{17}
\setrowm{5}{3}{6}{5}{3}{6}{17}
\setrowm{4}{1}{2}{4}{1}{2}{17}

\end{tikzpicture}
\caption{Tiling of $\mathbb O_6$ (in the sense of figures \ref{figure:3tiling}, \ref{figure:4tiling}) corresponding to the 6-satisfactory coloring $c$.}
\label{figure:coloringc}
\end{figure}

We  proceed to verify the claim that there are precisely six 6-satisfactory colorings $d$ with $d_5=c$. Note that if $d$ is multiplicative, then $d_k=d$ for all $k$. In particular, 
five of these colorings are nonmultiplicative.

We think of the problem of finding $d$ as that of extending the coloring in figure \ref{figure:coloringc} one extra layer down in the axis corresponding to the prime 5, so that 
$\operatorname{dom}(d)=\frac15\cdot K_6$ and also the points in the 2-dimensional grid corresponding to $\gamma=-1$ are colored. The condition we should maintain is that 
if a copy of the 3-dimensional polyomino $T_6$ (depicted in figure \ref{fig:t6}) is completely contained in the extended orthant, then it must contain all colors. 

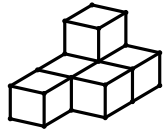
\begin{figure}[ht]
\begin{tikzpicture}[line cap=round,line join=round,>=triangle 45,x=1.0cm,y=1.0cm, scale = .5]

\draw [line width=1.2pt] (0.,0.)-- (0.,-0.8722001507199982);
\draw [line width=1.2pt] (0.,0.)-- (0.9,-0.2);
\draw [line width=1.2pt] (0.,0.)-- (0.724532936614435,0.4280339547378501);
\draw [line width=1.2pt] (0.,-0.8722001507199982)-- (0.9,-1.0722001507199983);
\draw [line width=1.2pt] (0.9,-1.0722001507199983)-- (0.9,-0.2);
\draw [line width=1.2pt] (0.9,-0.2)-- (1.624532936614435,0.22803395473785013);
\draw [line width=1.2pt] (1.624532936614435,0.22803395473785013)-- (1.6245329366144348,-0.6441661959821483);
\draw [line width=1.2pt] (1.6245329366144348,-0.6441661959821483)-- (0.9,-1.0722001507199983);
\draw [line width=1.2pt] (1.6245329366144348,-0.6441661959821483)-- (2.524532936614435,-0.8441661959821484);
\draw [line width=1.2pt] (2.524532936614435,-0.8441661959821484)-- (2.524532936614435,0.028033954737850107);
\draw [line width=1.2pt] (2.524532936614435,0.028033954737850107)-- (1.624532936614435,0.22803395473785013);
\draw [line width=1.2pt] (1.624532936614435,0.22803395473785013)-- (2.3490658732288696,0.6560679094757002);
\draw [line width=1.2pt] (2.3490658732288696,0.6560679094757002)-- (3.2490658732288695,0.4560679094757001);
\draw [line width=1.2pt] (3.2490658732288695,0.4560679094757001)-- (3.249065873228869,-0.4161322412442986);
\draw [line width=1.2pt] (3.249065873228869,-0.4161322412442986)-- (2.524532936614435,-0.8441661959821484);
\draw [line width=1.2pt] (2.524532936614435,0.028033954737850107)-- (3.2490658732288695,0.4560679094757001);
\draw [line width=1.2pt] (3.2490658732288695,0.4560679094757001)-- (3.973598809843305,0.8841018642135503);
\draw [line width=1.2pt] (3.973598809843305,0.8841018642135503)-- (3.973598809843304,0.011901713493551225);
\draw [line width=1.2pt] (3.973598809843304,0.011901713493551225)-- (3.249065873228869,-0.4161322412442986);
\draw [line width=1.2pt] (3.0735988098433045,1.0841018642135505)-- (3.973598809843305,0.8841018642135503);
\draw [line width=1.2pt] (2.3490658732288696,0.6560679094757002)-- (3.0735988098433045,1.0841018642135505);
\draw [line width=1.2pt] (0.724532936614435,0.4280339547378501)-- (1.44906587322887,0.8560679094757002);
\draw [line width=1.2pt] (1.44906587322887,0.8560679094757002)-- (2.3490658732288696,0.6560679094757002);
\draw [line width=1.2pt] (0.724532936614435,0.4280339547378501)-- (1.624532936614435,0.22803395473785013);
\draw [line width=1.2pt] (1.44906587322887,0.8560679094757002)-- (1.44906587322887,1.7282680601956997);
\draw [line width=1.2pt] (1.44906587322887,1.7282680601956997)-- (2.173598809843304,2.1563020149335497);
\draw [line width=1.2pt] (2.173598809843304,2.1563020149335497)-- (3.0735988098433045,1.9563020149335495);
\draw [line width=1.2pt] (3.0735988098433045,1.9563020149335495)-- (2.34906587322887,1.5282680601956997);
\draw [line width=1.2pt] (2.34906587322887,1.5282680601956997)-- (1.44906587322887,1.7282680601956997);
\draw [line width=1.2pt] (2.34906587322887,1.5282680601956997)-- (2.3490658732288696,0.6560679094757002);
\draw [line width=1.2pt] (3.0735988098433045,1.9563020149335495)-- (3.0735988098433045,1.0841018642135505);

\begin{scriptsize}
\draw [fill=black] (0.,0.) circle (1.5pt);
\draw [fill=black] (0.,-0.8722001507199982) circle (1.5pt);
\draw [fill=black] (0.724532936614435,0.4280339547378501) circle (1.5pt);
\draw [fill=black] (0.9,-0.2) circle (1.5pt);
\draw [fill=black] (0.9,-1.0722001507199983) circle (1.5pt);
\draw [fill=black] (1.624532936614435,0.22803395473785013) circle (1.5pt);
\draw [fill=black] (1.6245329366144348,-0.6441661959821483) circle (1.5pt);
\draw [fill=black] (1.44906587322887,0.8560679094757002) circle (1.5pt);
\draw [fill=black] (2.3490658732288696,0.6560679094757002) circle (1.5pt);
\draw [fill=black] (3.2490658732288695,0.4560679094757001) circle (1.5pt);
\draw [fill=black] (2.524532936614435,0.028033954737850107) circle (1.5pt);
\draw [fill=black] (2.524532936614435,-0.8441661959821484) circle (1.5pt);
\draw [fill=black] (3.249065873228869,-0.4161322412442986) circle (1.5pt);
\draw [fill=black] (3.0735988098433045,1.0841018642135505) circle (1.5pt);
\draw [fill=black] (3.973598809843305,0.8841018642135503) circle (1.5pt);
\draw [fill=black] (3.973598809843304,0.011901713493551225) circle (1.5pt);
\draw [fill=black] (3.0735988098433045,1.9563020149335495) circle (1.5pt);
\draw [fill=black] (2.34906587322887,1.5282680601956997) circle (1.5pt);
\draw [fill=black] (1.44906587322887,1.7282680601956997) circle (1.5pt);
\draw [fill=black] (2.173598809843304,2.1563020149335497) circle (1.5pt);
\end{scriptsize}

\end{tikzpicture}
\caption{The 3-dimensional polyomino $T_6$.}
\label{fig:t6}
\end{figure}

This leads to some restrictions, that we illustrate in figure \ref{figure:coloringd}, where for each $i\in[6]$ we show in the grid for $\gamma=-1$ the places where color $i$ is 
forced (as a result of $i$ not being present in any of the other 5 places within a tile, including the tile's top place in the grid for $\gamma=0$), as well as those where it is 
forbidden (because $i$ is the color of the top place in a tile; we indicate this by graying out the region occupied by the bottom layer of the tile). Since $c$ is periodic, it is easy 
to verify that the pattern suggested in figure \ref{figure:coloringd} indeed continues.

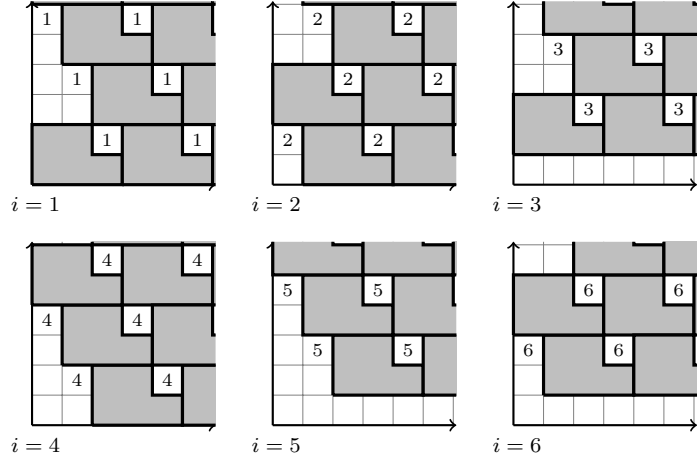
\begin{figure}[ht]
\begin{tikzpicture}[scale=.4]

\draw[step=1cm,gray,very thin] (0,0) grid (6.1,6.1); 
\draw[thick,->] (0,0) -- (6.1,0);
\draw[thick,->] (0,0) -- (0,6.1);
\draw (-1,-1.2) node[anchor=south west]{\footnotesize$i=1$};

\draw[very thick, fill=lightgray] (0,0) -- (0,2) -- (2,2) -- (2,1) -- (3,1) -- (3,0) -- (0,0);
\draw[very thick, fill=lightgray] (3,0) -- (3,2) -- (5,2) -- (5,1) -- (6,1) -- (6,0) -- (3,0);
\draw[very thick, fill=lightgray] (2,2) -- (2,4) -- (4,4) -- (4,3) -- (5,3) -- (5,2) -- (2,2);
\draw[very thick, fill=lightgray] (1,4) -- (1,6) -- (3,6) -- (3,5) -- (4,5) -- (4,4) -- (1,4);
\path[fill=lightgray] (4,4) -- (4,6) -- (6,6) -- (6,5) -- (6.1,5) -- (6.1,4) -- (4,4);
\draw[very thick] (6.1,5) -- (6,5) -- (6,6) -- (4,6) -- (4,4) -- (6.1,4);
\path[fill=lightgray] (0,6.1) -- (0,6) -- (3,6) -- (3,6.1) -- (0,6.1);
\draw[very thick] (0,6.1) -- (0,6) -- (3,6) -- (3,6.1);
\path[fill=lightgray] (3,6.1) -- (3,6) -- (6,6) -- (6,6.1) -- (3,6.1);
\draw[very thick] (3,6.1) -- (3,6) -- (6,6) -- (6,6.1);
\path[fill=lightgray] (6,6.1) -- (6,6) -- (6.1,6) -- (6.1,6.1) -- (6,6.1);
\draw[very thick] (6,6.1) -- (6,6) -- (6.1,6);
\path[fill=lightgray] (6.1,0) -- (6,0) -- (6,2) -- (6.1,2) -- (6.1,0);
\draw[very thick] (6.1,0) -- (6,0) -- (6,2) -- (6.1,2);
\path[fill=lightgray] (6.1,2) -- (5,2) -- (5,4) -- (6.1,4) -- (6.1,2);
\draw[very thick] (6.1,2) -- (5,2) -- (5,4) -- (6.1,4); 

\draw (2.5,1.5) node[anchor=center]{\scriptsize$1$};
\draw (5.5,1.5) node[anchor=center]{\scriptsize$1$};
\draw (1.5,3.5) node[anchor=center]{\scriptsize$1$};
\draw (4.5,3.5) node[anchor=center]{\scriptsize$1$};
\draw (0.5,5.5) node[anchor=center]{\scriptsize$1$};
\draw (3.5,5.5) node[anchor=center]{\scriptsize$1$};

\draw[step=1cm,gray,very thin] (8,0) grid (14.1,6.1);
\draw[thick,->] (8,0) -- (14.1,0);
\draw[thick,->] (8,0) -- (8,6.1);
\draw (7,-1.2) node[anchor=south west]{\footnotesize$i=2$};

\draw[very thick, fill=lightgray] (9,0) -- (9,2) -- (11,2) -- (11,1) -- (12,1) -- (12,0) -- (9,0);
\path[fill=lightgray] (14.1,1) -- (14,1) -- (14,2) -- (12,2) -- (12,0) -- (14.1,0) -- (14.1,1);
\draw[very thick] (14.1,1) -- (14,1) -- (14,2) -- (12,2) -- (12,0) -- (14.1,0); 
\draw[very thick, fill=lightgray] (8,2) -- (8,4) -- (10,4) -- (10,3) -- (11,3) -- (11,2) -- (8,2);
\draw[very thick, fill=lightgray] (11,2) -- (11,4) -- (13,4) -- (13,3) -- (14,3) -- (14,2) -- (11,2);
\path[fill=lightgray] (14.1,4) -- (14,4) -- (14,2) -- (14.1,2) -- (14.1,4);
\draw[very thick] (14.1,4) -- (14,4) -- (14,2) -- (14.1,2);
\draw[very thick, fill=lightgray] (10,4) -- (10,6) -- (12,6) -- (12,5) -- (13,5) -- (13,4) -- (10,4);
\path[fill=lightgray] (14.1,6) -- (13,6) -- (13,4) -- (14.1,4) -- (14.1,6);
\draw[very thick] (14.1,6) -- (13,6) -- (13,4) -- (14.1,4);
\path[fill=lightgray] (9,6.1) -- (9,6) -- (12,6) -- (12,6.1) -- (9,6.1);
\draw[very thick] (9,6.1) -- (9,6) -- (12,6) -- (12,6.1);
\path[fill=lightgray] (12,6.1) -- (12,6) -- (14.1,6) -- (14.1,6.1) -- (12,6.1);
\draw[very thick] (12,6.1) -- (12,6) -- (14.1,6);

\draw (8.5,1.5) node[anchor=center]{\scriptsize$2$};
\draw (11.5,1.5) node[anchor=center]{\scriptsize$2$};
\draw (10.5,3.5) node[anchor=center]{\scriptsize$2$};
\draw (13.5,3.5) node[anchor=center]{\scriptsize$2$};
\draw (9.5,5.5) node[anchor=center]{\scriptsize$2$};
\draw (12.5,5.5) node[anchor=center]{\scriptsize$2$};

\draw[step=1cm,gray,very thin] (16,0) grid (22.1,6.1);
\draw[thick,->] (16,0) -- (22.1,0);
\draw[thick,->] (16,0) -- (16,6.1);
\draw (15,-1.2) node[anchor=south west]{\footnotesize$i=3$};

\draw[very thick, fill=lightgray] (16,1) -- (19,1) -- (19,2) -- (18,2) -- (18,3) -- (16,3) -- (16,1);
\draw[very thick, fill=lightgray] (19,1) -- (22,1) -- (22,2) -- (21,2) -- (21,3) -- (19,3) -- (19,1);
\draw[very thick, fill=lightgray] (18,3) -- (21,3) -- (21,4) -- (20,4) -- (20,5) -- (18,5) -- (18,3);
\path[fill=lightgray] (22.1,3) -- (22,3) -- (22,1) -- (22.1,1) -- (22.1,3);
\draw[very thick] (22.1,3) -- (22,3) -- (22,1) -- (22.1,1);
\path[fill=lightgray] (22.1,5) -- (21,5) -- (21,3) -- (22.1,3) -- (22.1,5);
\draw[very thick] (22.1,5) -- (21,5) -- (21,3) -- (22.1,3);
\path[fill=lightgray] (17,6.1) -- (17,5) -- (20,5) -- (20,6) -- (19,6) -- (19,6.1) -- (17,6.1);
\draw[very thick] (17,6.1) -- (17,5) -- (20,5) -- (20,6) -- (19,6) -- (19,6.1);
\path[fill=lightgray] (20,6.1) -- (20,5) -- (22.1,5) -- (22.1,6) -- (22,6) -- (22,6.1) -- (20,6.1);
\draw[very thick] (20,6.1) -- (20,5) -- (22.1,5);
\draw[very thick] (22.1,6) -- (22,6) -- (22,6.1);

\draw (18.5,2.5) node[anchor=center]{\scriptsize$3$};
\draw (21.5,2.5) node[anchor=center]{\scriptsize$3$};
\draw (17.5,4.5) node[anchor=center]{\scriptsize$3$};
\draw (20.5,4.5) node[anchor=center]{\scriptsize$3$};

\draw[step=1cm,gray,very thin] (0,-8) grid (6.1,-1.9); 
\draw[thick,->] (0,-8) -- (6.1,-8);
\draw[thick,->] (0,-8) -- (0,-1.9);
\draw (-1,-9.2) node[anchor=south west]{\footnotesize$i=4$};

\draw[very thick, fill=lightgray] (2,-8) -- (5,-8) -- (5,-7) -- (4,-7) -- (4,-6) -- (2,-6) -- (2,-8);
\draw[very thick, fill=lightgray] (1,-6) -- (4,-6) -- (4,-5) -- (3,-5) -- (3,-4) -- (1,-4) -- (1,-6);
\draw[very thick, fill=lightgray] (0,-4) -- (3,-4) -- (3,-3) -- (2,-3) -- (2,-2) -- (0,-2) -- (0,-4);
\draw[very thick, fill=lightgray] (3,-4) -- (6,-4) -- (6,-3) -- (5,-3) -- (5,-2) -- (3,-2) -- (3,-4);
\path[fill=lightgray] (6.1,-8) -- (5,-8) -- (5,-6) -- (6.1,-6) -- (6.1,-8);
\draw[very thick] (6.1,-8) -- (5,-8) -- (5,-6) -- (6.1,-6);
\path[fill=lightgray] (6.1,-6) -- (4,-6) -- (4,-4) -- (6,-4) -- (6,-5) --(6.1,-5) -- (6.1,-6);
\draw[very thick] (6.1,-6) -- (4,-6) -- (4,-4) -- (6,-4) -- (6,-5) --(6.1,-5);
\path[fill=lightgray] (6.1,-4) -- (6,-4) -- (6,-2) -- (6.1,-2) -- (6.1,-4);
\draw[very thick] (6.1,-4) -- (6,-4) -- (6,-2) -- (6.1,-2);
\path[fill=lightgray] (2,-1.9) -- (2,-2) -- (5,-2) -- (5,-1.9) -- (2,-1.9);
\draw[very thick] (2,-1.9) -- (2,-2) -- (5,-2) -- (5,-1.9);
\path[fill=lightgray] (5,-1.9) -- (5,-2) -- (6.1,-2) -- (6.1,-1.9) -- (5,-1.9);
\draw[very thick] (5,-1.9) -- (5,-2) -- (6.1,-2);

\draw (1.5,-6.5) node[anchor=center]{\scriptsize$4$};
\draw (4.5,-6.5) node[anchor=center]{\scriptsize$4$};
\draw (0.5,-4.5) node[anchor=center]{\scriptsize$4$};
\draw (3.5,-4.5) node[anchor=center]{\scriptsize$4$};
\draw (2.5,-2.5) node[anchor=center]{\scriptsize$4$};
\draw (5.5,-2.5) node[anchor=center]{\scriptsize$4$};

\draw[step=1cm,gray,very thin] (8,-8) grid (14.1,-1.9);
\draw[thick,->] (8,-8) -- (14.1,-8);
\draw[thick,->] (8,-8) -- (8,-1.9);
\draw (7,-9.2) node[anchor=south west]{\footnotesize$i=5$};

\draw[very thick, fill=lightgray] (10,-7) -- (13,-7) -- (13,-6) -- (12,-6) -- (12,-5) -- (10,-5) -- (10,-7);
\draw[very thick, fill=lightgray] (9,-3) -- (9,-5) -- (12,-5) -- (12,-4) -- (11,-4) -- (11,-3) -- (9,-3);
\path[fill=lightgray] (14.1,-7) -- (13,-7) -- (13,-5) -- (14.1,-5) -- (14.1,-7);
\draw[very thick] (14.1,-7) -- (13,-7) -- (13,-5) -- (14.1,-5);
\path[fill=lightgray] (14.1,-4) -- (14,-4) -- (14,-3) -- (12,-3) -- (12,-5) -- (14.1,-5) -- (14.1,-4);
\draw[very thick] (14.1,-4) -- (14,-4) -- (14,-3) -- (12,-3) -- (12,-5) -- (14.1,-5);
\path[fill=lightgray] (8,-1.9) -- (8,-3) -- (11,-3) -- (11,-2) -- (10,-2) -- (10,-1.9) -- (8,-1.9);
\draw[very thick] (8,-1.9) -- (8,-3) -- (11,-3) -- (11,-2) -- (10,-2) -- (10,-1.9);
\path[fill=lightgray] (11,-1.9) -- (11,-3) -- (14,-3) -- (14,-2) -- (13,-2) -- (13,-1.9) -- (11,-1.9);
\draw[very thick] (11,-1.9) -- (11,-3) -- (14,-3) -- (14,-2) -- (13,-2) -- (13,-1.9);
\path[fill=lightgray] (14,-1.9) -- (14,-3) -- (14.1,-3) -- (14.1,-1.9) -- (14,-1.9);
\draw[very thick] (14,-1.9) -- (14,-3) -- (14.1,-3);

\draw (9.5,-5.5) node[anchor=center]{\scriptsize$5$};
\draw (12.5,-5.5) node[anchor=center]{\scriptsize$5$};
\draw (8.5,-3.5) node[anchor=center]{\scriptsize$5$};
\draw (11.5,-3.5) node[anchor=center]{\scriptsize$5$};

\draw[step=1cm,gray,very thin] (16,-8) grid (22.1,-1.9);
\draw[thick,->] (16,-8) -- (22.1,-8);
\draw[thick,->] (16,-8) -- (16,-1.9);
\draw (15,-9.2) node[anchor=south west]{\footnotesize$i=6$};

\draw[very thick, fill=lightgray] (17,-7) -- (20,-7) -- (20,-6) -- (19,-6) -- (19,-5) -- (17,-5) -- (17,-7);
\draw[very thick, fill=lightgray] (16,-5) -- (16,-3) -- (18,-3) -- (18,-4) -- (19,-4) -- (19,-5) -- (16,-5);
\draw[very thick, fill=lightgray] (19,-5) -- (19,-3) -- (21,-3) -- (21,-4) -- (22,-4) -- (22,-5) -- (19,-5);
\path[fill=lightgray] (22.1,-7) -- (20,-7) -- (20,-5) -- (22,-5) -- (22,-6) -- (22.1,-6) -- (22.1,-7);
\draw[very thick] (22.1,-7) -- (20,-7) -- (20,-5) -- (22,-5) -- (22,-6) -- (22.1,-6);
\path[fill=lightgray] (22.1,-5) -- (22,-5) -- (22,-3) -- (22.1,-3) -- (22.1,-5);
\draw[very thick] (22.1,-5) -- (22,-5) -- (22,-3) -- (22.1,-3);
\path[fill=lightgray] (18,-1.9) -- (18,-3) -- (21,-3) -- (21,-2) -- (20,-2) -- (20,-1.9) -- (18,-1.9);
\draw[very thick] (18,-1.9) -- (18,-3) -- (21,-3) -- (21,-2) -- (20,-2) -- (20,-1.9);
\path[fill=lightgray] (21,-1.9) -- (21,-3) -- (22.1,-3) -- (22.1,-1.9) -- (21,-1.9);
\draw[very thick] (21,-1.9) -- (21,-3) -- (22.1,-3);

\draw (16.5,-5.5) node[anchor=center]{\scriptsize$6$};
\draw (19.5,-5.5) node[anchor=center]{\scriptsize$6$};
\draw (18.5,-3.5) node[anchor=center]{\scriptsize$6$};
\draw (21.5,-3.5) node[anchor=center]{\scriptsize$6$};

\end{tikzpicture}
\caption{Extending $c$ to a 6-satisfactory coloring $d$.}
\label{figure:coloringd}
\end{figure}

The point is that these conditions do not determine $d$ entirely. On the grid corresponding to $\gamma=-1$, all colors are determined except for those in the bottom row, 
corresponding to $\beta=0$, see figure \ref{figure:coloringd-1}.

\begin{figure}[ht]
\begin{tikzpicture}[scale=.4]
\draw[step=1cm,gray,very thin] (0,0) grid (6.1,6.1);
\draw[thick,->] (0,0) -- (6.1,0);
\draw[thick,->] (0,0) -- (0,6.1);

    \setcounter{row}{1}
    \setrow {\scriptsize 1}{\scriptsize 2}{\scriptsize 4}{\scriptsize 1}{\scriptsize 2}{\scriptsize 4}  
    \setrow {\scriptsize 5}{\scriptsize 3}{\scriptsize 6}{\scriptsize 5}{\scriptsize 3}{\scriptsize 6}  
    \setrow {\scriptsize 4}{\scriptsize 1}{\scriptsize 2}{\scriptsize 4}{\scriptsize 1}{\scriptsize 2}
    \setrow {\scriptsize 6}{\scriptsize 5}{\scriptsize 3}{\scriptsize 6}{\scriptsize 5}{\scriptsize 3}  
    \setrow {\scriptsize 2}{\scriptsize 4}{\scriptsize 1}{\scriptsize 2}{\scriptsize 4}{\scriptsize 1}  

\end{tikzpicture}
\caption{Values taken by $d$ on points in $\{a5^{-1}:a\in K_5,5\nmid a\}$.}
\label{figure:coloringd-1}
\end{figure}

As for what colors $d$ must assign to points in that row, what we see is that if $d(5^{-1})=u,d(2\cdot 5^{-1})=v,d(4\cdot 5^{-1})=w$, then $\{u,v,w\}=\{3,5,6\}$, and $d$ still satisfies 
that $d(8\cdot x)=d(x)$ for any $x$ in its domain. But there are precisely six 6-colorings $d$ satisfying these requirements, and any of them is 6-satisfactory, as claimed.
\end{proof}

Note that the effort in the proof of theorem \ref{thm:nonmultiplicative} came in showing that the six colorings we identified are \emph{all} the satisfactory colorings $d$ with 
$d_5=c$. A direct verification would have sufficed if all we wanted was to show that there are at least six such colorings; again, thinking of them as having domain 
$\frac15\cdot K_6$, all of them are given by equation (\ref{eqn:col-6}) except for $\beta=0,\gamma=-1$, where they are given as described in the last paragraph of the proof. 

\begin{remark} \label{rmk:n6}
We can extend the argument of theorem \ref{thm:nonmultiplicative} in a few ways. For instance, their periodicity in the direction of 2 allows us to consider these colorings as 
defined on $\{a/2^n:a\in K_6,n\in\mathbb N\}$. And we can iterate the construction: assign to each permutation of $\{3,5,6\}$ a number in $[6]$, and do the same to the 
permutations of $\{1,2,4\}$. Denote by $[6]^{<\mathbb N}$ the set of finite strings of members of $[6]$, that is, 
 $$ [6]^{<\mathbb N}=\bigcup_{n\in\mathbb N}[6]^n. $$
For $\sigma\in [6]^{<\mathbb N}$ denote its length by $|\sigma|$. Starting with $d^\emptyset=c$, we can associate to each finite string $\sigma\in[6]^{<\mathbb N}$ a coloring 
$d^\sigma$ with the following properties:
\begin{enumerate}
\item 
$\operatorname{dom}(d^\sigma)=\bigl\{\frac{a}{2^n5^m}:a\in K_6,n\in\mathbb N,0\le m\le |\sigma|\bigr\}$ and $d^\sigma$ is 6-satisfactory on its domain.
\item 
The functions $d^{\sigma{}^\frown\langle i\rangle}$ for $i\in[6]$ are all the 6-satisfactory maps $d$ on their domain such that $d_5=d^\sigma$ (under the convention of theorem 
\ref{thm:nonmultiplicative}, where we think of the equation $d_5=d'$ for a given $d'$ as seeking a map $d$ with domain $\frac15\cdot \operatorname{dom}(d')$). 
\item 
Given $d^\sigma$, $d^{\sigma{}^\frown\langle i\rangle}$ is completely determined by its values on numbers of the form $2^\alpha 5^{-|\sigma|-1}$, $\alpha\in\mathbb Z$, and in 
turn these values are given by the permutation associated to $i$ corresponding to $\{3,5,6\}$ if $\sigma$ is even and to $\{1,2,4\}$ if $\sigma$ is odd, as follows: if the permutation 
is $(u_0,u_1,u_2)$, then $d^{\sigma{}^\frown\langle i\rangle}(2^\alpha 5^{-|\sigma|-1})=u_j$, where $\alpha\equiv j\pmod 3$. 
\end{enumerate}
(The proof of this is a straightforward extension of that of theorem \ref{thm:nonmultiplicative}, we omit the details.) In turn, this implies that if 
$\tilde K_6=\{2^\alpha 3^\beta 5^\gamma : \alpha,\gamma\in\mathbb Z,\beta\in\mathbb N\}$, then $|C_{\tilde K_6}|=\mathfrak c$, since to each infinite sequence 
$x\in[6]^{\mathbb Z^+}$ we can associate the coloring 
 $$ d^x=\bigcup_{n\in\mathbb N}d^{x\upharpoonright [n]}, $$
all these colorings are different, and all are 6-satisfactory.

Unfortunately, the argument does not seem to allow for a straightforward extension that would permit us to further extend the domains of these colorings to all of $\hat K_6$. 

The colorings so obtained have a further application, namely, they imply that question \ref{q:inverse2} has a negative answer. Indeed, for generic $x$, use $d^x$ to obtain a partial
tiling by $T_6$ of the image under $t$ of $\tilde K_6$, note that this can be extended to an essential tiling, and let $(d^x)'$ be the coloring coming from this tiling. We see that 
$(d^x)'\ne d^x$. Moreover, this is not an issue of the behavior of partial tiles at the boundary, as this tiling is actually quite tame. In fact, it suffices to take $x$ so that for 
$\gamma=-1$ we use the permutation $(3,5,6)$ and for any other $\gamma$ we use either $(3,6,5)$ or $(1,2,4)$. The point is that tiles contained in the orthant (so 
$\gamma\ge 0$) are colored in a certain pattern (copying the coloring of $T_6$ itself) while tiles involving points using the $(3,5,6)$ permutation are colored differently.
\end{remark}

That the analysis in the proof of theorem \ref{thm:nonmultiplicative} ended up working so neatly is because, in addition to the equation $c(8m)=c(m)$, the coloring $c$ also 
satisfies that for any fixed $\beta,\gamma$ in $\mathbb N$, the set $\{c(2^\alpha 3^\beta 5^\gamma):\alpha\in\mathbb N\}$ is either $\{1,2,4\}$ or $\{3,5,6\}$, and this only 
depends on the parity of $\beta+\gamma$. This is most readily apparent geometrically: consider an essential tiling of $\mathbb O_6$ by $T_6$ induced by $c$. For each 
fixed $\gamma$, the trace on this tiling on the plane grid corresponding to $\gamma$ shows up in horizontal strips of height two, as can be seen in figure \ref{figure:cstrips}. 

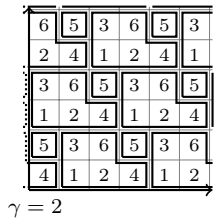
\begin{figure}[ht]
\begin{tikzpicture}[scale=.4]
\draw[step=1cm,gray,very thin] (0,0) grid (6.1,6.1);
\draw[thick,->] (0,0) -- (6.1,0);
\draw[thick,->] (0,0) -- (0,6.1);
\draw (-1,-1.2) node[anchor=south west]{\footnotesize$\gamma=2$};

\setcounter{row}{1}
\setrow{\scriptsize 6}{\scriptsize 5}{\scriptsize 3}{\scriptsize 6}{\scriptsize 5}{\scriptsize 3}
\setrow{\scriptsize 2}{\scriptsize 4}{\scriptsize 1}{\scriptsize 2}{\scriptsize 4}{\scriptsize 1}
\setrow{\scriptsize 3}{\scriptsize 6}{\scriptsize 5}{\scriptsize 3}{\scriptsize 6}{\scriptsize 5}
\setrow{\scriptsize 1}{\scriptsize 2}{\scriptsize 4}{\scriptsize 1}{\scriptsize 2}{\scriptsize 4}
\setrow{\scriptsize 5}{\scriptsize 3}{\scriptsize 6}{\scriptsize 5}{\scriptsize 3}{\scriptsize 6}
\setrow{\scriptsize 4}{\scriptsize 1}{\scriptsize 2}{\scriptsize 4}{\scriptsize 1}{\scriptsize 2}

\draw[thick] (0.1,2.1) -- (0.1,3.9) -- (1.9,3.9) -- (1.9,2.9) -- (2.9,2.9) -- (2.9,2.1) -- (0.1,2.1);
\draw[thick] (3.1,2.1) -- (3.1,3.9) -- (4.9,3.9) -- (4.9,2.9) -- (5.9,2.9) -- (5.9,2.1) -- (3.1,2.1);
\draw[thick] (1.1,0.1) -- (1.1,1.9) -- (2.9,1.9) -- (2.9,0.9) -- (3.9,0.9) -- (3.9,0.1) -- (1.1,0.1);
\draw[thick] (2.1,4.1) -- (2.1,5.9) -- (3.9,5.9) -- (3.9,4.9) -- (4.9,4.9) -- (4.9,4.1) -- (2.1,4.1);
\draw[thick] (0,6.1) -- (0.9,6.1);
\draw[thick,densely dotted] (-0.2,6.1) -- (0,6.1);
\draw[thick] (1.1,6.1) -- (3.9,6.1);
\draw[thick] (4.1,6.1) -- (6.1,6.1);
\draw[thick] (6.1,5.9) -- (5.1,5.9) -- (5.1,4.1) -- (6.1,4.1);
\draw[thick] (0,5.9) -- (0.9,5.9) -- (0.9,4.9) -- (1.9,4.9) -- (1.9,4.1) -- (0,4.1);
\draw[thick,densely dotted] (-0.2,5.9) -- (0,5.9); 
\draw[thick,densely dotted] (-0.2,4.1) -- (0,4.1);
\draw[thick] (6.1,3.9) -- (6.1,2.1);
\draw[thick] (0,0.9) -- (0.9,0.9) -- (0.9,0.1) -- (0,0.1);
\draw[thick,densely dotted] (0,0.9) -- (-0.1,0.9) -- (-0.1,1.9) -- (-0.2,1.9);
\draw[thick,densely dotted] (-0.2,0.1) -- (0,0.1);
\draw[thick,densely dotted] (-0.2,2.1) -- (-0.1,2.1) -- (-0.1,2.9) -- (-0.2,2.9);
\draw[thick] (6.1,0.9) -- (5.9,0.9) -- (5.9,1.9) -- (4.1,1.9) -- (4.1,0.1) -- (6.1,0.1);
\draw[thick] (0.1,1.1) -- (0.9,1.1) -- (0.9,1.9) -- (0.1,1.9) -- (0.1,1.1);
\draw[thick] (3.1,1.1) -- (3.9,1.1) -- (3.9,1.9) -- (3.1,1.9) -- (3.1,1.1);
\draw[thick] (2.1,3.1) -- (2.9,3.1) -- (2.9,3.9) -- (2.1,3.9) -- (2.1,3.1);
\draw[thick] (5.1,3.1) -- (5.9,3.1) -- (5.9,3.9) -- (5.1,3.9) -- (5.1,3.1);
\draw[thick] (1.1,5.1) -- (1.9,5.1) -- (1.9,5.9) -- (1.1,5.9) -- (1.1,5.1);
\draw[thick] (4.1,5.1) -- (4.9,5.1) -- (4.9,5.9) -- (4.1,5.9) -- (4.1,5.1);
\draw[thick,densely dotted] (-0.2,3.1) -- (-0.1,3.1) -- (-0.1,3.9) -- (-0.2,3.9);

\end{tikzpicture}
\caption{Trace of a tiling induced by $c$ on $\{(\alpha,\beta,\gamma):\alpha,\beta\in\mathbb N\}$ for fixed $\gamma$.}
\label{figure:cstrips}
\end{figure}

This suggests a natural approach towards strengthening the conclusion that there are nonmultiplicative colorings, that we now proceed to present.

\begin{theorem} \label{thm:n6}
$|C_{K_6}|=\mathfrak c$.
\end{theorem}

\begin{proof}
Write $\mathbb O_6=\{(\alpha,\beta,\gamma):\alpha,\beta,\gamma\in\mathbb N\}$, identifying each point $(\alpha,\beta,\gamma)$ with the number 
$2^\alpha 3^\beta 5^\gamma\in K_6$. We will define a family of 6-satisfactory colorings $d$ by specifying certain essential tilings of $\mathbb O_6$. For each of them, as in the 
example just discussed, see also figure \ref{figure:cstrips}, the trace of the tiling on any plane $\gamma=\gamma_0$ is naturally organized along horizontal strips of height two 
(which, in particular ensures that any such coloring $d$ satisfies $d(8m)=d(m)$ for all $m$), and that there are many such colorings comes from the fact that different strips are 
independent of each other.

Fixing $\gamma=\gamma_0$, each of the horizontal strips we consider has the form $H_k(\gamma_0)$ for some $k\in\mathbb N$, where 
 $$ H_k(\gamma_0)\coloneqq\{(\alpha,\beta,\gamma_0):\alpha\in\mathbb N, \beta=k\mbox{ or }k+1\}. $$  
We follow our usual convention that $d(i)=i$ for $i\in[6]$, meaning that $T_6$ itself is one of the tiles we use (equivalently, if the tiling is $T_6+B\supset \mathbb O_6$, the sum 
being direct, then $\mathbf 0\in B$). For the examples we consider, the trace of the tiling on a strip has one of six possible types, but it is enough for our purposes to only describe 
three of them. Each description refers to figure \ref{figure:cstrips} and the (essential) tiling induced by $c$ depicted there; for instance, to be of type 1 means to be exactly as the 
tiling of $H_2(2)$ shown in figure \ref{figure:cstrips}. Accordingly, in the descriptions below we omit the sentence ``shown in figure \ref{figure:cstrips}'' each time.

A tiling of a strip is of type
\begin{itemize}
\item 
1 if and only if it is (precisely) the tiling of $H_2(2)$,
\item
2 if and only if it is the tiling of $H_0(2)$, and 
\item
3 if and only if it is the tiling of $H_4(2)$.
\end{itemize}
(All tilings here are actually essential tilings, and we omit the word ``essential'' in what follows.) Note that if in a 6-satisfactory coloring a strip $H_k(\gamma)$ with $k+\gamma$ 
even is of type $j\in[3]$, then $H_{k-1}(\gamma+1)$ is of type $j+1$ (using cyclic notation, so that 4 is identified with 1). This includes the case $k=0$ in cases where $\gamma$ 
is even.

We are ready to describe the colorings, they are of the form $d^x$ for $x\in[3]^{\mathbb Z^+}$, where $d^x$ is defined as follows: the tiling induced by the coloring $d^x$ has the 
following trace on the plane $\gamma=0$: $H_0(0)$ has type 1 (as required by our convention). For each  $k>0$, $H_{2k}(0)$ has type $x(k)$. For $\gamma>0$, the traces are 
recursively defined according to the last sentence of the previous paragraph. 

It is immediate from the construction that if $x\ne x'$, then $d^x\ne d^{x'}$, so we have defined $3^{\aleph_0}=\mathfrak c$ colorings of $K_6$, and that all of them are 
6-satisfactory,  since what we have actually done is to give $d^x$ via an explicit essential tiling of $\mathbb O_6$. 
\end{proof}

Note that, for $x^1$ the constant function taking the value 1, the coloring $d^{x^1}$ just described is the $\mathbb Z_6$-coloring with strong representative 103 and 
determined by table \ref{table:z6103}. Also, note that the set $\{d^x:x\in[3]^{\mathbb Z^+}\}$ not only has the same size as the reals but is in fact a perfect subset of $C_{K_6}$. 

There are infinitely many ways of choosing $x$ so that the resulting $d^x$ is periodic in the direction of 3 (and therefore periodic), and we can moreover ensure that this period is 
not a factor of 6, in fact, $o(3)$ can be taken to be arbitrarily large. As remarked in \S\,\ref{subs:translation}, this indicates that the existence of periodic $n$-satisfactory colorings 
is not enough to ensure the existence of translation invariant (i.e., multiplicative) ones.

\begin{remark}
We previously obtained a different proof of theorem \ref{thm:n6} that used the fact that $C_{K_6}$ is closed in its natural topology. The argument proceeded in two stages. First,
theorem \ref{thm:nonmultiplicative} was extended using a variant of the construction in remark \ref{rmk:n6} to obtain an infinite countable family of 6-satisfactory colorings such 
that for each $c$ in the family there were six colorings $d$ in the family with $d_5=c$. The members of the family were arranged as nodes on a complete senary tree, in such a 
way that the colorings along each branch converged, and the limit colorings so produced were pairwise different. Further, all these colorings $d$ satisfy that $d(8m)=d(m)$ and 
$d(3m)=d(10m)$ for any $m\in K_6$. Lon Mitchell also found an elegant proof; his key insight was that one could get rid of the topological argument and instead define directly 
in a combinatorial way the colorings that the previous limit process had produced.
\end{remark}

\subsection{Nonmultiplicative 8-satisfactory colorings} \label{subs:n8}

We adapt the proof of theorem \ref{thm:n6} to show that the result applies to $C_{K_8}$ as well. 

\begin{theorem} \label{thm:n8}
$|C_{K_8}|=\mathfrak c$.
\end{theorem}

\begin{proof}
We construct a perfect set of 8-satisfactory colorings of $K_8$ by describing the associated essential tilings of 
$\mathbb O_8=\{(\alpha,\beta,\gamma,\delta):\alpha,\beta,\gamma,\delta\in\mathbb N\}$. As before, the tilings have the form $T_8+B$ with $\mathbf 0\in B$ and are given by tiling
strips, which are now of the form $H_k(\gamma_0,\delta_0)$ for some fixed $\gamma_0,\delta_0$ and some $k\in\mathbb N$, where 
 $$ H_k(\gamma_0,\delta_0)\coloneqq\{(\alpha,\beta,\gamma_0,\delta_0):\alpha\in\mathbb N, \beta=k\mbox{ or }k+1\}. $$  

We begin by describing the four possible types a strip $H_k(\gamma_0,\delta_0)$ may be depending on the trace of the tiling on the strip, for which we refer to figure \ref{fig:trace}
(as in the $n=6$ case, more types are possible, but these are the only ones we consider). All our tilings are periodic in the sense that the associated coloring $d$ satisfies 
$d(16m)=d(m)$ for any $m\in K_8$. 

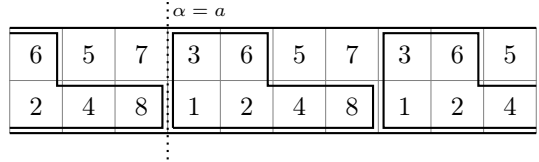
\begin{figure}[htt]
\begin{tikzpicture}[scale=.7]
\draw[step=1cm,gray,very thin] (0,0) grid (10,2);
\draw[thick] (0,0) -- (10,0);
\draw[thick] (0,2) -- (10,2);
\draw[thick, dotted] (3,-0.5) -- (3,2.5);

\draw[thick] (0,0.1) -- (2.9,0.1) -- (2.9,0.9) -- (0.9,0.9) -- (0.9,1.9) -- (0,1.9);
\draw[thick] (3.1,0.1) -- (6.9,0.1) -- (6.9,0.9) -- (4.9,0.9) -- (4.9,1.9) -- (3.1,1.9) -- (3.1,0.1);
\draw[thick] (10,0.1) -- (7.1,0.1) -- (7.1,1.9) -- (8.9,1.9) -- (8.9,0.9) -- (10,0.9);

\node[anchor=center] at (0.5,0.5) {2};
\node[anchor=center] at (1.5,0.5) {4};
\node[anchor=center] at (2.5,0.5) {8};
\node[anchor=center] at (3.5,0.5) {1};
\node[anchor=center] at (4.5,0.5) {2};
\node[anchor=center] at (5.5,0.5) {4};
\node[anchor=center] at (6.5,0.5) {8};
\node[anchor=center] at (7.5,0.5) {1};
\node[anchor=center] at (8.5,0.5) {2};
\node[anchor=center] at (9.5,0.5) {4};
\node[anchor=center] at (0.5,1.5) {6};
\node[anchor=center] at (1.5,1.5) {5};
\node[anchor=center] at (2.5,1.5) {7};
\node[anchor=center] at (3.5,1.5) {3};
\node[anchor=center] at (4.5,1.5) {6};
\node[anchor=center] at (5.5,1.5) {5};
\node[anchor=center] at (6.5,1.5) {7};
\node[anchor=center] at (7.5,1.5) {3};
\node[anchor=center] at (8.5,1.5) {6};
\node[anchor=center] at (9.5,1.5) {5};

\node at (3.6,2.3) {\scriptsize$\alpha=a$};

\end{tikzpicture}
\caption{Trace of a tiling on a planar strip $H_k(\gamma_0,\delta_0)$.}
\label{fig:trace}
\end{figure}

We say the tiling is of type
\begin{itemize}
\item
1 if and only if it is (precisely) the depicted tiling with $a=0$,
\item
2 if and only if it is the depicted tiling with $a=-1$,
\item
3 if and only if it is the depicted tiling with $a=-2$,
\item
4 if and only if it is the depicted tiling with $a=-3$.
\end{itemize} 

We will define the colorings we are interested in by describing the types of the strips $H_{2k}(0,0)$, and extending this to a coloring of all of $K_8$ recursively by the rule
that if $k+\gamma_0+\delta_0$ is even and $H_k(\gamma_0,\delta_0)$ is of type $i$, then $H_{k-1}(\gamma_0+1,\delta_0)$ is of type $i+2$ and 
$H_{k-1}(\gamma_0,\delta_0+1)$ is of type $i+3$, using cyclic notation modulo 4. 

We must verify that this procedure is well-defined, specifically, that the recursion just described assigns a unique color to any point in $\mathbb O_8$ once the colors of points 
in the (intersection of the orthant with the) plane $\gamma=\delta=0$ are specified. For this, first note that the recursion describes how to assign colors to points in 
$H_k(\gamma,\delta)$ for $k$ of the same parity as $\gamma+\delta$ so, in particular, fixing the values of $\gamma$ and $\delta$, for any point $x$ in the resulting plane 
there is a unique $k\ge -1$ with $x\in H_k(\gamma,\delta)$ and $k$ of the relevant parity. Now, by induction on $\gamma+\delta$, let $i$ be the type of 
$H_{k+\gamma+\delta}(0,0)$, and check that the rules specify that the type of $H_k(\gamma,\delta)$ is precisely $i+2\gamma+3\delta$ (using the cyclic convention), so the 
color assigned to $x$ is indeed unambiguous.

Finally, we define the colorings $d^x$ for $x\in[4]^{\mathbb Z^+}$ by setting $H_0(0,0)$ to be of type 1 and, for $k>0$, $H_{2k}(0,0)$ to be of type $x(k)$. The colorings so 
described are pairwise different and are 8-satisfactory since, by construction, $d^x(16m)=d^x(m)$ for any $m$ while $d^x(m),d^x(2m),d^x(4m),d^x(8m)$ are pairwise distinct, 
and they are all different from $d^x(3m),d^x(6m),d^x(12m),d^x(24m)$ and, moreover, $d^x(5m)=d^x(12m)$ and $d^x(7m)=d^x(24m)$. Note that, as in the $n=6$ case, the set 
$\{d^x:x\in [4]^{\mathbb Z^+}\}$ is a perfect subset of $C_{K_8}$.
\end{proof}

Note that for $x^1$ the constant function taking the value 1, the coloring $d^{x^1}$ is the $\mathbb Z/2\mathbb Z\times\mathbb Z/4\mathbb Z$-coloring determined by table 
\ref{table:7531} for $a=1$.

\begin{remark}
We could have presented the proof purely algebraically by building the colorings $d^x$ recursively, while requiring they satisfy the rules indicated in the last paragraph, but we 
chose the geometric presentation as it seems more intuitive.
\end{remark}

\section{Open questions} \label{sec:questions}

For the reader's convenience, we close the paper by listing the questions we have mentioned throughout the paper. We omit those that, like question 
\ref{q:zncoloringrepresentation}, were asked as rhetorical devices and are answered in the text.

\begin{questionn}{q:problem}
Given any positive integer $n$, is there a coloring of the positive integers using $n$ colors such that for any positive integer $a$, the numbers $a,2a,\dots,na$ all have 
different colors?
\end{questionn}

\begin{questionn}{q:prime}
Assuming that question \ref{q:problem} has a negative answer for $n$, can we find a better bound than the smallest prime larger than $n$ on the number of colors required to ensure 
a positive answer?
\end{questionn}

We refine the original formulation of question \ref{q:number} as follows:

\begin{questionn}{q:number}
Given $n>1$, how many $n$-satisfactory colorings of $K_n$ are there, if any at all? Is the map $n\mapsto |C_{K_n}|$ that assigns to each $n$ the number of $n$-satisfactory 
colorings of the core a recursive function (taking values in $\mathbb N\cup\{\aleph_0,\mathfrak c\}$)?
\end{questionn}

\begin{questionn}{q:closed} 
Given $n\in\mathbb Z^+$, suppose that $C_{K_n}$ is nonempty. Should it have isolated points?
\end{questionn}

For the operation $d\mapsto d_k$ on colorings, see item (3) in \S\,\ref{subsec:structure}.

\begin{questionn}{q:equation}
Given an $n$-satisfactory coloring $c$ and $k\in K_n$, is there an $n$-satisfactory coloring $d$ such that $d_k=c$? In that case, how many such colorings $d$ are there?
\end{questionn}

\begin{questionn}{q:tilingz}
Let $n\in\mathbb Z^+$.
\begin{enumerate}
\item
Does any $n$-satisfactory coloring of $K_n$ extend to one of $\hat K_n$?
\item
If $T_n$ essentially tiles $\mathbb O_n$ via a tiling $T_n+B$, is there a tiling by $T_n$ of all of $\mathbb Z^{\pi(n)}$ that essentially extends it?
\end{enumerate}
\end{questionn}

\begin{questionn}{q:inverse}
Given an $n$-satisfactory coloring $c$ of $\hat K_n$, let $B$ be the image under $t$ of a color class of $c$. The proof of proposition \ref{proposition:tiling} shows that the sum 
$T_n+B$ is a tiling of $\mathbb Z^{\pi(n)}$. From this tiling we can define an $n$-satisfactory coloring $c'$ with color classes the preimages under $t$ of the translates $t(i)+B$, 
$i\in[n]$. Is $c'=c$?
\end{questionn}

See the discussion surrounding the original presentation of question \ref{q:inverse2} for additional details of the setting it references. Briefly, given $n$, from an $n$-satisfactory 
coloring $c$ of $K_n$, we obtain $n$ partial tilings of the orthant $\mathbb O_n$, and any point in $\mathbb O_n$ belongs to at least one of the resulting direct sums $T_n+B$.
Any of these sums in turn defines a partial $n$-satisfactory coloring of $K_n$.

\begin{questionn}{q:inverse2}
Are the resulting partial colorings compatible? If they are, their union gives us a coloring $c'$ of $K_n$. Is $c'=c$?
\end{questionn}

Originally, question \ref{qu:natdensity} was listed with three parts, but we proceeded to solve positively the first two. The following remains, though we expect the answer to be 
negative and easily accessible from the techniques we discuss in \S\,\ref{subsec:density}.

\begin{questionn}{qu:natdensity}
Let $n\in\mathbb Z^+$. Suppose $n$ admits a strong representative. For a satisfactory $n$-coloring $c$, is the natural density of the set of strong representatives of order $n$ for 
$c$ independent of $c$?
\end{questionn}

\begin{questionn}{q:ref}
Is the set of groupless $n$ infinite, or even of natural density 1?
\end{questionn}

Many combinatorial questions remain besides those just listed. They appear intractable with current methods.

\begin{question} \label{q:many}
Is there an $n$ admitting precisely a countable infinity of $n$-sat\-is\-fac\-to\-ry colorings of $K_n$? Which finite $m$ are precisely the number of $n$-satisfactory colorings of 
$K_n$ for some $n$? 
\end{question}

\subsection*{Acknowledgements} \label{subsec:acknowledgements}

We thank Amanda Francis for creating figure \ref{fig:t6}. Thanks are also due to Zach Teitler for alerting us of \cite{ForcadeLamoreauxPollington86}, and to Rodney Forcade for 
providing us with the data for table \ref{table:groupless}. Special thanks to D\"om\"ot\"or P\'alv\"olgyi for promoting the question we study in this paper (and for making the 
first-named author aware of it!) by posting it as question 26358 in MathOverflow. Thanks are also due to the MathOverflow community for their ideas and suggestions, and in 
particular to Darij Grinberg, Gergely Harcos and Noam D. Elkies for allowing us to include their results. Thanks to Ben Barber and P\'eter Csikv\'ari for their interest on the 
topic of this paper and their suggestions. Thanks to Felipe Voloch and David E Speyer for their illuminating suggestions and assistance regarding the subject of 
\S\,\ref{subsec:density}, and to Lon Mitchell for his interest on the topic of \S\,\ref{subs:n6}. We also want to thank the anonymous referee for their careful reading of the manuscript 
and valuable suggestions.

The first-named author gave talks on this topic at Albion College, Albion, MI, and at Miami University, Oxford, OH, and wants to thank his respective hosts for the invitations 
and support. The visit to Miami University was partially supported by NSF grant DMS-1201494. The third-named author was supported by the Lend\"ulet program of the Hungarian 
Academy of Sciences (MTA), the National Research, Development and Innovation Office NKFIH (Grant Nr. PD115978, K129335 and BME NC TKP2020), the New National 
Excellence Program of the Ministry of Human Capacities (UNKP-18-4) and the J\'anos Bolyai Research Scholarship of the Hungarian Academy of Sciences.  

\bibliographystyle{amsalpha}
\bibliography{coloring}

\newcommand{\etalchar}[1]{$^{#1}$}
\providecommand{\bysame}{\leavevmode\hbox to3em{\hrulefill}\thinspace}
\providecommand{\MR}{\relax\ifhmode\unskip\space\fi MR }
\providecommand{\MRhref}[2]{%
  \href{http://www.ams.org/mathscinet-getitem?mr=#1}{#2}
}
\providecommand{\href}[2]{#2}
\begin{thebibliography}{BDG{\etalchar{+}}18}

\bibitem[BDG{\etalchar{+}}18]{Boseketal18}
Bart{\l}omiej Bosek, Micha{\l} D\k{e}bski, Jaros{\l}aw Grytczuk, Joanna
  Sok\'{o}{\l}, Ma{\l}gorzata \'{S}leszy\'{n}ska Nowak, and Wiktor \.{Z}elazny,
  \emph{Graph coloring and {G}raham's greatest common divisor problem},
  Discrete Math. \textbf{341} (2018), no.~3, 781--785. \MR{3754390}

\bibitem[BM12]{BlackburnMcKee12}
Simon~R. Blackburn and James~F. McKee, \emph{Constructing {$k$}-radius
  sequences}, Math. Comp. \textbf{81} (2012), no.~280, 2439--2459. \MR{2945165}

\bibitem[BS96]{BalasubramanianSoundararajan96}
R.~Balasubramanian and K.~Soundararajan, \emph{On a conjecture of {R}. {L}.
  {G}raham}, Acta Arith. \textbf{75} (1996), no.~1, 1--38. \MR{1379389}

\bibitem[Cha88]{Chandler88}
K.~A. Chandler, \emph{Groups formed by redefining multiplication}, Canad. Math.
  Bull. \textbf{31} (1988), no.~4, 419--423. \MR{971568}

\bibitem[Dav00]{Davenport00}
Harold Davenport, \emph{Multiplicative number theory}, third ed., Graduate
  Texts in Mathematics, vol.~74, Springer-Verlag, New York, 2000, Revised and
  with a preface by Hugh L. Montgomery. \MR{1790423}

\bibitem[dBE51]{deBruijnErdos51}
N.~G. de~Bruijn and P.~Erd\"{o}s, \emph{A colour problem for infinite graphs
  and a problem in the theory of relations}, Nederl. Akad. Wetensch. Proc. Ser.
  A. {\bf 54} = Indagationes Math. \textbf{13} (1951), 369--373. \MR{0046630}

\bibitem[FLLP86]{ForcadeLamoreauxPollington86}
Rodney Forcade, Jack Lomoreaux~{[L}amoreaux{]}, and Andrew Pollington,
  \emph{Unsolved {P}roblems: {A} {G}roup of {T}wo {P}roblems in {G}roups},
  Amer. Math. Monthly \textbf{93} (1986), no.~2, 119--121. \MR{1540801}

\bibitem[FP90]{ForcadePollington90}
R.~W. Forcade and A.~D. Pollington, \emph{What is special about {$195$}?
  {G}roups, {$n$}th power maps and a problem of {G}raham}, Number theory
  ({B}anff, {AB}, 1988), de Gruyter, Berlin, 1990, pp.~147--155. \MR{1106658}

\bibitem[Gra70]{Graham70}
Ronald~{L}. Graham, \emph{Advanced problem 5749}, The American Mathematical
  Monthly \textbf{77} (1970), no.~7, 775.

\bibitem[GS81]{GalovichStein81}
Steven Galovich and Sherman Stein, \emph{Splittings of abelian groups by
  integers}, Aequationes Math. \textbf{22} (1981), no.~2-3, 249--267.
  \MR{645422}

\bibitem[HKP10]{HuangKePilz10}
Po-Yi Huang, Wen-Fong Ke, and G{\"u}nter~F. Pilz, \emph{The cardinality of some
  symmetric differences}, Proc. Amer. Math. Soc. \textbf{138} (2010), no.~3,
  787--797. \MR{2566544}

\bibitem[Ink59]{Inkeri59}
K.~Inkeri, \emph{The real roots of {B}ernoulli polynomials}, Ann. Univ. Turku.
  Ser. A I \textbf{37} (1959), 20. \MR{0110835}

\bibitem[Lan94]{Lang94}
Serge Lang, \emph{Algebraic number theory}, second ed., Graduate Texts in
  Mathematics, vol. 110, Springer-Verlag, New York, 1994. \MR{1282723}

\bibitem[LW96]{LagariasWang96}
Jeffrey~C. Lagarias and Yang Wang, \emph{Tiling the line with translates of one
  tile}, Invent. Math. \textbf{124} (1996), no.~1-3, 341--365. \MR{1369421}

\bibitem[Mil63]{Mills63}
W.~H. Mills, \emph{Characters with preassigned values}, Canad. J. Math.
  \textbf{15} (1963), 169--171. \MR{0156828}

\bibitem[Pil92]{Pilz92}
G{\"u}nter Pilz, \emph{On polynomial near-ring codes}, Contributions to general
  algebra, 8 ({L}inz, 1991), H{\"o}lder-Pichler-Tempsky, Vienna, 1992,
  pp.~233--238. \MR{1281844}

\bibitem[PS11]{PachSzabo11}
P{\'e}ter~P{\'a}l Pach and Csaba Szab{\'o}, \emph{On the minimal distance of a
  polynomial code}, Discrete Math. Theor. Comput. Sci. \textbf{13} (2011),
  no.~4, 33--43. \MR{2862558}

\bibitem[SL96]{StevenhagenLenstra96}
P.~Stevenhagen and H.~W. Lenstra, Jr., \emph{Chebotar\"{e}v and his density
  theorem}, Math. Intelligencer \textbf{18} (1996), no.~2, 26--37. \MR{1395088}

\bibitem[TV06]{TaoVu06}
Terence Tao and Van Vu, \emph{Additive combinatorics}, Cambridge Studies in
  Advanced Mathematics, vol. 105, Cambridge University Press, Cambridge, 2006.
  \MR{2289012}

\bibitem[Was97]{Washington97}
Lawrence~C. Washington, \emph{Introduction to cyclotomic fields}, second ed.,
  Graduate Texts in Mathematics, vol.~83, Springer-Verlag, New York, 1997.
  \MR{1421575}

\end{thebibliography}

\end{document}